\documentclass[10pt]{book}
\usepackage[sectionbib]{natbib}
\usepackage{array,epsfig,fancyheadings,rotating}
\usepackage[dvipdfm]{hyperref}
\usepackage{amsmath}
\usepackage{amssymb}
\usepackage{amsfonts}
\usepackage{multirow}
\usepackage{amsthm}

\newtheorem{theorem}{Theorem}
\newtheorem{lemma}{Lemma}
\newtheorem{corollary}{Corollary}
\newtheorem{proposition}{Proposition}
\theoremstyle{definition}

\newtheorem{remark}{Remark}

\newcommand{\bA}{\mbox{\boldmath {$A$}}}
\newcommand{\bb}{\mbox{\boldmath {$b$}}}
\newcommand{\bB}{\mbox{\boldmath {$B$}}}

\newcommand{\bC}{\mbox{\boldmath {$C$}}}

\newcommand{\be}{\mbox{\boldmath {$e$}}}

\newcommand{\bh}{\mbox{\boldmath {$h$}}}
\newcommand{\bH}{\mbox{\boldmath {$H$}}}

\newcommand{\bI}{\mbox{\boldmath {$I$}}}

\newcommand{\bM}{\mbox{\boldmath {$M$}}}

\newcommand{\bsA}{\mbox{\scriptsize{\bf $A$}}}

\newcommand{\bO}{\mbox{\boldmath {$O$}}}

\newcommand{\bP}{\mbox{\boldmath {$P$}}}

\newcommand{\bS}{\mbox{\boldmath {$S$}}}

\newcommand{\bu}{\mbox{\boldmath {$u$}}}
\newcommand{\bU}{\mbox{\boldmath {$U$}}}

\newcommand{\bv}{\mbox{\boldmath {$v$}}}
\newcommand{\bV}{\mbox{\boldmath {$V$}}}
\newcommand{\bw}{\mbox{\boldmath {$w$}}}

\newcommand{\bx}{\mbox{\boldmath {$x$}}}
\newcommand{\bX}{\mbox{\boldmath {$X$}}}
\newcommand{\by}{\mbox{\boldmath {$y$}}}

\newcommand{\bz}{\mbox{\boldmath {$z$}}}

\newcommand{\bze}{\mbox{\boldmath {$0$}}}
\newcommand{\bbbeta}{\mbox{\boldmath $ \eta $}}
\newcommand{\bone}{\mbox{\boldmath {$1$}}}

\newcommand{\bmu}{\mbox{\boldmath $ \mu $}}
\newcommand{\bSig}{\mbox{\boldmath $ \Sigma $}}

\newcommand{\bSigma}{\mbox{\boldmath $ \Sigma $}}

\newcommand{\bLam}{\mbox{\boldmath $ \Lambda $}}

\newcommand{\bal}{\mbox{\boldmath $ \alpha $}}

\newcommand{\bome}{\mbox{\boldmath $ \omega $}}
\newcommand{\bOme}{\mbox{\boldmath $ \Omega $}}
\newcommand{\bGamma}{\mbox{\boldmath $\Gamma$}}

\newcommand{\bTheta}{\mbox{\normalsize{\bf $\Theta$}}}

\newcommand{\bep}{\mbox{\boldmath $ \varepsilon $}}
\newcommand{\bxi}{\mbox{\boldmath $ \xi $}}

\newcommand{\bzeta}{\mbox{\boldmath $ \zeta $}}

\newcommand{\tr}{\mbox{tr}}

\newcommand{\argmax}{\mathop{\rm argmax}\limits}

\newcommand{\Var}{\mbox{Var}}
\begin{document}

\renewcommand{\baselinestretch}{1.2}
\markright{ 
}

\markboth{\hfill{\footnotesize\rm Makoto Aoshima and Kazuyoshi Yata} \hfill}
{\hfill {\footnotesize\rm Two-sample tests for strongly spiked models } \hfill}

\renewcommand{\thefootnote}{}
$\ $\par

\fontsize{10.95}{14pt plus.8pt minus .6pt}\selectfont
\vspace{0.8pc}
\centerline{\large\bf 
Two-sample tests for high-dimension,}
\vspace{2pt}
\centerline{\large\bf strongly spiked eigenvalue models}
\vspace{.4cm}
\centerline{Makoto Aoshima and Kazuyoshi Yata}
\vspace{.4cm}
\centerline{\it University of Tsukuba}
\vspace{.55cm}
\fontsize{9}{11.5pt plus.8pt minus .6pt}\selectfont
\begin{quotation}

\noindent {\it Abstract:} 
We consider two-sample tests for high-dimensional data under two disjoint models: the strongly spiked eigenvalue (SSE) model and the non-SSE (NSSE) model. 
We provide a general test statistic as a function of a positive-semidefinite matrix.  
We give sufficient conditions for the test statistic to satisfy a consistency property and to be asymptotically normal. 
We discuss an optimality of the test statistic under the NSSE model. 
We also investigate the test statistic under the SSE model by considering strongly spiked eigenstructures and create a new effective test procedure for the SSE model. 
Finally, we discuss the performance of the classifiers numerically. 
\par

\vspace{9pt}
\par
\end{quotation}\par

\def\thefigure{\arabic{figure}}
\def\thetable{\arabic{table}}
\fontsize{10.95}{14pt plus.8pt minus .6pt}\selectfont
\setcounter{chapter}{1}
\setcounter{equation}{0} 
\noindent {\bf 1. Introduction}
\par
\vspace{9pt}
A common feature of high-dimensional data is that the data dimension is high, however, the sample size is relatively low. 
This is the so-called ``HDLSS" or ``large $p$, small $n$" data, where $p$ is the data dimension, $n$ is the sample size and $p/n\to\infty$.
Statistical inference on this type of data is becoming increasingly relevant, especially in the areas of medical diagnostics, engineering and other big data.
Suppose we have independent samples of $p$-variate random variables from two populations, $\pi_i,\ i=1,2$, having an unknown mean vector $\bmu_i$ and unknown positive-definite covariance matrix $\bSig_i$ for each $\pi_i$. 
We do not assume the normality of the population distributions.
The eigen-decomposition of $\bSig_{i}$ $(i=1,2)$ is given by $\bSig_{i}=\bH_{i}\bLam_{i}\bH_{i}^T=\sum_{j=1}^p\lambda_{ij}\bh_{ij} \bh_{ij}^T$, 
where $\bLam_{i}=\mbox{diag}(\lambda_{i1},...,\lambda_{ip})$ is a diagonal matrix of eigenvalues, $\lambda_{i1}\ge \cdots \ge \lambda_{ip}> 0$, and $\bH_{i}=[\bh_{i1},...,\bh_{ip}]$ is an orthogonal matrix of the corresponding eigenvectors. 
Note that $\lambda_{i1}$ is the largest eigenvalue of $\bSig_i$ for $i=1,2$.
For the eigenvalues, we consider two disjoint models: the strongly spiked eigenvalue (SSE) model, which will be defined by (\ref{1.7}), and the non-SSE (NSSE) model, which will be defined by (\ref{1.5}).  

In this paper, we consider the two-sample test:
\begin{equation}
H_0:\bmu_1=\bmu_2 \quad \mbox{vs.}\quad H_1:\bmu_1\neq \bmu_2. \label{1.2}
\end{equation}
Having recorded i.i.d. samples, $\bx_{ij},\ j=1,...,n_i$, of size $n_i$ from each $\pi_i$, we define 
$\overline{\bx}_{i n_i}=\sum_{j=1}^{n_i} \bx_{ij}/n_i$ and $\bS_{in_i}=\sum_{j=1}^{n_i}(\bx_{ij}-\overline{\bx}_{i n_i})(\bx_{ij}-\overline{\bx}_{i n_i})^T/(n_i-1)$ for $i=1,2$. 
We assume $n_i \ge 4$ for $i=1,2$. 
Hotelling's $T^2$-statistic is defined by 
$$
T^2=(n_1+n_2)^{-1}n_1n_2(\overline{\bx}_{1n_1}-\overline{\bx}_{2n_2})^T\bS^{-1}(\overline{\bx}_{1n_1}-\overline{\bx}_{2n_2}),
$$
where $\bS=\{(n_1-1)\bS_{1n_1}+(n_2-1)\bS_{2n_2}\}/(n_1+n_2-2)$. 
However, $\bS^{-1}$ does not exist in the HDLSS context such as $p/n_i\to \infty,\ i=1,2$. 
In such situations, \citet{Dempster:1958,Dempster:1960} and \citet{Srivastava:2007} considered the test when $\pi_1$ and $\pi_2$ are Gaussian. 
When $\pi_1$ and $\pi_2$ are non-Gaussian, \citet{Bai:1996} and \citet{Cai:2014} considered the test under homoscedasticity, $\bSig_1=\bSig_2$.
On the other hand, \citet{Chen:2010} and \citet{Aoshima:2011,Aoshima:2015} 
considered the test under heteroscedasticity, $\bSig_1\neq \bSig_2$. 
\lhead[\footnotesize\thepage\fancyplain{}\leftmark]{}\rhead[]{\fancyplain{}\rightmark\footnotesize\thepage}

In this paper, we first consider the following test statistic with a positive-semidefinite matrix $\bA$ of dimension $p$: 
\begin{align}
T(\bA)&=(\overline{\bx}_{1n_1}-\overline{\bx}_{2n_2})^T\bA(\overline{\bx}_{1n_1}-\overline{\bx}_{2n_2})-\sum_{i=1}^2\tr(\bS_{in_i}\bA)/n_i
\notag \\
&=2\sum_{i=1}^2\frac{\sum_{ j<j'  }^{n_i}\bx_{ij}^T\bA\bx_{ij'} }{n_i(n_i-1)}-
2\overline{\bx}_{1n_1}^T\bA \overline{\bx}_{2n_2} .
\label{1.3}
\end{align}
Note that $E\{T(\bA)\}=(\bmu_1-\bmu_2)^T\bA(\bmu_1-\bmu_2)$.
Let $\bI_{p}$ denote the identity matrix of dimension $p$. 
We note that $T(\bI_p)$ is equivalent to the statistics given by 
\citet{Chen:2010} and \citet{Aoshima:2011}. 
We call the test with $T(\bI_p)$ the ``distance-based two-sample test".
In Section 3, we discuss a choice of $\bA$. 
In this paper, we consider the divergence condition such as $p\to \infty$, $n_1\to \infty$ and $n_2\to \infty$, which is equivalent to
$$
m\to \infty,\quad \mbox{where \ $m=\min\{p,n_{\min}\}$ \ with \ $n_{\min}=\min\{n_1,n_2\}$}.
$$
By using Theorem 1 in \citet{Chen:2010} or Theorem 4 in \citet{Aoshima:2015}, we can claim that under $H_0$ in (\ref{1.2}) 
\begin{equation}
T(\bI_p)/\{K_{1}(\bI_p)\}^{1/2} \Rightarrow N(0,1) \ \ \mbox{as $m\to \infty$}
\label{1.4}
\end{equation}
if we assume (A-i) that is given in Section 2 and the condition that
\begin{equation}
\frac{\lambda_{i1}^2}{ \tr(\bSigma_{i}^2)}\to 0 \ \ \mbox{as $p\to \infty$ for $i=1,2$.} 
\label{1.5}
\end{equation}
Here, $K_{1}(\bA)$ is defined in Section 2.1, ``$\Rightarrow$" denotes the convergence in distribution and $N(0,1)$ denotes a random variable distributed as the standard normal distribution. 
Thus, by using $T(\bI_p)$ and an estimate of $K_{1}(\bI_p)$, one can construct a test procedure of (\ref{1.2}) for high-dimensional data. 
As discussed in Section 2 of \citet{Aoshima:2015}, the distance-based two-sample test is quite flexible for high-dimension, non-Gaussian data. 
In Section 3, we shall investigate an optimality of the test statistic in (\ref{1.3}) and discuss a choice of $\bA$.  
\begin{remark}
If all $\lambda_{ij}$s are bounded as $\limsup_{p\to \infty}\lambda_{ij}<\infty$ and $\liminf_{p\to \infty}$
$\lambda_{ij}>0$, (\ref{1.5}) trivially holds.
On the other hand, they often have a spiked model such as 
\begin{equation}
\lambda_{ij}=a_{ij}p^{\alpha_{ij}}\ (j=1,...,t_i)\ \ 
\mbox{and}\ \ \lambda_{ij}=c_{ij}\ (j=t_i+1,...,p), 
\label{1.6}
\end{equation}
where $a_{ij}$s, $c_{ij}$s and $\alpha_{ij}$s are positive fixed constants and $t_i$s are positive fixed integers.
If they have (\ref{1.6}), (\ref{1.5}) holds when $\alpha_{i1}<1/2$ for $i=1,2$. 
See \citet{Yata:2012} for the details. 
\end{remark}
For eigenvalues of high-dimensional data, \citet{Jung:2009}, \citet{Yata:2012,Yata:2013b}, \citet{Onatski:2012} and \citet{Fan:2013} considered spiked models such as $\lambda_{ij}\to \infty$ as $p\to \infty$ for $j=1,...,k_i$, with some positive integer $k_i$. 
The above references all show that spiked models are quite natural because the first several eigenvalues should be spiked for high-dimensional data. 
Hence, we consider the following situation as well: 
\begin{equation}
\liminf_{p\to \infty}\Big\{ \frac{\lambda_{i1}^2}{\tr(\bSigma_{i}^2)}\Big\}>0\ \ 
\label{1.7}
\mbox{for $i=1$ or 2.}
\end{equation}
In (\ref{1.7}), the first eigenvalue is more spiked than in (\ref{1.5}). 
For example, (\ref{1.7}) holds for the spiked model in (\ref{1.6}) with $\alpha_{i1}\ge 1/2$. 
We call (\ref{1.7}) the ``strongly spiked eigenvalue (SSE) model". 
We emphasize that the asymptotic normality in (\ref{1.4}) is not satisfied under the SSE model. 
See Section 4.1. 
See also \citet{Katayama:2013} and \citet{Ma:2015}. 
Recall that (\ref{1.4}) holds under (\ref{1.5}).
We call (\ref{1.5}) the ``non-strongly spiked eigenvalue (NSSE) model". 

The organization of this paper is as follows. 
In Section 2, we give sufficient conditions for $T(\bA)$ to satisfy a consistency property and the asymptotic normality. 
In Section 3, under the NSSE model, we give a test procedure with $T(\bA)$ and discuss the choice of $\bA$.
In Section 4, under the SSE model, we investigate test procedures by considering strongly spiked eigenstructures. 
In Section 5, we create a new test procedure by estimating the eigenstructures for the SSE model. 
We show that the power of the new test procedure is much higher than the distance-based two-sample test for the SSE model. 
In Section 6, we discuss the performance of the test procedures for the SSE model by simulation studies. 
In Section 7, we highlight the benefits of the new models. 
In the supplementary material, 
we give additional simulations, actual data analyses and proofs of the theoretical results.
We also provide a method to distinguish between the NSSE model and the SSE model, and estimate the required parameters.
\\

\setcounter{chapter}{2}
\setcounter{equation}{0} 
\noindent {\bf 2. Asymptotic Properties of $T(\bA)$}
\par
\vspace{9pt}
In this section, we give sufficient conditions for $T(\bA)$ to satisfy a consistency property and to be asymptotically normal. 
As for any positive-semidefinite matrix $\bA$, we write the square root of $\bA$ as $\bA^{1/2}$.
Let $\bx_{ij}=\bH_i \bLam_i^{1/2}\bz_{ij}+\bmu_i$, where $\bz_{ij}=(z_{i1j},...,z_{ipj})^T$ is considered as a sphered data vector having the zero mean vector and identity covariance matrix. 
We assume that the fourth moments of each variable in $\bz_{ij}$ are uniformly bounded. 
More specifically, we assume that
\begin{equation}
\bx_{ij}=\bGamma_i\bw_{ij}+\bmu_i \ \mbox{ for } i=1,2;\ j=1,...,n_i,
\label{2.0}
\end{equation}
where $\bGamma_i$ is a $p\times r_i$ matrix for some $r_i\ge p$ such that $\bGamma_i\bGamma_i^T=\bSigma_i$, 
and $\bw_{ij},\ j=1,...,n_i$, are i.i.d. random vectors having $E(\bw_{ij})=\bze$ and Var$(\bw_{ij})=\bI_{r_i}$. 
Note that (\ref{2.0}) includes the case that $\bGamma_i=\bH_i \bLam_i^{1/2}$ and $\bw_{ij}=\bz_{ij}$. 
Refer to \citet{Bai:1996}, \citet{Chen:2010} and \citet{Aoshima:2015} for the details of the model. 
As for $\bw_{ij}=(w_{i1j},...,w_{ir_ij})^T$, we assume the following assumption for $\pi_i,\ i=1,2$, as necessary:
\begin{description}
\item[(A-i)] \ 
The fourth moments of each variable in $\bw_{ij}$ are uniformly bounded, 
$E( w_{isj}^2w_{itj}^2)=E( w_{isj}^2)E(w_{itj}^2)$ and $E( w_{isj} w_{itj}w_{iuj}w_{ivj})=0$ for all $s \neq t,u,v$. 
\end{description} 
When the $\pi_i$s are Gaussian, (A-i) naturally holds.
\\

\noindent {\bf 2.1. Consistency and Asymptotic Normality of $T(\bA)$}
\par
\vspace{9pt}
Let $\bmu_{\bsA}=\bA^{1/2}(\bmu_1-\bmu_2)$, $\bSigma_{i,\bsA}=\bA^{1/2}\bSig_i\bA^{1/2}, \ i=1,2$, 
and $\Delta({\bA})=||\bmu_{\bsA}||^2$,
where $||\cdot||$ denotes the Euclidean norm.
Let $K({\bA})=K_{1}(\bA)+K_{2}(\bA)$, where
$$
K_{1}(\bA)=2 \sum_{i=1}^2 \frac{\tr(\bSig_{i,\bsA}^2)}{n_i(n_i-1)}
+4\frac{\tr(\bSig_{1,\bsA}\bSig_{2,\bsA})}{n_1n_2}
\ \ 
\mbox{and} \ \ K_{2}(\bA)=4\sum_{i=1}^2 \frac{\bmu_{\bsA}^T\bSig_{i,\bsA}\bmu_{\bsA}}{n_i}.
$$
Note that $E\{T(\bA)\}=\Delta(\bA)$ and $\Var\{T(\bA)\}=K(\bA)$. 
Also, note that $\Delta(\bA)=0$ under $H_0$ in (\ref{1.2}). 
Let $\lambda_{\max}(\bB)$ denote the largest eigenvalue of any positive-semidefinite matrix, $\bB$. 
We assume the following condition for $\bSigma_{i,\bsA}$s as necessary: 
\begin{description}
  \item[(A-ii)]\ $\displaystyle \frac{\{\lambda_{\max}(\bSigma_{i,\bsA})\}^2}{\tr(\bSigma_{i,\bsA}^2)}\to 0$ \ as $p\to \infty$ for $i=1,2$. 
\end{description}
When $\bA=\bI_p$, (A-ii) becomes (\ref{1.5}). 
We assume one of the following three conditions as necessary:
\begin{description}
  \item[(A-iii)]\  $\displaystyle \frac{K_{1}(\bA) }{\{\Delta(\bA)\}^2}\to 0$ \ as $m\to \infty$;\quad \ {\bf (A-iv)}  \ $\displaystyle \limsup_{m\to \infty} \frac{\{\Delta(\bA)\}^2}{K_{1}(\bA)}<\infty$;
  \item[(A-v)]\  $\displaystyle \frac{K_{1}(\bA)}{K_{2}(\bA)}\to 0$ \ as $m\to \infty$.
\end{description}
Note that (A-iv) holds under $H_0$ in (\ref{1.2}). 
If $\bSig_1=\bSig_2\ (=\bSig$, say), (A-iii) holds when $\tr\{(\bSig \bA)^2\}/\{n_{\min} \Delta(\bA)\}^2\to 0$ as $m\to \infty$.
On the other hand, (A-iv) holds when $\liminf_{m\to \infty}$
$\tr\{(\bSig \bA)^2\}/\{n_{\min} \Delta(\bA)\}^2> 0$. 
See Section 3.2 for the details of (A-v). 
For (A-iii) and (A-v), we have the following proposition. 
\begin{proposition}
(A-v) implies (A-iii).
\end{proposition}
When (A-iii) is met, we have the following result. 
\begin{theorem} 
Assume (A-iii). 
It holds that $T(\bA)/\Delta(\bA)=1+o_P(1)$ as $m\to \infty$.
\end{theorem}
When (A-iv) or (A-v) is met, we have the following results. 
\begin{theorem} 
Assume (A-i). 
Assume either (A-ii) and (A-iv) or (A-v). 
Then, it holds that $\{ T(\bA)-\Delta(\bA)\}/\{K(\bA)\}^{1/2}\Rightarrow N(0,1)$ as $m\to \infty$. 
\end{theorem}
\begin{lemma} 
Assume (A-ii) and (A-iv).
It holds that $K(\bA)/K_{1}(\bA)=1+o(1)$ as $m\to \infty$.
\end{lemma}
Since $\bSig_i$s are unknown, it is necessary to estimate $K_{1}(\bA)$. 
Let us consider an estimator of $K_{1}(\bA)$ by 
$$
\widehat{K}_{1}(\bA)=2 \sum_{i=1}^2 \frac{W_{in_i}(\bA)}{n_i(n_i-1)}
+4\frac{ \tr(\bS_{1n_1}\bA\bS_{2n_2}\bA)}{n_1n_2},
$$
where $W_{in_i}(\bA)$ is defined by (\ref{2.2}) in Section 2.2.
\begin{lemma} 
Assume (A-i).
It holds that $\widehat{K}_{1}(\bA)/K_{1}(\bA)=1+o_P(1)$ as $m\to \infty$.
\end{lemma}
By combining Theorem 2 with Lemmas 1 and 2, we have the following result. 
\begin{corollary} 
Assume (A-i), (A-ii) and (A-iv). Then, 
it holds that $\{T(\bA)-\Delta(\bA)\}/\{\widehat{K}_{1}(\bA)\}^{1/2}\Rightarrow N(0,1)$ as $m\to \infty$.
\end{corollary}
\vspace{9pt}
\noindent {\bf 2.2. Estimation of $\tr(\bSig_{\bsA}^2)$}
\par
\vspace{9pt}
Throughout this section, we omit the subscript with regard to the population. 
\citet{Chen:2010b} considered an unbiased estimator of $\tr(\bSig^2)$ by 
$W_{n}=\sum_{i\neq j}^n(\bx_i^T\bx_j)^2/_{n}P_{2}-2\sum_{i\neq j \neq s}^n\bx_i^T\bx_j\bx_j^T\bx_s/_{n}P_{3}+\sum_{i\neq j \neq s \neq t}^n\bx_i^T\bx_j\bx_s^T\bx_t/_{n}P_{4}$, 
where $_{n}P_{r}=n!/(n-r)!$.
\citet{Aoshima:2011} and \citet{Yata:2013a} gave a different unbiased estimator of $\tr(\bSig^2)$. 
From these backgrounds, we construct an unbiased estimator of $\tr(\bSigma_{\bsA}^2)$ as follows:
\begin{equation}
W_{n}(\bA)=
\sum_{i\neq j}^n\frac{ (\bx_i^T\bA \bx_j)^2}{_{n}P_{2}}-2\sum_{i\neq j \neq s}^n\frac{ \bx_i^T\bA \bx_j\bx_j^T\bA \bx_s}{_{n}P_{3}}+\sum_{i\neq j \neq s \neq t}^n\frac{\bx_i^T\bA \bx_j\bx_s^T\bA\bx_t}{_{n}P_{4}}.
\label{2.2}
\end{equation}
Note that $E\{W_{n}(\bA)\}=\tr(\bSigma_{\bsA}^2)$ and $W_n(\bI_p)=W_n$. 
In view of \citet{Chen:2010b}, one can claim that 
\begin{equation}
\Var\{W_{n}(\bA)/\tr(\bSigma_{\bsA}^2)\}\to  0
\label{2.3}\end{equation}
as $p\to \infty$ and $n\to \infty$ under (A-i), so that $W_{n}(\bA)=\tr(\bSigma_{\bsA}^2)\{1+o_P(1)\}$.
\\

\setcounter{chapter}{3}
\setcounter{equation}{0} 
\noindent {\bf 3. Test Procedures for Non-Strongly Spiked Eigenvalue Model}
\par
\vspace{9pt}
In this section, we consider test procedures given by $T(\bA)$ when (A-ii) is met as in the NSSE model. 
With the help of the asymptotic normality, we discuss an optimality of $T(\bA)$ for high-dimensional data.
\\

\noindent {\bf 3.1. Test Procedure by $T(\bA)$}
\par
\vspace{9pt}
Let $z_{c}$ be a constant such that $P\{N(0,1)>z_{c}\}=c$ for $c \in (0,1)$. 
For given $\alpha \in(0,1/2)$, from Corollary 1, we consider testing the hypothesis in (\ref{1.2}) by 
\begin{equation}
\mbox{rejecting $H_0$}\Longleftrightarrow T(\bA)/\{\widehat{K}_1(\bA)\}^{1/2}>z_{\alpha}.
\label{3.1}
\end{equation}
Note that the power of the test (\ref{3.1}) depends on $\Delta(\bA)$. 
We denote it by power($\Delta(\bA)$). 
Then, we have the following result. 
\begin{theorem}
Assume (A-i) and (A-ii). 
Then, the test (\ref{3.1}) has as $m\to \infty$
$$
\mbox{size}=\alpha+o(1) \ \ \mbox{and} \ \ \mbox{power$(\Delta(\bA))$}-\Phi \bigg(\frac{\Delta(\bA)}{\{K(\bA)\}^{1/2}}
-z_{\alpha}\Big( \frac{ K_{1}(\bA) }{  K(\bA) }\Big)^{1/2}  \bigg)=o(1),
$$
where $\Phi(\cdot)$ denotes the cumulative distribution function (c.d.f.) of $N(0,1)$. 
\end{theorem}
When (A-iii), (A-iv) or (A-v) is met under $H_1$, we have the following result from Theorem 3.
\begin{corollary}
Assume (A-i). 
Then, under $H_1$, the test (\ref{3.1}) has as $m\to \infty$
\begin{align*}
&\mbox{power$(\Delta(\bA))$}=1+o(1)\quad \mbox{under (A-iii)}; \\
&\mbox{power$(\Delta(\bA))$}-\Phi \Big(\frac{\Delta(\bA)}{\{K_{1}(\bA)\}^{1/2}}-z_{\alpha} \Big)=o(1)\quad \mbox{under (A-ii) and (A-iv)};\\
\mbox{and}\quad 
&\mbox{power$(\Delta(\bA))$}-\Phi \Big(\frac{\Delta(\bA)}{\{K_{2}(\bA)\}^{1/2}} \Big)=o(1)\quad \mbox{under (A-v)}.
\end{align*}
\end{corollary}
\vspace{9pt}
\noindent {\bf 3.2. Choice of $\bA$ in (\ref{3.1})}
\par
\vspace{9pt}
First, we consider the case when (A-v) is met under $H_1$. 
From Corollary 2, we simply write that 
$$
\mbox{power$(\Delta(\bA))$}\approx \Phi \big({\Delta(\bA)}/{\{K_{2}(\bA)\}^{1/2}} \big).
$$
Let $\bA_{\star}=c_{\star}(\bSig_1/n_1+\bSig_2/n_2)^{-1}$ with $c_{\star}=1/n_1+1/n_2$. 
Note that $\bA_{\star}=\bSig^{-1}$ when $\bSig_1=\bSig_2\ (=\bSig)$. 
Also, note that $\Delta(\bA_{\star})=(\bmu_1-\bmu_2)^T\bSig^{-1}(\bmu_1-\bmu_2)\ (=\Delta_{MD},\ \mbox{say})$ when $\bSig_1=\bSig_2$, 
where $\Delta_{MD}^{1/2}$ is the Mahalanobis distance. 
Then, from Proposition S1.1 of the supplementary material, 
$\bA_{\star}$ maximizes $\Delta(\bA)/\{K_{2}(\bA)\}^{1/2}$ over the set of positive-definite matrices of dimension $p$.
Here, let us consider (A-v). 
Note that 
$
c_{\star}^2p
=c_{\star}^2\tr\{(\bA_{\star} \bA_{\star}^{-1})^2\}=
\sum_{i=1}^2
\tr\{(\bSig_i\bA_{\star})^2\}/n_i^2
+2
\tr(\bSig_1\bA_{\star}\bSig_2\bA_{\star} )/(n_1 n_2),
$
so that $K_{1}(\bA_{\star})=2c_{\star}^2p\{1+o(1)\}$ as $m\to \infty$. 
Also, note that $K_{2}(\bA_{\star})=4c_{\star} \Delta(\bA_{\star})$. 
Thus, if (A-v) is met, it holds that as $m\to \infty$ 
$$
{K_{1}(\bA_{\star})}/{K_{2}(\bA_{\star})}
=O\big({pc_{\star}}/{\Delta(\bA_{\star})} \big)
=O\big( {p}/\{n_{\min}\Delta(\bA_{\star})\}\big)\to 0
\quad \mbox{as $m\to \infty$}.
$$
However, such a situation is severe for high-dimensional data.
For example, when $\bSig_1=\bSig_2$ and the Mahalanobis distance is bounded as $\limsup_{p\to \infty}\Delta_{MD}<\infty$, the sample size should be large enough to claim $n_{\min}/p\to \infty$ because $\Delta(\bA_{\star})=\Delta_{MD}$. 
Hence, we have to say that (A-v) is quite strict for high-dimensional data. 
To begin with, from Proposition 1 and Corollary 2, for any choice of $\bA$ in (\ref{3.1}), it holds that $\mbox{power$(\Delta(\bA))$}=1+o(1)$ under (A-v). 
Hence, the optimal choice of $\bA$ does not make much improvement in the power when (A-v) is met.  
On the other hand, when (A-v) is not met (i.e., (A-iv) is met), the test (\ref{3.1}) has 
$$
\mbox{power$(\Delta(\bA))$}\approx \Phi \big( {\Delta(\bA)}/\{K_{1}(\bA)\}^{1/2}-z_{\alpha} \big)
$$
from Corollary 2.
In this case, $\bA_{\star}$ is not the optimal choice any longer. 
Because of the above reasons, we do not recommend to use a test procedure based on the Mahalanobis distance such as (\ref{3.1}) with $\bA=\bA_{\star}$.
In addition, it is difficult to estimate $\bA_{\star}$ for high-dimensional data unless the $\bSig_i$s are sparse. 
When the $\bSig_i$s are sparse, see \citet{Bickel:2008}.
 
\citet{Srivastava:2013} considered a two-sample test by using $\bA_{\star(d)}=c_{\star}(\bSig_{1(d)}/n_1+\bSig_{2(d)}/n_2)^{-1}$ for $\bA$, where $\bSig_{i(d)}=\mbox{diag}(\sigma_{i(1)},...,\sigma_{i(p)})$ with $\sigma_{i(j)}\ (>0)$ the $j$-th diagonal element of $\bSig_i$ for $i=1,2;\ j=1,...,p$. 
However, we do not recommend to choose $\bA_{\star(d)}$ unless (A-v) is met and the $\bSig_i$s are diagonal matrices. 
If (A-ii) is met as in the NSSE model, we rather recommend to choose $\bA=\bI_p$ in (\ref{3.1}) that yields the distance-based two-sample test. 
When $\bA=\bI_p$, it is not necessary to estimate $\bA$ and it is quite flexible for high-dimension, non-Gaussian data.
See Section 2 of \citet{Aoshima:2015} for the details.  
\\

\noindent {\bf 3.3. Simulations}
\par
\vspace{9pt}
We used computer simulations to study the performance of the test procedure given by (\ref{3.1}) when $\bA=\bI_p$, $\bA=\bA_{\star}$, $\bA=\bA_{\star(d)}$ and $\bA=\widehat{\bA}_{\star(d)}$.
Here, $\widehat{\bA}_{\star(d)}=c_{\star}(\bS_{1n_1(d)}/n_1+\bS_{2n_2(d)}/n_2)^{-1}$, where $\bS_{in_i(d)}=\mbox{diag}(s_{in_i1},...,s_{in_ip})$, $i=1,2,$ with $s_{in_ij}$ the $j$-th diagonal element of $\bS_{in_i}$. 
\citet{Srivastava:2013} considered a test procedure given by $T(\widehat{\bA}_{\star(d)})$.
We set $\alpha=0.05$. 
Independent pseudo-random observations were generated from $\pi_i:N_p(\bmu_i,\bSig_i)$,\ $i=1,2$. 
We set $p=2^{s},\ s=4,...,10$ and $n_1=n_2=\lceil p^{1/2} \rceil$, where $\lceil x \rceil $ denotes the smallest integer $\ge x$.
We set $\bmu_1=\bze$ and $\bSig_1=\bSig_2=\bC( 0.3^{|i-j|^{1/2}})\bC$, where 
$
\bC=\mbox{diag}[\{0.5+1/(p+1)\}^{1/2},...,\{0.5+p/(p+1)\}^{1/2} ].
$
We considered three cases: (a) $\bmu_2=\bze$, (b) $\bmu_{2}=(1,...,1,0,...,0)^T$ whose first $10$ elements are $1$,  
and (c) $\bmu_{2}=(0,...,0,1,...,1)^T$ whose last $10$ elements are $1$. 
When $\bA=\bI_p$, $\bA=\bA_{\star}$ and $\bA=\bA_{\star(d)}$, 
we note that (A-ii) and (A-iv) are met for (a), (b) and (c). 

We checked the performance of the test procedures given by (\ref{3.1}) with (I) $\bA=\bI_p$, (II) $\bA=\bA_{\star}$, (III) $\bA=\bA_{\star(d)}$ and (IV) $\bA=\widehat{\bA}_{\star(d)}$.
The findings were obtained by averaging the outcomes from 2000 ($=R$, say) replications in each situation. 
We defined $P_{r}=1\ (\mbox{or}\ 0)$ when $H_0$ was falsely rejected (or not) for $r=1,...,2000$ for (a) and defined $\overline{\alpha}=\sum_{r=1}^{R}P_{r}/R$ to estimate the size. 
We also defined $P_{r}=1\ (\mbox{or}\ 0)$ when $H_1$ was falsely rejected (or not) for $r=1,...,2000$ for (b) and (c) and defined $1-\overline{\beta}=1-\sum_{r=1}^{R}P_{r}/R$ to estimate the power. 
Note that their standard deviations are less than $0.011$. 
In Fig. 1, we plotted $\overline{\alpha}$ for (a) and $1-\overline{\beta}$ for (b) and (c). 
We also plotted the asymptotic power, $\Phi(\Delta(\bA)/\{K(\bA)\}^{1/2}-z_{\alpha}\{K_1(\bA)/K(\bA)\}^{1/2})$, for (I) to (III) by using Theorem 3. 
\begin{figure}[h!]
\includegraphics[scale=0.46]{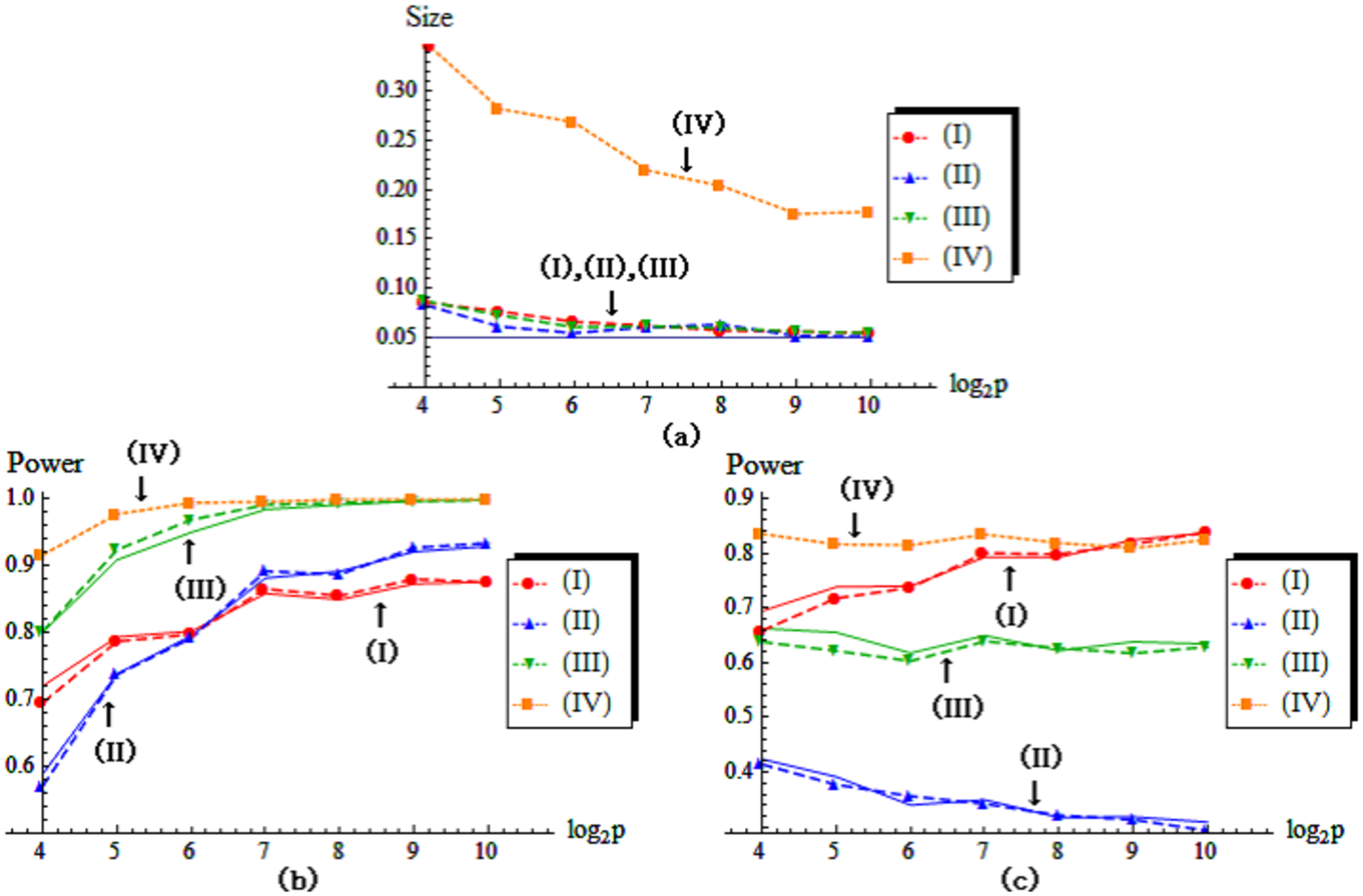} 
\caption{Tests by (\ref{3.1}) when (I) $\bA=\bI_p$, (II) $\bA=\bA_{\star}$, (III) $\bA=\bA_{\star(d)}$ and (IV) $\bA=\widehat{\bA}_{\star(d)}$. 
The values of $\overline{\alpha}$ are denoted by the dashed lines in the top panel.
The values of $1-\overline{\beta}$ are denoted by the dashed lines in the left panel for (b) and in the right panel for (c). 
The asymptotic powers were given by $\Phi(\Delta(\bA)/\{K(\bA)\}^{1/2}-z_{\alpha}\{K_1(\bA)/K(\bA)\}^{1/2})$ for (I) to (III) which are denoted by the solid lines both in the panels. 
}
\end{figure}
As expected theoretically, we observed that the plots become close to the theoretical values. 
The test with (II) gave a better performance compared to (I) for (b); however, it gave quite a bad performance for (c). 
We note that the test procedure based on the Mahalanobis distance does not always give a preferable performance for high-dimensional data even when the population distributions are Gaussian having a known and common covariance matrix.
See Section 3.2 for the details. 
On the other hand, we observed that the test with (III) gives a good performance compared to (I) for (b); however, they trade places for (c). 
This is because $\Delta(\bI_p)<\Delta(\bA_{\star(d)})$ for (b) and $\Delta(\bI_p)>\Delta(\bA_{\star(d)})$ for (c) when $p$ is sufficiently large. 
The test with (IV) gave quite a bad performance because the size for (IV) was much higher than $\alpha$ even when $p$ and $n_{i}$s are large. 
Hence, we do not recommend to use the test procedures based on the Mahalanobis distance or the diagonal matrices unless 
$n_{i}$s are quite large enough to claim (A-v). 

We also checked the performance of the test procedures by (\ref{3.1}) for the multivariate skew normal (MSN) distribution. 
See \cite{Azzalini:1996} for the details of the MSN distribution.
We observed the performance similar to that in Fig 1. 
We gave the results in Section S4.1 of the supplementary material. 
\\

\setcounter{chapter}{4}
\setcounter{equation}{0} 
\noindent {\bf 4. Test Procedures for Strongly Spiked Eigenvalue Model}
\par
\vspace{9pt}
In this section, we consider test procedures when (A-ii) is not met as in the SSE model. 
We emphasize that high-dimensional data often have the SSE model. 
See Fig. 1 in \citet{Yata:2013b} or Section S3 of the supplementary material as well. 
In case of (A-iv), $T(\bA)$ does not satisfy the asymptotic normality in Theorem 2, so that one cannot use the test (\ref{3.1}). 
For example, as for $T(\bI_p)$, we cannot claim either (\ref{1.4}) or ``size$=\alpha+o(1)$" under the SSE model. 
In such situations, we consider alternative test procedures.
\\

\noindent {\bf 4.1. Distance-Based Two-Sample Test}
\par
\vspace{9pt}
We simply write $T_I=T(\bI_p)$, $K_{1(I)}=K_1(\bI_p)$ and $\widehat{K}_{1(I)}=\widehat{K}_1(\bI_p)$ when $\bA=\bI_p$. 
For the SSE model, \citet{Katayama:2013} considered a one sample test. 
\citet{Ma:2015} considered a two sample test for a factor model which is a special case of the SSE model.
\citet{Katayama:2013} showed that a test statistic is asymptotically distributed as a $\chi^2$ distribution under the Gaussian assumption. 
For the two sample test in (\ref{1.2}), we have the following result. 
\begin{theorem}
Assume
\begin{equation}
|\bh_{11}^T\bh_{21}|=1+o(1) \quad \mbox{and} \quad \Psi_{i(2)}/\lambda_{i1}^{2}\to 0, \  i=1,2, \ \mbox{as $p\to \infty$},
\label{4.1}
\end{equation}
where
$$
\Psi_{i(s)}
= \sum_{j=s}^p\lambda_{ij}^2\quad \mbox{for $i=1,2$; $s=1,...,p$.}
$$
Then, it holds that $(2/K_{1(I)})^{1/2}T_I+1 \Rightarrow  \chi_1^2$ as $m\to \infty$ under $H_0$,
where $\chi_\nu^2$ denotes a random variable having a $\chi^2$ distribution with $\nu$ degrees of freedom. 
\end{theorem}
We test (\ref{1.2}) by
\begin{equation}
\mbox{rejecting $H_0$}\Longleftrightarrow (2/\widehat{K}_{1(I)})^{1/2}T_I+1>\chi_1^2(\alpha), 
\label{4.2}
\end{equation}
where $\chi_1^2(\alpha)$ denotes the $(1-\alpha)$th quantile of $\chi_1^2$. 
Note that $\widehat{K}_{1(I)}/K_{1(I)}=1+o_P(1)$ as $m\to \infty$ under (A-i).
Then, from Theorem 4, 
the test (\ref{4.2}) ensures that size$=\alpha+o(1)$ as $m\to \infty$ under (A-i). 

We note that ``$|\bh_{11}^T\bh_{21}|=1+o(1)$ as $p\to \infty$" in (\ref{4.1}) is not a general condition for high-dimensional data, so that it is necessary to check the condition in actual data analyses. 
See Lemma 4.1 in \citet{Ishii:2016} for checking the condition. 
When (\ref{4.1}) is not met, the test (\ref{4.2}) cannot ensure the accuracy. 

\vspace{9pt}
\noindent {\bf 4.2. Test Statistics Using Eigenstructures}
\par
\vspace{9pt}
We consider the following model: 
\begin{description}
\item[(A-vi)]\quad For $i=1,2$, there exists a positive fixed integer $k_i$ such that $\lambda_{i1},...,\lambda_{ik_i}$ are distinct in the sense that $\liminf_{p\to \infty}(\lambda_{ij}/\lambda_{ij'}-1)>0$ when $1\le j<j'\le k_i$, and $\lambda_{ik_i}$ and $\lambda_{ik_i+1}$ satisfy 
$$
\liminf_{p\to \infty}\frac{\lambda_{i k_i }^2}{\Psi_{i(k_i)} }>0
\ \ \mbox{and} \ \ \frac{\lambda_{ik_i+1}^2}{ \Psi_{i(k_i+1)} }\to 0 \ \ \mbox{as $p\to \infty$}.
$$ 
\end{description}
Note that (A-vi) implies (\ref{1.7}), that is (A-vi) is one of the SSE models.
(A-vi) is also a power spiked model given by \citet{Yata:2013b}. 
For the spiked model in (\ref{1.6}), (A-vi) holds under the conditions that $\alpha_{ik_i}\ge 1/2$, $a_{ij} \neq a_{ij'}$ for $1\le j <j'\le k_i\ (< t_i)$ and $\alpha_{ik_i+1}< 1/2$ for $i=1,2$. 
We consider the following test statistic with positive-semidefinite matrices, $\bA_i,\ i=1,2,$ of dimension $p$: 
$$
T(\bA_{1},\bA_{2})=2\sum_{i=1}^2\frac{\sum_{ j<j'  }^{n_i}\bx_{ij}^T\bA_i\bx_{ij'} }{n_i(n_i-1)}-
2\overline{\bx}_{1n_1}^T\bA_1^{1/2}\bA_2^{1/2} \overline{\bx}_{2n_2}.
$$
We do not recommend to choose $\bA_i=\bSig_i^{-1},\ i=1,2$. 
See Section S1.2 in the supplementary material for the details. 
In addition, it is difficult to estimate $\bSig_i^{-1}$s for high-dimension, non-sparse data. 
Here, we consider $\bA_i$s as 
$$
\bA_{i(k_i)}=\bI_p-\sum_{j=1}^{k_i}\bh_{ij}\bh_{ij}^T=\sum_{j=k_i+1}^{p}\bh_{ij}\bh_{ij}^T \quad \mbox{for $i=1,2$}. 
$$
Note that $\bA_{i(k_i)}=\bA_{i(k_i)}^{1/2}$. 
Let us write that $\bmu_{*}=\bA_{1(k_1)}\bmu_1-\bA_{2(k_2)}\bmu_2$ and $\bSig_{i*}={\bA}_{i(k_i)}\bSig_i{\bA}_{i(k_i)}=\sum_{j=k_i+1}^p\lambda_{ij}\bh_{ij}\bh_{ij}^T$ for $i=1,2$. 
Let $T_*=T(\bA_{1(k_1)},\bA_{2(k_2)})$, $\Delta_*=||\bmu_{*}||^2$ and 
$K_*=K_{1*}+K_{2*}$, 
where 
$$
K_{1*}=2 \sum_{i=1}^2 \frac{\tr(\bSig_{i*}^2)}{n_i(n_i-1)}+4 \frac{\tr(\bSig_{1*}\bSig_{2*})}{n_1n_2}
\quad \mbox{and} \quad 
 K_{2*}=4\sum_{i=1}^2 \frac{\bmu_{*}^T\bSig_{i*}\bmu_{*}}{n_i}.
$$ 
Note that $E(T_*)=\Delta_*$ and $\Var(T_*)=K_*$. 
Also, we note that $\tr(\bSig_{i*}^2)=\Psi_{i(k_i+1)}$ and $\lambda_{\max}(\bSig_{i*})=\lambda_{k_i+1}$ for $i=1,2,$ so that
$$
\lambda_{\max}^2(\bSig_{i*})/\tr(\bSig_{i*}^2)\to 0 \ \ \mbox{as $p\to \infty$ for $i=1,2$, under (A-vi)}. 
$$
From Theorem 2, we have the following result.
\begin{corollary} 
Assume (A-i) and $\limsup_{m\to \infty}\Delta_*^2/K_{1*}<\infty$. 
Then, under (A-vi), it holds that 
$(T_*-\Delta_* )/K_*^{1/2}\Rightarrow N(0,1)$ as $m\to \infty$. 
\end{corollary}
It does not always hold that $\Delta_*=0$ under $H_0$ when $\bA_{1(k_1)} \neq \bA_{2(k_2)}$. 
We assume the following condition:
\begin{description}
  \item[(A-vii)]\quad $\displaystyle  \frac{\Delta_*^2}{K_{1*}}\to 0$ \ as $m\to \infty$ under $H_0$. 
\end{description}
Note that (A-vii) is a mild condition because $\bA_{1(k_1)}-\bA_{2(k_2)}=\sum_{j=1}^{k_2}\bh_{2j}\bh_{2j}^T-\sum_{j=1}^{k_1}\bh_{1j}\bh_{1j}^T$ is a low-rank matrix with rank $ k_1+k_2$ at most and under $H_0$, $\Delta_*=||(\bA_{1(k_1)}-\bA_{2(k_2)})\bmu_1||^2$ is small. 
From Corollary 3, under $H_0$, 
it follows that  
$
P(T_*/K_{1*}^{1/2} >z_{\alpha})=\alpha+o(1)$.
Similar to (\ref{3.1}), one can construct a test procedure by using $T_*$. 
Let 
$$
x_{ijl}={\bh}_{ij}^T\bx_{il}=\lambda_{ij}^{1/2}z_{ijl}+\mu_{i(j)}\quad \mbox{for all $i,j,l$, where $\mu_{i(j)}=\bh_{ij}^T\bmu_i$}.
$$
Then, we write that 
\begin{align*}
T_*=&2\sum_{i=1}^2\frac{\sum_{l<l'}^{n_i}(\bx_{il}^T\bx_{il'}-\sum_{j=1}^{k_i}
{x}_{ijl}{x}_{ijl'})}{n_i(n_i-1)}\\
&-2\frac{\sum_{l=1}^{n_1}\sum_{l'=1}^{n_2}
(\bx_{1l}-\sum_{j=1}^{k_1}{x}_{1jl}{\bh}_{1j})^T
(\bx_{2l'}-\sum_{j=1}^{k_2}{x}_{2jl'}{\bh}_{2j})}{n_1n_2}.
\end{align*}
In order to use $T_*$, it is necessary to estimate ${x}_{ijl}$s and $\bh_{ij}$s. 
\\

\setcounter{chapter}{5}
\setcounter{equation}{0} 
\noindent {\bf 5. Test Procedure Using Eigenstructures for Strongly Spiked Eigenvalue Model}
\par
\vspace{9pt}
In this section, we assume (A-vi) and the following assumption for $\pi_i$s:
\begin{description}
\item[(A-viii)] \ 
$E( z_{isj}^2z_{itj}^2)=E( z_{isj}^2)E(z_{itj}^2)$, \ $E( z_{isj}z_{itj}z_{iuj})=0$ \ and \\
$E( z_{isj} z_{itj}z_{iuj}z_{ivj})=0$ for all $s \neq t,u,v$, with $z_{ijl}$s defined in Section 2.
\end{description} 
Note that (A-viii) implies (A-i) because $E(z_{ijl}^4)$'s are bounded and (\ref{2.0}) includes the case that $\bGamma_i=\bH_i \bLam_i^{1/2}$ and $\bw_{ij}=\bz_{ij}$. 
When the $\pi_i$s are Gaussian, (A-viii) naturally holds. 
First, we discuss estimation of the eigenvalues and eigenvectors in the SSE model. 
\\

\noindent {\bf 5.1. Estimation of Eigenvalues and Eigenvectors}
\par
\vspace{9pt}
Throughout this section, we omit the subscript with regard to the population for the sake of simplicity.
Let $\hat{\lambda}_{1}\ge\cdots\ge\hat{\lambda}_{p}\ge 0$ be the eigenvalues of $\bS_{n}$. 
Let us write the eigen-decomposition of $\bS_{n}$ as $\bS_{n}=\sum_{j=1}^{p}\hat{\lambda}_{j}\hat{\bh}_{j}\hat{\bh}_{j}^T$, 
where $\hat{\bh}_{j}$ denotes a unit eigenvector corresponding to $\hat{\lambda}_{j}$. 
We assume $\bh_{j}^T\hat{\bh}_{j} \ge 0$ w.p.1 for all $j$ without loss of generality. 
Let $\bX=[\bx_{1},...,\bx_{n}]$ and $\overline{\bX}=[\overline{\bx}_{n},...,\overline{\bx}_{n}]$. 
Then, we define the $n \times n$ dual sample covariance matrix by 
$
\bS_{D}=(n-1)^{-1}(\bX-\overline{\bX})^T(\bX-\overline{\bX}).
$
Note that $\bS_n$ and $\bS_{D}$ share non-zero eigenvalues.
Let us write the eigen-decomposition of $\bS_{D}$ as $\bS_{D}=\sum_{j=1}^{n-1}\hat{\lambda}_{j}\hat{\bu}_{j}\hat{\bu}_{j}^T $, where $\hat{\bu}_{j}=(\hat{u}_{j1},...,\hat{u}_{jn})^T$ denotes a unit eigenvector corresponding to $\hat{\lambda}_{j}$. 
Note that $\hat{\bh}_{j}$ can be calculated by 
$\hat{\bh}_{j}=\{(n-1)\hat{\lambda}_{j}\}^{-1/2}(\bX-\overline{\bX}) \hat{\bu}_{j}$.
Let $\delta_j=\lambda_j^{-1}\sum_{s=k+1}^{p}\lambda_{s}/(n-1)$ for $j=1,...,k$. 
Let $m_0=\min\{p,n\}$. 
First, we have the following result. 
\begin{proposition}
Assume (A-vi) and (A-viii). 
It holds for $j=1,...,k$, that 
$\hat{\lambda}_{j}/\lambda_{j}=1+\delta_j+O_P(n^{-1/2})$ and 
$(\hat{\bh}_{j}^T\bh_{j})^2=(1+\delta_j)^{-1}+O_P(n^{-1/2})$ as $m_0 \to \infty$.
\end{proposition}
If $\delta_j\to \infty$ as $m_0 \to \infty$, $\hat{\lambda}_{j}$ and $\hat{\bh}_{j}$ are strongly inconsistent in the sense that $\lambda_{j}/\hat{\lambda}_{j}=o_P(1)$ and $(\hat{\bh}_{j}^T\bh_{j})^2=o_P(1)$.  
See \citet{Jung:2009} for the concept of the strong inconsistency. 
Also, from Proposition 2, under (A-vi) and (A-viii), it holds that as $m_0\to \infty$ 
\begin{equation}
||\hat{\bh}_{j}-\bh_{j}||^2=2\{1-(1+\delta_j)^{-1/2}\}+O_P(n^{-1/2})\quad \mbox{for $j=1,...,k$}.
\label{5.1}
\end{equation}
In order to overcome the curse of dimensionality, \citet{Yata:2012} proposed an eigenvalue estimation called the noise-reduction (NR) methodology, which was brought about by a geometric representation of $\bS_{D}$.
If one applies the NR methodology, the $\lambda_{j}$s are estimated by
\begin{equation}
\tilde{\lambda}_{j}=\hat{\lambda}_{j}-\frac{\tr(\bS_{D})-\sum_{l=1}^j\hat{\lambda}_{l} }{n-1-j}\quad (j=1,...,n-2).
\label{5.2}
\end{equation}
Note that $\tilde{\lambda}_j \ge 0$ w.p.1 for $j=1,...,n-2$, and the second term in (\ref{5.2}) is an estimator of $\lambda_j\delta_j$. 
When applying the NR methodology to the PC direction vector, one obtains 
\begin{equation}
\tilde{\bh}_{j}=\{(n-1)\tilde{\lambda}_{j}\}^{-1/2}(\bX-\overline{\bX})\hat{\bu}_{j}
\label{5.3}
\end{equation}
for $j=1,...,n-2$. 
Then, we have the following result. 
\begin{proposition}
Assume (A-vi) and (A-viii). 
It holds for $j=1,...,k$, that 
$\tilde{\lambda}_{j}/\lambda_{j}=1+O_P(n^{-1/2})$ and 
$(\tilde{\bh}_{j}^T\bh_{j})^2=1+O_P(n^{-1})$ 
as $m_0\to \infty$.
\end{proposition}
We note that $\tilde{\bh}_{j}$ is not a unit vector because $||\tilde{\bh}_{j}||^2=\hat{\lambda}_{j}/\tilde{\lambda}_{j}$.
From Propositions 2 and 3, under (A-vi) and (A-viii), it holds that $||\tilde{\bh}_{j}-\bh_{j}||^2=\delta_j\{1+o_P(1)\}+O_P(n^{-1/2})$ as $m_0\to \infty$ for $j=1,...,k$. We note that $2\{1-(1+\delta_j)^{-1/2}\}<\delta_j$. 
Thus, in view of (\ref{5.1}), the norm loss of $\tilde{\bh}_{j}$ is larger than that of $\hat{\bh}_{j}$. 
However, $\tilde{\bh}_{j}$ is a consistent estimator of $\bh_{j}$ in terms of the inner product even when $\delta_j\to \infty$ as $m_0\to \infty$. 

On the other hand, we note that $\bh_{j}^T(\bx_{l}-\bmu)=\lambda_{j}^{1/2}z_{jl}$ for all $j,l$. 
For $\hat{\bh}_{j}$ and $\tilde{\bh}_{j}$, we have the following result. 
\begin{proposition}
Assume (A-vi) and (A-viii). 
It holds for $j=1,...,k\ (l=1,...,n)$ that
$
\lambda_{j}^{-1/2}
\hat{\bh}_{j}^T(\bx_{l}-\bmu)=(1+\delta_j)^{-1/2}[z_{jl}+(n-1)^{1/2}\hat{u}_{jl} \delta_j\{1+o_P(1)\}]
+O_P(n^{-1/2})$ 
and 
$
\lambda_{j}^{-1/2}
\tilde{\bh}_{j}^T(\bx_{l}-\bmu)=z_{jl}+(n-1)^{1/2}\hat{u}_{jl} \delta_j\{1+o_P(1)\}
+O_P(n^{-1/2})$ 
as $m_0\to \infty$.
\end{proposition}
Let us consider the standard deviation of the above quantities.
Note that $[\sum_{l=1}^{n}\{(n-1)^{1/2}\hat{u}_{jl}\delta_j\}^2/n]^{1/2}=O(\delta_j)$ and 
$ \delta_j=O\{p/(n\lambda_{j})\}$ for $\lambda_{k+1}=O(1)$. 
Hence, in Proposition 4, the inner products are very biased when $p$ is large. 
Now, we explain the main reason why the inner products involve the large bias terms. 
Let $\bP_{n}=\bI_{n}-\bone_{n}\bone_{n}^T/n$, where $\bone_{n}=(1,...,1)^T$. 
Note that $\bone_{n}^T\hat{\bu}_{j}=0$ and $\bP_{n}\hat{\bu}_{j}=\hat{\bu}_{j}$ when $\hat{\lambda}_{j}>0$ since $\bone_{n}^T\bS_{D}\bone_{n}=0$. 
Also, when $\hat{\lambda}_{j}>0$, note that
$$
\{(n-1)\tilde{\lambda}_{j}\}^{1/2}\tilde{\bh}_{j}=(\bX-\overline{\bX})\hat{\bu}_{j}=(\bX-\bM)\bP_n\hat{\bu}_{j}=(\bX-\bM)\hat{\bu}_{j},
$$
where $\bM=[\bmu,...,\bmu]$. 
Thus it holds that $\{(n-1)\tilde{\lambda}_{j}\}^{1/2}\tilde{\bh}_{j}^T(\bx_{l}-\bmu)=\hat{\bu}_{j}^T(\bX-\bM)^T(\bx_{l}-\bmu)=\hat{u}_{jl}||\bx_{l}-\bmu ||^2+\sum_{s=1 (\neq l)}^n\hat{u}_{js}(\bx_{s}-\bmu)^T(\bx_{l}-\bmu)$, so that $\hat{u}_{jl}||\bx_{l}-\bmu ||^2$ is very biased since $E(||\bx_{l}-\bmu ||^2)/\{(n-1)^{1/2}\lambda_j\}\ge (n-1)^{1/2} \delta_j$. 
Hence, one should not apply the $\hat{\bh}_{j}$s or the $\tilde{\bh}_{j}$s to the estimation of the inner product. 

Here, we consider a bias-reduced estimation of the inner product. 
Let us write that 
$$
\hat{\bu}_{jl}
=(\hat{u}_{j1},...,\hat{u}_{jl-1},-\hat{u}_{jl}/(n-1),\hat{u}_{jl+1},...,\hat{u}_{jn})^T
$$ 
whose $l$-th element is $-\hat{u}_{jl}/(n-1)$ for all $j,l$. 
Note that $\hat{\bu}_{jl}=\hat{\bu}_{j}-(0,...,0,\{n/(n-1)\}\hat{u}_{jl},0,...,0)^T$ and $\sum_{l=1}^n\hat{\bu}_{jl}/n=\{(n-2)/(n-1)\}\hat{\bu}_{j}$. 
Let 
\begin{equation}
c_n=(n-1)^{1/2}/(n-2) \ \ \mbox{ and } \ \ 
\tilde{\bh}_{jl}=c_n\tilde{\lambda}_{j}^{-1/2}(\bX-\overline{\bX})\hat{\bu}_{jl}
\label{5.4}
\end{equation}
for all $j,l$. 
Note that $\sum_{l=1}^n\tilde{\bh}_{jl}/n=\tilde{\bh}_{j}$.
When $\hat{\lambda}_{j}>0$, 
we note that $c_n^{-1} \tilde{\lambda}_{j}^{1/2}\tilde{\bh}_{jl}=(\bX-\bM)\bP_n\hat{\bu}_{jl}=(\bX-\bM)\hat{\bu}_{j(l)}$ since $\bone_{n}^T\hat{\bu}_{j}=\sum_{l=1}^n\hat{u}_{jl}=0$, where
$$
\hat{\bu}_{j(l)}=(\hat{u}_{j1},...,\hat{u}_{jl-1},0,\hat{u}_{jl+1},...,\hat{u}_{jn})^T+(n-1)^{-1}\hat{u}_{jl}\bone_{n(l)}
\ \ \mbox{for $l=1,...,n$}.
$$ 
Here, $\bone_{n(l)}=(1,...,1,0,1,...,1)^T$ whose $l$-th element is $0$. 
Thus it holds that 
\begin{align*}
c_n^{-1}\tilde{\lambda}_{j}^{1/2}\tilde{\bh}_{jl}^T(\bx_{l}-\bmu)&=\hat{\bu}_{j(l)}^T(\bX-\bM)^T(\bx_{l}-\bmu)\\
&=\sum_{s=1 (\neq l)}^n \{\hat{u}_{js}+(n-1)^{-1}\hat{u}_{jl}\}(\bx_{s}-\bmu)^T(\bx_{l}-\bmu),
\end{align*}
so that the large biased term, $||\bx_{l}-\bmu ||^2$, has vanished. 
Then, we have the following result. 
\begin{proposition}
Assume (A-vi) and (A-viii). 
It holds for $j=1,...,k\ (l=1,...,n)$ that 
$
\lambda_{j}^{-1/2}\tilde{\bh}_{jl}^T(\bx_{l}-\bmu)=z_{jl}+\hat{u}_{jl}\times O_P\{(n^{1/2}\lambda_{j})^{-1} \lambda_{1}\}+O_P(n^{-1/2})
$
as $m_0\to \infty$.
\end{proposition}
Note that $[\sum_{l=1}^{n}\{\hat{u}_{jl}\lambda_{1}/(n^{1/2}\lambda_{j})\}^2/n]^{1/2}=\lambda_{1}/(\lambda_{j}n)$. 
The bias term is small when $\lambda_{1}/\lambda_{j}$ is not large.
\\

\noindent {\bf 5.2. Test Procedure Using Eigenstructures}
\par
\vspace{9pt}
Let $\tilde{x}_{ijl}=\tilde{\bh}_{ijl}^T\bx_{il}$ for all $i,j,l$, where $\tilde{\bh}_{ijl}$s are defined by (\ref{5.4}). 
From Propositions 3 and 5, we consider the following test statistic for (\ref{1.2}): 
\begin{align*}
\widehat{T}_*=&2\sum_{i=1}^2\frac{\sum_{l<l'}^{n_i}(\bx_{il}^T\bx_{il'}-\sum_{j=1}^{k_i}
\tilde{x}_{ijl}\tilde{x}_{ijl'})}{n_i(n_i-1)}\\
&-2\frac{\sum_{l=1}^{n_1}\sum_{l'=1}^{n_2}
(\bx_{1l}-\sum_{j=1}^{k_1}\tilde{x}_{1jl}\tilde{\bh}_{1j})^T
(\bx_{2l'}-\sum_{j=1}^{k_2}\tilde{x}_{2jl'}\tilde{\bh}_{2j})}{n_1n_2},
\end{align*}
where $\tilde{\bh}_{ij}$s are defined by (\ref{5.3}). 
We assume the following conditions when (A-vi) is met.
\begin{description}
  \item[(A-ix)]\ 
  $\displaystyle  \frac{\lambda_{i1}^2}{n_{i}{\Psi}_{i(k_i+1)} }\to 0$ \ as $m\to \infty$ for $i=1,2$;
  \item[(A-x)]\ $\displaystyle  \frac{\bmu_{1*}^T\bSig_{i*}\bmu_{1*}+\bmu_{2*}^T\bSig_{i*}\bmu_{2*} }{{\Psi}_{i(k_i+1)} }\to 0$ \ as $p\to \infty$ \ and \\
  $\displaystyle  \limsup_{m\to \infty} \frac{n_i\{\mu_{i(j)}^2+(\bh_{ij}^T\bmu_{i'*})^2\}}{\lambda_{ij}}<\infty$ $(i'\neq i)$ \ 
 for $i=1,2;\ j=1,...,k_i$.   
\end{description}
Then, we have the following result. 
\begin{theorem}
Assume (A-vi) and (A-viii) to (A-x). 
It holds that $\widehat{T}_*-T_*=o_P(K_{1*}^{1/2})$ as $m\to \infty$. 
Furthermore, assume also $\limsup_{m\to \infty}{\Delta_*^2}/{K_{1*}}<\infty$. 
Then, it holds that $(\widehat{T}_*-\Delta_*)/K_*^{1/2}\Rightarrow N(0,1)$ as $m\to \infty$.
\end{theorem}
By using Lemma 1, it holds that $K_{1*}/K_{*}=1+o(1)$ as $m\to \infty$ under (A-vi) and $\limsup_{m\to \infty}\Delta_*^2/K_{1*}<\infty$. 
Thus, we consider estimating $K_{1*}$.  
Let $\widehat{\bA}_{i(k_i)}=\bI_p-\sum_{j=1}^{k_i}\hat{\bh}_{ij}\hat{\bh}_{ij}^T$ for $i=1,2$. 
We estimate $K_{1*}$ by 
$$
\widehat{K}_{1*}=2\sum_{i=1}^2\frac{\widehat{\Psi}_{i(k_i+1)}  }{n_i(n_i-1)}
+4\frac{ \tr(\bS_{1n_1}\widehat{\bA}_{1(k_1)} \bS_{2n_2}\widehat{\bA}_{2(k_2)})}{n_1n_2}, 
$$
where $\widehat{\Psi}_{i(k_i+1)}$ is defined by (S2.1) of the supplementary material.  
Then, we have the following result.
\begin{lemma}
Assume (A-vi), (A-viii) and (A-ix). 
It holds that $\widehat{K}_{1*}/K_{1*}=1+o_P(1)$ as $m\to \infty$. 
\end{lemma}
Now, we test (\ref{1.2}) by 
\begin{equation}
\mbox{rejecting $H_0$}\Longleftrightarrow \widehat{T}_*/\widehat{K}_{1*}^{1/2}>z_{\alpha}.
\label{5.5}
\end{equation}
Let power($\Delta_*$) denote the power of the test (\ref{5.5}). 
Then, from Theorem 5 and Lemma 3, we have the following result.  
\begin{theorem}
Assume (A-vi) and (A-vii) to (A-x).
The test (\ref{5.5}) has as $m\to \infty$
$$
\mbox{size}=\alpha+o(1) \quad \mbox{and}\quad \mbox{power$(\Delta_*)$}-\Phi \bigg(\frac{\Delta_*}{K_{*}^{1/2} }
-z_{\alpha}\Big( \frac{K_{1*}}{K_{*}}\Big)^{1/2}  \bigg)=o(1).
$$
\end{theorem}
In general, $k_i$s are unknown in $\widehat{T}_*$ and $\widehat{K}_{1*}$.
See Section S2.2 in the supplementary material for estimation of $k_i$s. 
If (\ref{4.1}) is met, one may use the test (\ref{4.2}). 
However, under (\ref{4.1}), (A-vi) and $\limsup_{m\to \infty}\Delta_*^2/K_{1*}<\infty$, 
we note that $\Var(T_*)/\Var(T_I)=O(K_{1*}/K_1 )\to 0$ as $m\to \infty$, so that the power of (\ref{4.2}) must be lower than that of (\ref{5.5}). 
See Section 6 for numerical comparisons. 
We recommend to use the test (\ref{5.5}) for the SSE model in general. 
\\

\noindent {\bf 5.3. How to Check SSE Models and Estimate Parameters}
\par
\vspace{9pt}
We provide a method to distinguish between the NSSE model defined by (\ref{1.5}) and the SSE model defined by (\ref{1.7}).
We also give a method to estimate the parameters required in the test procedure (\ref{5.5}). 
We summarized the results in Section S2 of the supplementary material.
\\

\noindent {\bf 5.4. Demonstration}
\par
\vspace{9pt}
We introduce two high-dimensional data sets that have the SSE model.
We demonstrate the proposed test procedure by (\ref{5.5}) by using the microarray data sets. 
We summarized the results in Section S3 of the supplementary material. 
\\

\setcounter{chapter}{6}
\setcounter{equation}{0} 
\noindent {\bf 6. Simulations for Strongly Spiked Eigenvalue Model}
\par
\vspace{9pt}
We used computer simulations to study the performance of the test procedures by (\ref{4.2}) and (\ref{5.5}) for the SSE model. 
In general, $k_i$s are unknown for (\ref{5.5}). 
Hence, we estimated $k_i$ by $\hat{k}_{i}$, where $\hat{k}_{i}$ is given in Section S2.2 of the supplementary material. 
We set $\kappa(n_i)=(n_i^{-1} \log{n_i})^{1/2}$ in (S2.2) of the supplementary material.
We checked the performance of the test procedure by (\ref{5.5}) with $k_i=\hat{k}_i$, $i=1,2$. 
We considered a naive estimator of $T_{*}$ as $T(\widehat{\bA}_{1(k_1)},\widehat{\bA}_{2(k_2)})$ and checked the performance of the test procedure given by 
\begin{equation}
\mbox{rejecting $H_0$}\Longleftrightarrow T(\widehat{\bA}_{1(k_1)},\widehat{\bA}_{2(k_2)})/\widehat{K}_{1*}^{1/2}>z_{\alpha}.
\label{6.1}
\end{equation}
We also checked the performance of the test procedure by (\ref{3.1}) with $\bA=\bI_p$. 
We set $\alpha=0.05$, $\bmu_1=\bze$ and 
$$
\bSig_{i}=\left( \begin{array}{cc}
\bSig_{(1)} & \bO_{2, p-2} \\
\bO_{p-2, 2} & c_i\bSig_{(2)}
\end{array} \right) \ \ \mbox{with $\bSig_{(1)}=\mbox{diag}(p^{2/3},p^{1/2})$ and $\bSig_{(2)}=(0.3^{|i-j|^{1/2}})$}
$$
for $i=1,2$, 
where $\bO_{l,l'}$ is the $l\times l'$ zero matrix and $(c_1,c_2)=(1,1.5)$. 
Note that (\ref{4.1}) and (A-vi) with $k_1=k_2=2$ are met.
When considering the alternative hypothesis, we set $\bmu_2=(0,...,0,1,1,1,1)^T$ whose last $4$ elements are $1$. 
We considered three cases: \\[2mm] 
(a) $\pi_i:N_p(\bmu_i, \bSig_i)$, $p=2^s,\ n_1=3\lceil p^{1/2} \rceil$ and  
$n_2=4 \lceil p^{1/2} \rceil$ for $s=4,...,10$; \\
(b) $\bz_{ij}$s are i.i.d. as $p$-variate $t$-distribution, $t_p(\nu)$, with mean zero, covariance matrix $\bI_p$ and degrees of freedom $\nu=15$, $(n_1,n_2)=(40,60)$ and $p=50+100(s-1)$ for $s=1,...,7$; and \\
(c) $z_{itj}=(v_{itj}-5)/{10}^{1/2}$ $(t=1,...,p)$ in which $v_{itj}$s are i.i.d. as $\chi_5^2$, $p=500$, $n_1=10s$ and $n_2=1.5n_1$ for $s=2,...,8$. \\[2mm]
Note that (A-viii) is met both for (a) and (c). 
However, (A-viii) (or (A-i)) is not met for (b). 
Similar to Section 3.3, we calculated $\overline{\alpha}$ and $1-\overline{\beta}$ with 2000 replications for five test procedures: (I) from (\ref{3.1}) with $\bA=\bI_p$, (II) from (\ref{4.2}), (III) from (\ref{5.5}), (IV) from (\ref{5.5}) with $k_i=\hat{k}_i$, $i=1,2$, and (V) from (\ref{6.1}). 
Their standard deviations are less than $0.011$. 
In Fig. 2, for (a) to (c), we plotted $\overline{\alpha}$ in the left panel and $1-\overline{\beta}$ in the right panel. 
From Theorem 6, we plotted the asymptotic power, $\Phi(\Delta_*/K_*^{1/2}-z_{\alpha}(K_{1*}/K_*)^{1/2})$, for (III). 
\begin{figure}[h!]
\includegraphics[scale=0.63]{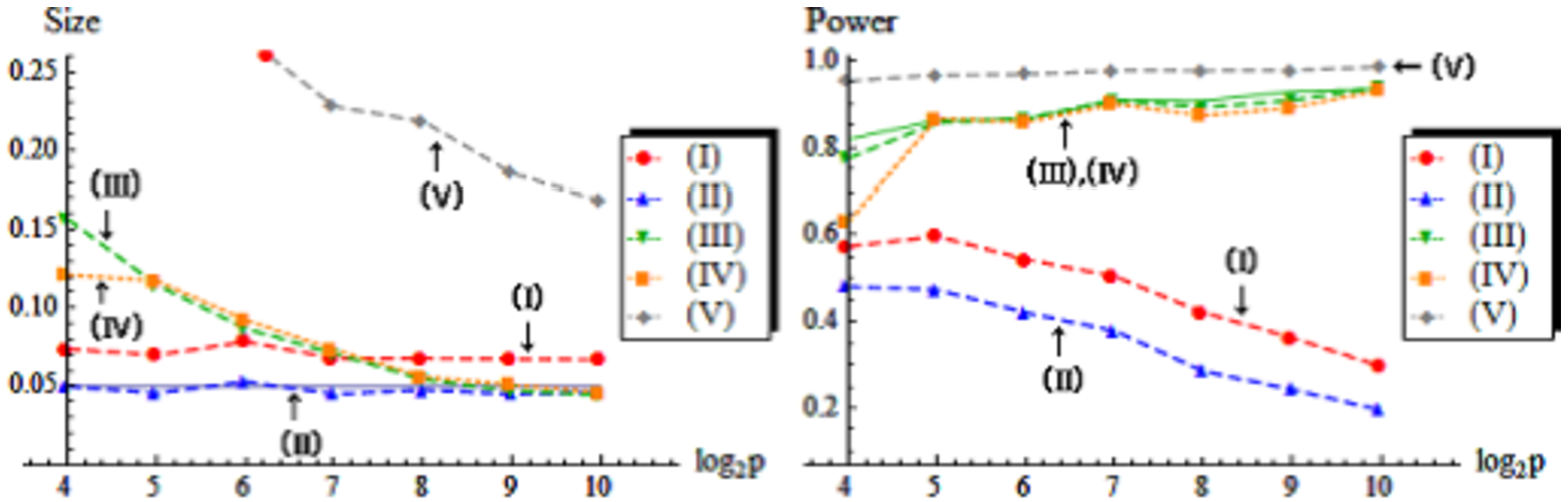} \\[-0.7mm]
(a) $\pi_i:N_p(\bmu_i, \bSig_i)$, $p=2^s,\ n_1=3 \lceil p^{1/2} \rceil$ and  
$n_2=4 \lceil p^{1/2} \rceil$ for $s=4,...,10$. \\[2mm]
\includegraphics[scale=0.63]{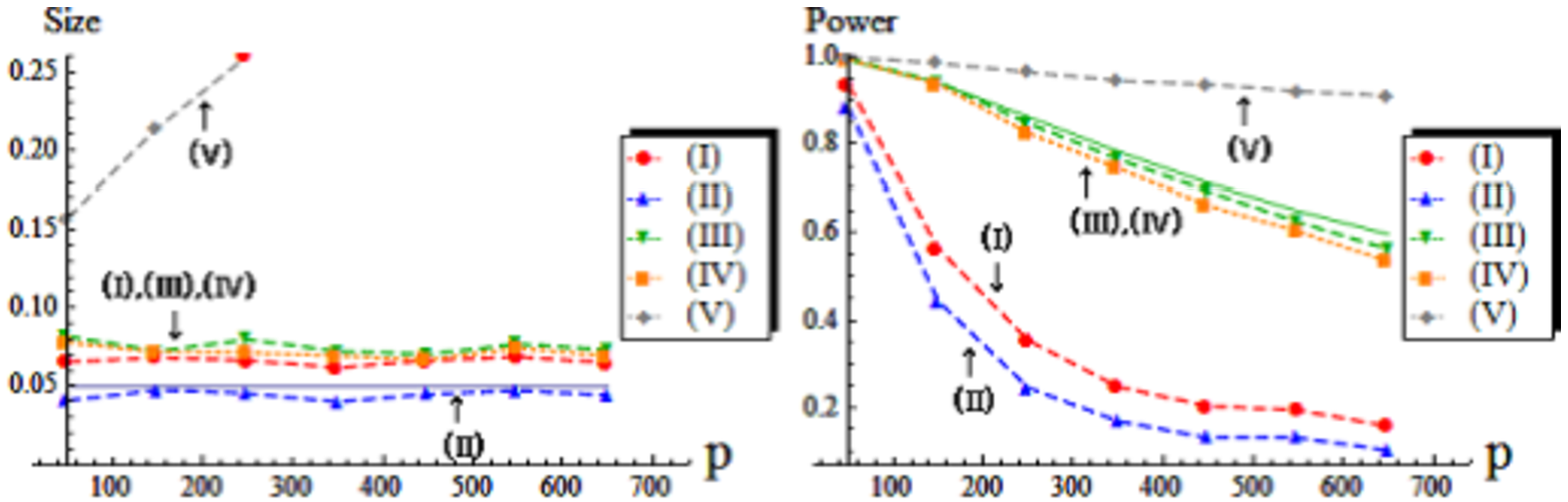}  \\[-0.7mm]
(b) $\bz_{ij}$s are i.i.d. as $t_p(15)$, $(n_1,n_2)=(40,60)$ and $p=50+100(s-1)$ for $s=1,...,7$. 
\\[2mm]
\includegraphics[scale=0.63]{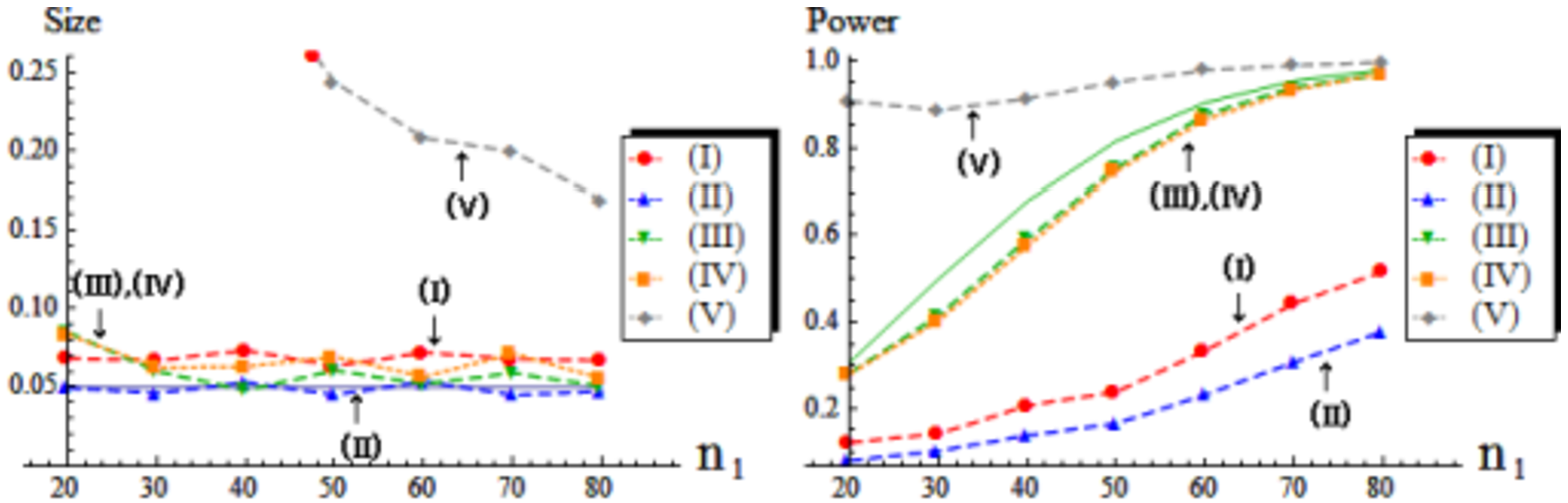}  \\[-0.7mm]
(c) $z_{irj}=(v_{itj}-5)/{10}^{1/2}$ $(t=1,...,p)$ in which $v_{itj}$s are i.i.d. as $\chi_5^2$, $p=500$, $n_1=10s$ and $n_2=1.5n_1$ for $s=2,...,8$. 
\caption{The performances of five tests: (I) from (\ref{3.1}) with $\bA=\bI_p$, (II) from (\ref{4.2}), (III) from (\ref{5.5}), (IV) from (\ref{5.5}) with $k_i=\hat{k}_i$, $i=1,2$, and (V) from (\ref{6.1}). 
For (a) to (c), the values of $\overline{\alpha}$ are denoted by the dashed lines in the left panel and the values of $1-\overline{\beta}$ are denoted by the dashed lines in the right panel.
The asymptotic power of (III) was given by $\Phi(\Delta_*/K_*^{1/2}-z_{\alpha}(K_{1*}/K_*)^{1/2})$ 
which is denoted by the solid line in the right panels. 
When $n_i$s are small or $p$ is large, $\overline{\alpha}$ for (V) was too high to describe. 
}
\end{figure}

We observed that (II) gives better performances compared to (I) regarding the size. 
The size by (I) did not become close to $\alpha$. 
This is probably because $T_I$ does not satisfy the asymptotic normality given in Theorem 2 when (\ref{1.5}) is not met. 
On the other hand, (II) (or (I)) gave quite bad performances compared to (III) and (IV) regarding the power. 
This is probably because $\Var(T_I)/\Var(T_*)\to \infty$ as $p\to \infty$ in the current setting. 
The size of (V) was much higher than $\alpha$. 
This is probably because of the bias of $T(\widehat{\bA}_{1(k_1)},\widehat{\bA}_{2(k_2)})$. 
See Section 5.1 for the details. 
We observed that (III) and (IV) give adequate performances even in the non-Gaussian cases.
The performances of (III) and (IV) became quite similar to each other in almost all cases. 
When $p$ and $n_{i}$s are not small, the plots of (IV) became close to the theoretical values. 
Hence, we recommend to use the test procedure by (\ref{5.5}) with $k_i=\hat{k}_i$, $i=1,2$ when (\ref{1.7}) holds. 

We also checked the performance of the test procedures for the MSN distribution and the multivariate skew $t$ (MST) distribution.  
See \cite{Azzalini:2003} and \cite{Gupta:2003} for the details of the MST distribution.
We gave the results in Section S4.2 of the supplementary material. 
\\

\setcounter{chapter}{9}
\setcounter{equation}{0} 
\noindent {\bf 7. Conclusion}
\par
\vspace{9pt}
By classifying eigenstructures into two classes, the SSE and NSSE models, and then selecting a suitable test procedure depending on the eigenstructure, we can quickly obtain a much more accurate result at lower computational cost.
These benefits are vital in groundbreaking research of medical diagnostics, engineering, big data analysis, etc. 
\\[1.4cm]
\begin{center}
{\Large
{\bf Supplementary Material}}
\end{center}

\setcounter{section}{0}
\setcounter{equation}{0}
\def\theequation{S\arabic{section}.\arabic{equation}}
\def\thesection{S\arabic{section}}

\def\theproposition{S\arabic{section}.\arabic{proposition}}

\def\thelemma{S\arabic{section}.\arabic{lemma}}

\section{Additional Propositions}
\setcounter{proposition}{0}
\noindent

In this section, we give two propositions and proofs of the propositions. 
\subsection{Proposition S1.1}
\begin{proposition}
Let $\Theta $ be the set of positive definite matrices of dimension $p$.
It holds that
$$
\argmax_{\bsA \in   \bTheta  }\Big\{ \frac{\Delta(\bA)}{\{K_{2}(\bA)\}^{1/2}}\Big\}
=c(\bSig_1/n_1+\bSig_2/n_2)^{-1}
$$
for any constant $c>0$.
\end{proposition}
\begin{proof}
We assume $\bA \in \bTheta$.
Let $\dot{\bmu}_{\bsA}=\bmu_{\bsA}/||\bmu_{\bsA} ||$ and 
$\bSig_{\bsA \star}=\bSig_{1,\bsA}/n_1+\bSig_{2,\bsA}/n_2$. 
Then, we have that 
$$
2\Delta(\bA)/\{K_{2}(\bA)\}^{1/2}
=||\bmu_{\bsA} ||/(\dot{\bmu}_{\bsA}^T \bSig_{\bsA \star}\dot{\bmu}_{\bsA})^{1/2}.
$$
The eigen-decomposition of $\bSig_{\bsA \star}$ is given by 
$\bSig_{\bsA \star}=\bH_{\bsA}\bLam_{\bsA}\bH_{\bsA}^T$, 
where $\bLam_{\bsA}=\mbox{diag}(\lambda_{1,\bsA},...,\lambda_{p,\bsA})$ is a diagonal matrix of eigenvalues, $\lambda_{1,\bsA}\ge \cdots \ge\lambda_{p,\bsA}>0$, and $\bH_{\bsA}=[\bh_{1,\bsA},...,\bh_{p,\bsA}]$ is an orthogonal matrix of the corresponding eigenvectors. 
There exist some constants $c_1,...,c_p$ such that $\dot{\bmu}_{\bsA}=\sum_{j=1}^pc_j\bh_{j,\bsA}$ and $\sum_{j=1}^pc_j^2=1$. 
From Schwarz's inequality, it holds that 
$
( \dot{\bmu}_{\bsA}^T\bSig_{\bsA\star} \dot{\bmu}_{\bsA})(\dot{\bmu}_{\bsA}^T\bSig_{\bsA \star}^{-1}\dot{\bmu}_{\bsA})=
(\sum_{j=1}^pc_j^2 \lambda_{j,\bsA} )(\sum_{j=1}^pc_j^2 \lambda_{j,\bsA}^{-1} )\ge 1, 
$
so that 
$$
||\bmu_{\bsA} ||/(\dot{\bmu}_{\bsA}^T \bSig_{\bsA\star}\dot{\bmu}_{\bsA})^{1/2}\le 
(||\bmu_{\bsA} ||^2 \dot{\bmu}_{\bsA}^T \bSig_{\bsA \star}^{-1} \dot{\bmu}_{\bsA})^{1/2}=
\{(\bmu_{1}-\bmu_2)^T(\bSig_1/n_1+\bSig_2/n_2)^{-1}(\bmu_{1}-\bmu_2)\}^{1/2}.
$$
Note that 
$||\bmu_{\bsA} ||/(\dot{\bmu}_{\bsA}^T \bSig_{\bsA\star}\dot{\bmu}_{\bsA})^{1/2}=\{(\bmu_{1}-\bmu_2)^T(\bSig_1/n_1+\bSig_2/n_2)^{-1}(\bmu_{1}-\bmu_2)\}^{1/2}$ 
when $\bA=c(\bSig_1/n_1+\bSig_2/n_2)^{-1}$ for any constant $c>0$.
It concludes the result.
\end{proof}
\subsection{Proposition S1.2}
\noindent

Let us write that $\bmu_{\bsA_{12}}=\bA_1^{1/2}\bmu_1-\bA_2^{1/2}\bmu_2$ and $\bSigma_{i,\bsA_i}=\bA_i^{1/2}\bSig_i\bA_i^{1/2}$, $i=1,2$. 
Let $\Delta({\bA}_1,{\bA}_2)=||\bmu_{\bsA_{12}}||^2$ and 
$K({\bA}_1,{\bA}_2)=K_{1}({\bA}_1,{\bA}_2)+K_{2}({\bA}_1,{\bA}_2)$, 
where 
$K_{1}({\bA}_1,{\bA}_2)=2 \sum_{i=1}^2\tr(\bSig_{i,\bsA_i}^2)/\{n_i(n_i-1)\}+4\tr(\bSig_{1,\bsA_i}\bSig_{2,\bsA_i})/(n_1n_2)$ 
and $ K_{2}(\bA_1,\bA_2)=4\sum_{i=1}^2 \bmu_{\bsA_{12}}^T\bSig_{i,\bsA}\bmu_{\bsA_{12}}/n_i$. 
Note that $E\{T(\bA_{1},\bA_{2})\}=\Delta({\bA}_1,{\bA}_2)$ and $\Var\{T(\bA_{1},\bA_{2})\}=K({\bA}_1,{\bA}_2)$. 
Then, we have the following result. 
\begin{proposition}
Assume (A-i) and the following conditions:
\begin{description}
  \item[(S-i)]\quad $\displaystyle \frac{\{\lambda_{\max}(\bSigma_{i,\bsA_i})\}^2}{\tr(\bSigma_{i,\bsA_i}^2)}\to 0$
  \ as $p\to \infty$ for $i=1,2$;
  \item[(S-ii)]\quad $\displaystyle  \frac{\{\Delta(\bA_1,\bA_2)\}^2}{K_{1}(\bA_1,\bA_2)}\to 0$ \ as $m\to \infty$ under $H_0$. 
\end{description}
Then, it holds that as $m\to \infty$
$$
P\Big(\frac{T(\bA_1,\bA_2)}{ \{K_1(\bA_1,\bA_2)\}^{1/2} }>z_{\alpha} \Big)=\alpha+o(1) \ \ \mbox{under $H_0$}.
$$
\end{proposition}
\begin{proof}
From Theorem 2 and Lemma 1, the result is obtained straightforwardly.
\end{proof}
Note that (S-i) is naturally met when $\bA_i=\bSig_i^{-1},\ i=1,2$, because $\bSigma_{i,\bsA_i}=\bI_p$ when $\bA_i=\bSig_i^{-1}$. 
However, (S-ii) is difficult to meet when $\bSig_1\neq \bSig_2$ and $\bA_i=\bSig_i^{-1},\ i=1,2$. 
For example, when $\bSig_1=c\bSig_2=\bI_p\ (c>1)$ and $\bmu_1=\bmu_2=(1,...,1)^T$, it follows that $\Delta(\bSig_1^{-1},\bSig_2^{-1})=(1-c^{1/2})^2p$. 
Then, (S-ii) does not hold because $K_{1}(\bSig_1^{-1},\bSig_2^{-1})=O(p/n_{\min}^2)$. 
Hence, we do not recommend to choose $\bA_i=\bSig_i^{-1},\ i=1,2$. 
In addition, it is difficult to estimate $\bSig_i^{-1}$s for high-dimension, non-sparse data. 
\section{How to Check SSE Models and Estimate Parameters}
\setcounter{proposition}{0}
\setcounter{lemma}{0}
\noindent

In this section, we provide a method to distinguish between the NSSE model defined by (1.4) and the SSE model defined by (1.6).
We also give a method to estimate the parameters required in the test procedure (5.5).
\subsection{Checking Whether (1.4) Holds or Not}
\noindent

As discussed in Section 3, we recommend to use the test by (3.1) with $\bA=\bI_p$ when (A-ii) is met, otherwise the test by (5.5). 
It is crucial to check whether (1.4) holds or not (that is, whether (1.6) holds).

Let $\hat{\eta}_i=\tilde{\lambda}_{i1}^2/W_{in_i}$ for $i=1,2$, where 
$W_{in_i}$s are defined in Section 2.2 and $\tilde{\lambda}_{ij}$s are defined by (5.2). 
Then, we have the following result. 
\begin{proposition}
Assume (A-i). 
It holds that as $m\to \infty$
\begin{align*}
&\hat{\eta}_i=o_P(1) \quad \mbox{for $i=1,2$, under (1.4)};\\
&P(\hat{\eta}_i>c)\to 1\quad \mbox{with some fixed constant $c\in(0,1)$ for some $i$ under (1.6)}.
\end{align*}
\end{proposition}
By using Proposition S2.1, one can distinguish between (1.4) and (1.6). 
One may claim (1.4) if both $\hat{\eta}_1$ and $\hat{\eta}_2$ are sufficiently small, otherwise (1.6). 
In addition, we have the following result for $\hat{\eta}_i$.
\begin{proposition}
Assume (A-viii). 
Assume also $\lambda_{i1}^2/\tr(\bSig_i^2)=O(n_i^{-c})$ as $m\to \infty$ 
with some fixed constant $c>1/2$ for $i=1,2$. 
It holds as $m\to \infty$
$$
P\Big(\hat{\eta}_i< \kappa(n_i) \Big) \to 1 \ \ \mbox{for $i=1,2$},
$$
where $\kappa(n_i)$ is a function such that $\kappa(n_i)\to 0$ and $n_i^{1/2}\kappa(n_i)\to \infty$ as $n_i\to \infty$.   
\end{proposition}
From Proposition S2.2 one may claim (1.4) if $\hat{\eta}_i<\kappa(n_i)$ both for $i=1,2$, 
otherwise (1.6).
One can choose $\kappa(n_i)$ such as $(n_i^{-1} \log{n_i})^{1/2}$ or $n_i^{-c}$ with $c\in (0,1/2)$. 
In Section S3, we use $\kappa(n_i)=(n_i^{-1} \log{n_i})^{1/2}$ in actual data analyses. 
\subsection{Estimation of ${\Psi}_{i(j)}$ and $k_i$}
\noindent

Let $n_{i(1)}=\lceil  n_i/2 \rceil$ and $n_{i(2)}=n_i-n_{i(1)}$.
Let $\bX_{i1}=[\bx_{i1},...,\bx_{in_{i(1)}}]$ and $\bX_{i2}=[\bx_{in_{i(1)}+1},...,\bx_{in_i}]$ for $i=1,2$. 
We define 
$$
\bS_{iD(1)}=\{(n_{i(1)}-1)(n_{i(2)}-1)\}^{-1/2}(\bX_{i1}-\overline{\bX}_{i1})^T(\bX_{i2}-\overline{\bX}_{i2})
$$ 
for $i=1,2$, where $\overline{\bX}_{ij}=[\overline{\bx}_{in_i(j)},...,\overline{\bx}_{in_i(j)}]$ with $\overline{\bx}_{in_i(1)}=\sum_{l=1}^{n_{i(1)}}\bx_{il}/n_{i(1)}$ and $\overline{\bx}_{in_i(2)}=\sum_{l=n_{i(1)}+1}^{n_{i}}\bx_{il}/n_{i(2)}$. 
By using the cross-data-matrix (CDM) methodology by \citet{Yata:2010}, we estimate $\lambda_{ij}$ by the $j$-th singular value, $\acute{\lambda}_{ij}$, of $\bS_{iD(1)}$, where $\acute{\lambda}_{i1}\ge \cdots \ge \acute{\lambda}_{in_{i(2)}-1}\ge 0$.
\citet{Yata:2010,Yata:2013b} showed that $\acute{\lambda}_{ij}$ has several consistency properties for high-dimensional non-Gaussian data.  
\citet{Aoshima:2011} applied the CDM methodology to obtaining an unbiased estimator of $\tr(\bSig_i^2)$ by $\tr(\bS_{iD(1)}\bS_{iD(1)}^T)$, $i=1,2$.  
Note that $E\{\tr(\bS_{iD(1)}\bS_{iD(1)}^T)\}=\tr(\bSig_i^2)$. 
Based on the CDM methodology, we consider estimating ${\Psi}_{i(j)}$ as follows: 
Let $\widehat{\Psi}_{i(1)}=\tr(\bS_{iD(1)}\bS_{iD(1)}^T)$ and 
\begin{equation}
\widehat{\Psi}_{i(j)}=\tr(\bS_{iD(1)}\bS_{iD(1)}^T)-\sum_{l=1}^{j-1}\acute{\lambda}_{il}^2\quad \mbox{for $i=1,2;\ j=2,...,n_{i(2)}$}.
\label{s2.1}
\end{equation}
Note that $\widehat{\Psi}_{i(j)}\ge 0$ w.p.1 for $j=1,...,n_{i(2)}$. 
Then, we have the following result.
\begin{lemma}
Assume (A-i) and (A-vi). 
Then, it holds that ${\widehat{\Psi}_{i(j)}}/{{\Psi}_{i(j)}}=1+o_P(1)$ as $m\to \infty$ for $i=1,2;\ j=1,...,k_i+1$.
\end{lemma}
Let $\hat{\tau }_{i(j)}=\widehat{\Psi}_{i(j+1)}/\widehat{\Psi}_{i(j)}\ (=1-\acute{\lambda}_{ij}^2/\widehat{\Psi}_{i(j)})$ for $i=1,2$. 
Note that $\hat{\tau }_{i(j)}\in [0,1)$ for $\acute{\lambda}_{ij}>0$. 
Then, we have the following result. 
\begin{proposition}
Assume (A-i) and (A-vi). 
It holds for $i=1,2$ that as $m\to \infty$
\begin{align*}
&P(\hat{\tau }_{i(j)}<1-c_j)\to 1 \ \ \mbox{with some fixed constant $c_j\in(0,1)$ for $j=1,...,k_i$};\\
&
\hat{\tau }_{i(k_i+1)}=1+o_P(1).
\end{align*}
\end{proposition}
From Proposition S2.3, one may choose $k_i$ as the first integer $j$ such that $1-\hat{\tau }_{i(j+1)}$ is sufficiently small. 
In addition, we have the following result for $\hat{\tau }_{i(k_i+1)}$.
\begin{proposition}
Assume (A-vi), (A-viii) and (A-ix). 
Assume also $\lambda_{ik_i+1}^2/\Psi_{i(k_i+1)}=O(n_i^{-c})$ as $m\to \infty$ 
with some fixed constant $c>1/2$ for $i=1,2$. 
It holds for $i=1,2$ that as $m\to \infty$
$$
P\Big(\hat{\tau }_{i(k_i+1)}> \{1+(k_i+1)\kappa(n_i)\}^{-1} \Big) \to 1,
$$
where $\kappa(n_i)$ is defined in Proposition S2.2.  
\end{proposition}
From Propositions S2.3 and S2.4, if one can assume the conditions in Proposition S2.4, 
one may consider $k_i$ as the first integer $j\ (=\hat{k}_{oi},\ \mbox{say})$ such that 
\begin{equation}
\hat{\tau }_{i(j+1)}\{1+(j+1)\kappa(n_i)\}>1 \quad (j\ge 0). \label{s2.2}
\end{equation}
Then, it holds that $P(\hat{k}_{oi}=k_i)\to 1$ as $m\to \infty$. 
Note that $\widehat{\Psi}_{i(n_{i(2)})}=0$ from the fact that rank$(\bS_{iD(1)})\le n_{i(2)}-1$. 
Thus one may choose $k_i$ as $\hat{k}_{i}=\min\{\hat{k}_{oi},n_{i(2)}-2\}$ in actual data analyses. 
For $\kappa(n_i)=(n_i^{-1} \log{n_i})^{1/2}$ in (\ref{s2.2}), 
the test procedure by (5.5) with $k_i=\hat{k}_{i}$, $i=1,2$, gave preferable performances throughout our simulations in Sections 6 and S4.2. 
If $\hat{k}_i=0$ (that is, (\ref{s2.2}) holds when $j=0$), one may consider the test with $\bA_{i(k_i)}=\bI_p$. 
In addition, if $\hat{k}_i=0$ for $i=1,2$, we recommend to use the test by (3.1) with $\bA=\bI_p$. 
\section{Demonstration}
\noindent

In this section, we introduce two high-dimensional data sets that have the SSE model.
We demonstrate the proposed test procedure by (5.5) by using the microarray data sets. 
We set $\alpha=0.05$. 

We first analyzed leukemia data with $7129\ (=p)$ genes consisting of $\pi_1:$ acute lymphoblastic leukemia ($n_1=47$ samples) and $\pi_2:$ acute myeloid leukemia ($n_2=25$ samples) given by \citet{Golub:1999}. 
We transformed each sample by $\bx_{ij}-(\overline{\bx}_{1n_1}+\overline{\bx}_{2n_2})/2$ for all $i,j$, so that $\bmu_1=\bmu_2=0$ under $H_0:\ \bmu_1=\bmu_2$. 
Then, (A-vii) and (A-x) hold under $H_0$. 
We calculated that $\hat{\eta}_1=0.697$ and $\hat{\eta}_2=0.602$. 
Since $\hat{\eta}_i$s are larger than $(n_1^{-1} \log{n_1})^{1/2}=0.286$ or $(n_2^{-1} \log{n_2})^{1/2}=0.359$, we concluded from Proposition S2.2 that (1.6) holds for $i=1,2$. 
We used the test procedure by (5.5). 
We set $\kappa(n_i)=(n_i^{-1} \log{n_i})^{1/2}$ in (\ref{s2.2}). Let $\tilde{\tau }_{i(j)}=\hat{\tau }_{i(j)}\{1+j\kappa(n_i)\}$ for all $i,j$. 
We calculated that $(\tilde{\tau }_{1(1)},\tilde{\tau }_{1(2)},\tilde{\tau }_{1(3)})=(0.407, 0.993, 1.302)$ 
and $(\tilde{\tau }_{2(1)},\tilde{\tau }_{2(2)},\tilde{\tau }_{2(3)},\tilde{\tau }_{2(4)})=(0.579, 0.7, 0.902,1.307)$, so that $\hat{k}_{1}=2$ and $\hat{k}_{1}=3$.
Thus, we chose $k_1=2$ and $k_2=3$.
We calculated that $\widehat{T}_*/\widehat{K}_{1*}^{1/2}=46.866$. 
By using (5.5), we rejected $H_0$ with size $0.05$ according to the arguments in Section 5.2. 

Next, we analyzed prostate cancer data with $12625\ (=p)$ genes consisting of $\pi_1:$ normal prostate ($n_1=50$ samples) and $\pi_2:$ prostate tumor ($n_2=52$ samples) given by \citet{Singh:2002}. 
We transformed each sample as before. 
We calculated that $(\hat{\eta}_1,\hat{\eta}_2)=(1.01,1.009)$ and $(\hat{k}_{1},\hat{k}_{2})=(4,3)$ from (\ref{s2.2}) with $\kappa(n_i)=(n_i^{-1} \log{n_i})^{1/2}$. 
Hence, we used the test procedure by (5.5) with $k_1=4$ and $k_2=3$.
Then, we calculated that $\widehat{T}_*/\widehat{K}_{1*}^{1/2}=27.497$. 
Hence, we rejected $H_0$ by using (5.5). 
In addition, we considered two cases: (a) $\pi_1:$ the first 25 samples ($n_1=25$) and $\pi_2:$ the last 25 samples ($n_2=25$) from the normal prostate; and 
(b) $\pi_1:$ the first 26 samples ($n_1=26$) and $\pi_2:$ the last 26 samples ($n_2=26$) from the prostate tumor. 
Note that $H_0$ is true for (a) and (b). 
We applied the test procedure by (5.5) to the cases. 
Then, we accepted $H_0$ both for (a) and (b). 
We also applied the test procedures by (3.1) with $\bA=\bI_p$ and (4.2) to the cases. 
Then, $H_0$ was rejected by them both for (a) and (b).

\section{Additional Simulations}
\setcounter{equation}{0}
\noindent

In this section, we give additional simulations for Sections 3.3 and 6. 
\subsection{Simulations for NSSE Model}
\noindent

In this section, we give additional simulations for Section 3.3 under the NSSE model. 

We set $\alpha=0.05$, $p=2^{s},\ s=4,...,10$, $n_1=\lceil p^{1/2} \rceil$, $n_2=2n_1$ and $\bmu_1=\bze$. 
When considering the alternative hypothesis, we set $\bmu_2=(1,...,1, 0,...,0,-1,...,-1)^T$ whose first $5$ elements are $1$ and last $5$ elements are $-1$. 
We generated $\breve{\bx}_{ij}$, $j=1,2,...,\ (i=1,2)$ independently from a multivariate 
skew normal (MSN) distribution, SN$_p(\bOme, \bal)$, with correlation matrix $\bOme=(0.3^{|i-j|^{1/2}})$ and shape parameter vector $\bal$. 
Note that $E(\breve{\bx}_{ij})=(2/\pi)^{1/2}\bOme\bal/(1+\bal^T\bOme\bal)^{1/2}\ (=\breve{\bmu}$, say) and 
Var$(\breve{\bx}_{ij})=\bOme-\breve{\bmu}\breve{\bmu}^T\ (=\breve{\bSig}$, say). 
We set $\bx_{ij}=c_i^{1/2}(\breve{\bx}_{ij}-\breve{\bmu})+\bmu_i$ for all $i,j$, where $(c_1,c_2)=(1,1.5)$. 
Note that $\bSig_1=\breve{\bSig}$ and $\bSig_2=1.5\breve{\bSig}$. 
We considered three cases: (a) $\bal=\bone_p$; (b) $\bal=4\bone_p$; and (c) $\bal=16\bone_p$, 
where $\bone_p=(1,...,1)^T$. 
See \cite{Azzalini:1996} and \cite{Azzalini:1999} for the details of the MSN distribution.
Note that (1.4) is met.
Also, note that (A-i) is met. 
See Remark S4.1.
Similar to Section 3.3, we calculated $\overline{\alpha}$ and $1-\overline{\beta}$ with 2000 replications for 
the test procedures given by (3.1) with (I) $\bA=\bI_p$, (II) $\bA=\bA_{\star}$, (III) $\bA=\bA_{\star(d)}$ and (IV) $\bA=\widehat{\bA}_{\star(d)}$. 
Note that (A-iv) is met for (I) to (III).
In Fig. S4.1, for (a) to (c), we plotted $\overline{\alpha}$ in the left panel and $1-\overline{\beta}$ in the right panel.
We also plotted the asymptotic power, $\Phi(\Delta(\bA)/\{K(\bA)\}^{1/2}-z_{\alpha}\{K_1(\bA)/K(\bA)\}^{1/2})$, for (I) to (III) by using Theorem 3. 
\begin{figure}[h!]
\includegraphics[scale=0.48]{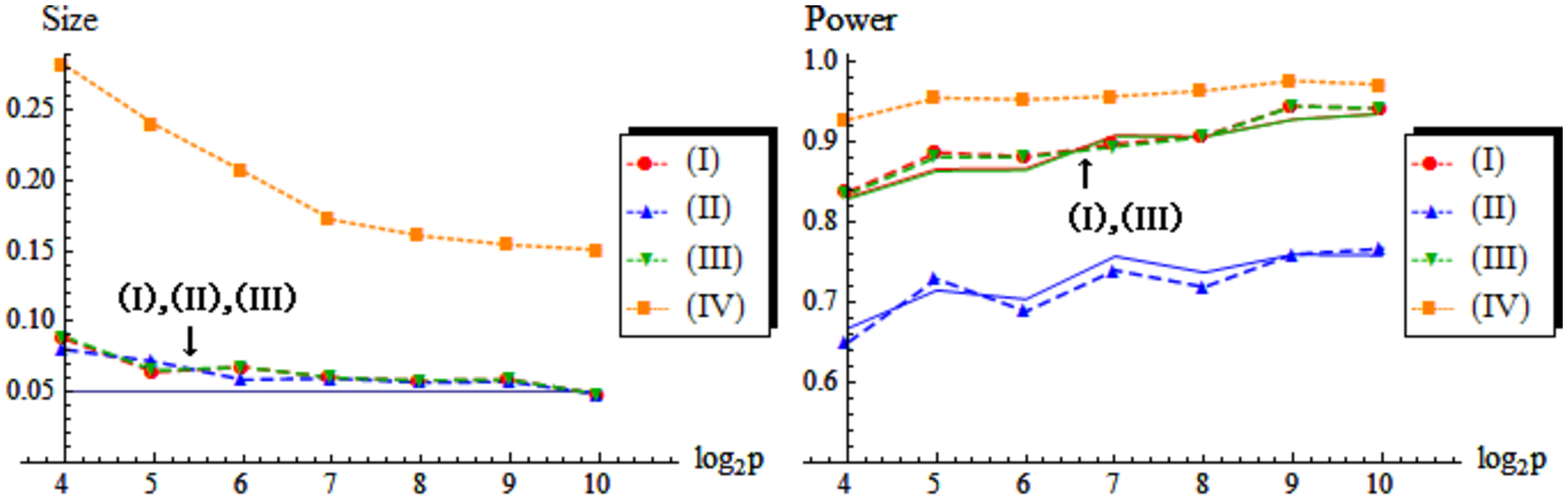} \\[-0.5mm]
(a) SN$_p(\bOme, \bal)$ with $\bal=\bone_p$.
 \\[4mm]
\includegraphics[scale=0.48]{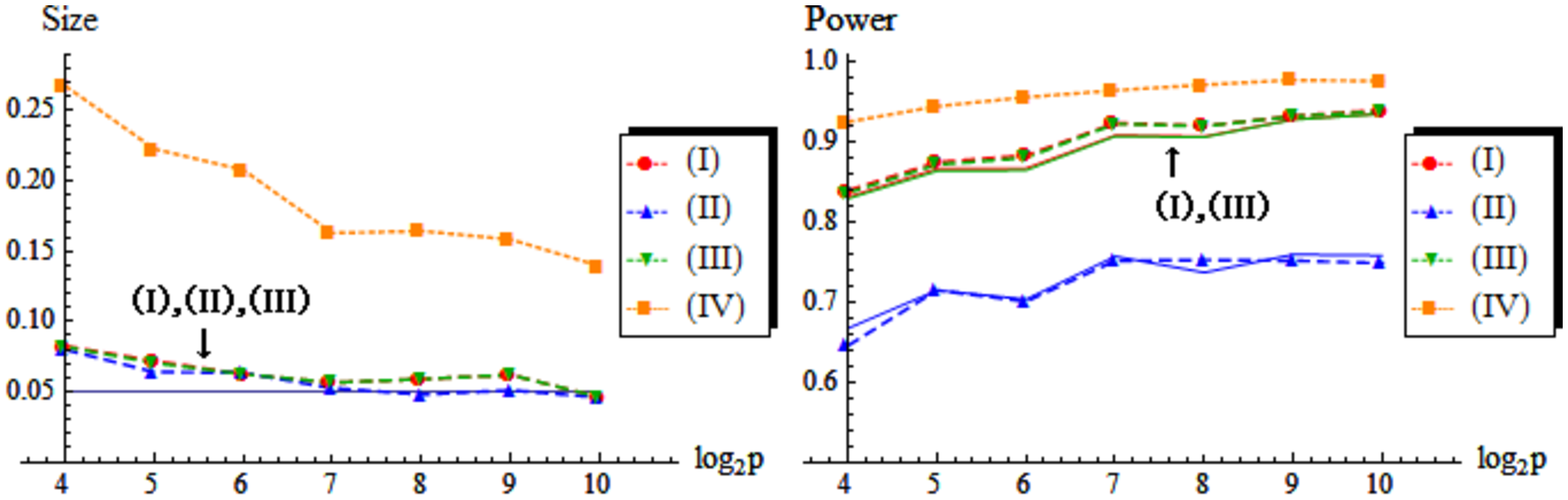}  \\[-0.5mm]
(b) SN$_p(\bOme, \bal)$ with $\bal=4\bone_p$.
\\[4mm]
\includegraphics[scale=0.48]{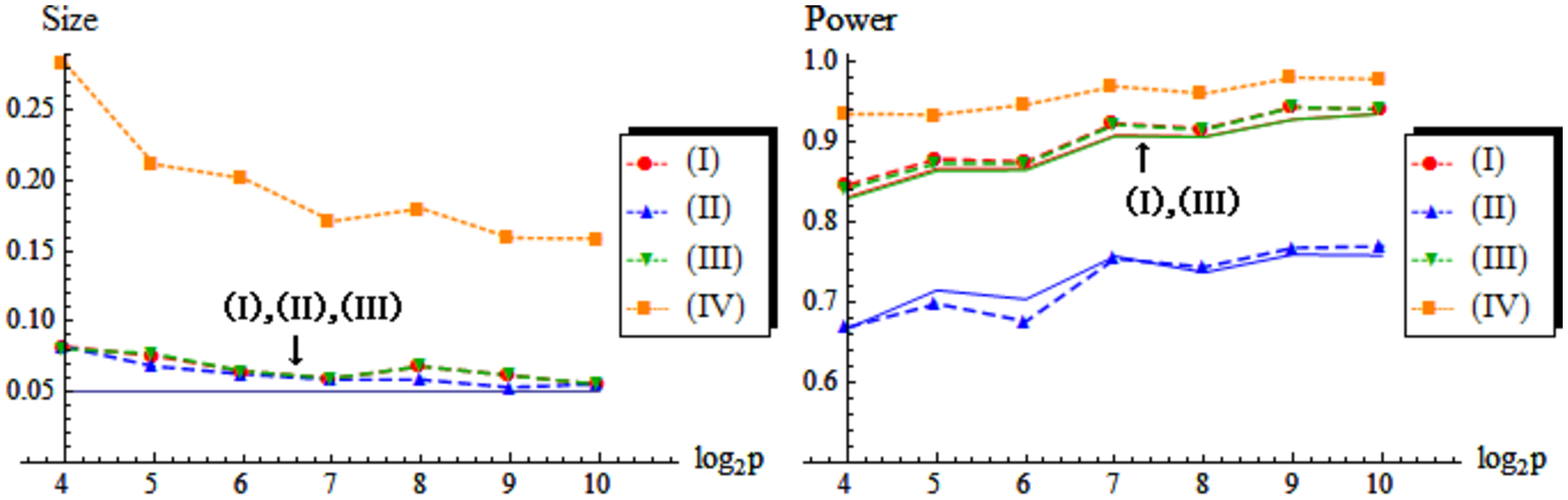}  \\[-0.5mm]
(c) SN$_p(\bOme, \bal)$ with $\bal=16\bone_p$.
\\[4mm]
{\small 
Figure S4.1: Test procedures by (3.1) when (I) $\bA=\bI_p$, (II) $\bA=\bA_{\star}$, (III) $\bA=\bA_{\star(d)}$ and (IV) $\bA=\widehat{\bA}_{\star(d)}$ for $p=2^{s},\ s=4,...,10$, $n_1=\lceil p^{1/2} \rceil$ and $n_2=2n_1$. 
For (a) to (c), the values of $\overline{\alpha}$ are denoted by the dashed lines in the left panel and the values of $1-\overline{\beta}$ are denoted by the dashed lines in the right panel.
The asymptotic powers were given by $\Phi(\Delta(\bA)/\{K(\bA)\}^{1/2}-z_{\alpha}\{K_1(\bA)/K(\bA)\}^{1/2})$ for (I) to (III) which are denoted by the solid lines in the right panels. 
}
\end{figure}

We observed that the plots become close to the theoretical value even for the skewed distributions. 
The tests with (I) and (III) gave similar performances for (a) to (c). 
This is probably because $\sigma_{i(j)}\to c_i$ as $p\to \infty$ for all $i,j$ in those settings.
Similar to Fig. 1, the test with (I) gave better performances compared to (II) for (a) to (c). 
See Sections 3.2 and 3.3 for the details.
\\[3mm]
\noindent
{\bf Remark S4.1.} \ Let $\bb_{1}= \bOme^{1/2}\bal/||\bOme^{1/2}\bal||$ and $\bb_{2},...,\bb_{p}$ be $p$-dimensional vectors such that
$||\bb_{s}||=1$, $\bb_{1}^T\bb_{s}=0$ for $s=2,...,p$, and $\sum_{s=1}^p\bb_s\bb_s^T=\bI_p$. 
Then, from Propositions 3 and 6 in \cite{Azzalini:1999}, $\bb_{1}^T\bOme^{-1/2}\breve{\bx}_{ij},...,\bb_{p}^T\bOme^{-1/2}\breve{\bx}_{ij}$ are independent. 
Hence, (A-i) is met from the fact that $\bx_{ij}-\bmu_i=c_i^{1/2}\sum_{s=1}^p \bOme^{1/2}\bb_{s}\{\bb_{s}^T\bOme^{-1/2}(\breve{\bx}_{ij}-\breve{\bmu})\}$. 
\subsection{Simulations for SSE Model}
\noindent

In this section, we give additional simulations for Section 6 under the SSE model. 

We set $\alpha=0.05$, $\bmu_1=\bze$ and 
\begin{equation}
\bSig_{i}=\left( \begin{array}{cc}
\bSig_{(1)} & \bO_{2, p-2} \\
\bO_{p-2, 2} & \bSig_{i(2)}
\end{array} \right) \quad \mbox{with $\bSig_{(1)}=\mbox{diag}(p^{2/3},p^{1/2})$ }
\label{sig1}
\end{equation}
for $i=1,2$. 
When considering the alternative hypothesis, we set $\bmu_2=(0,...,0,1,1,1,1)^T$ whose last $4$ elements are $1$. 
We set $\kappa(n_i)=(n_i^{-1} \log{n_i})^{1/2}$ in (\ref{s2.2}). 
We checked the performance of five tests: (I) from (3.1) with $\bA=\bI_p$, (II) from (4.2), (III) from (5.5), (IV) from (5.5) with $k_i=\hat{k}_i$, $i=1,2$, and (V) from (6.1). 
Let us write that $\bx_{ij}=(x_{i1(j)},...,x_{ip(j)})^T$,  $\bmu_i=(\mu_{i1},...,\mu_{ip})^T$,
$\bx_{ij(2)}=(x_{i3(j)},...,x_{ip(j)})^T$ and $\bmu_{i(2)}=(\mu_{i3},...,\mu_{ip})^T$ for all $i,j$. 
We supposed that $(x_{i1(j)},x_{i2(j)})^T$s are i.i.d. as $N_2(\bze,\bSig_{(1)})$.

First, we checked the performance of the test procedures for the MSN distribution. 
We set $p=2^s,\ n_1=3 \lceil p^{1/2} \rceil$ and  $n_2=4 \lceil p^{1/2} \rceil$ for $s=4,...,10$.
We generated $\breve{\bx}_{ij(2)}$, $j=1,2,...,\ (i=1,2)$ independently from 
SN$_{p-2}(\bOme_i, \bal)$ with $\bOme_1=(0.3^{|i-j|^{1/2}})$ and $\bOme_2=(0.5^{|i-j|^{1/2}})$, 
where $(x_{i1(j)},x_{i2(j)})^T$ and $\breve{\bx}_{ij(2)}$ are independent for each $j$.
We considered two cases: (a) $\bal=4\bone_{p-2}$; and (b) $\bal=16\bone_{p-2}$. 
Similar to Section S4.1, 
we set $\bx_{ij(2)}=\breve{\bx}_{ij(2)}-\breve{\bmu}_i+\bmu_{i(2)}$ for all $i,j$, where 
$\breve{\bmu}_i=E(\breve{\bx}_{ij(2)})=(2/\pi)^{1/2}\bOme_i\bal/(1+\bal^T\bOme_i\bal)^{1/2}$, $i=1,2$. 
Then, we had $\bSig_{i(2)}=\bOme_i-\breve{\bmu}_i\breve{\bmu}_i^T$, $i=1,2,$ in (\ref{sig1}). 
Note that (4.1) and (A-vi) with $k_1=k_2=2$ are met. 
Similar to Remark S4.1, we note that (A-i) is met. 
However, (A-viii) is not met. 
Similar to Section 6, we calculated $\overline{\alpha}$ and $1-\overline{\beta}$ with 2000 replications for 
the five test procedures. 
In Fig. S4.2, for (a) and (b), we plotted $\overline{\alpha}$ in the left panel and $1-\overline{\beta}$ in the right panel. 
We observed the performances similar to those in Fig. 2 (a). 
\begin{figure}[h!]
\includegraphics[scale=0.63]{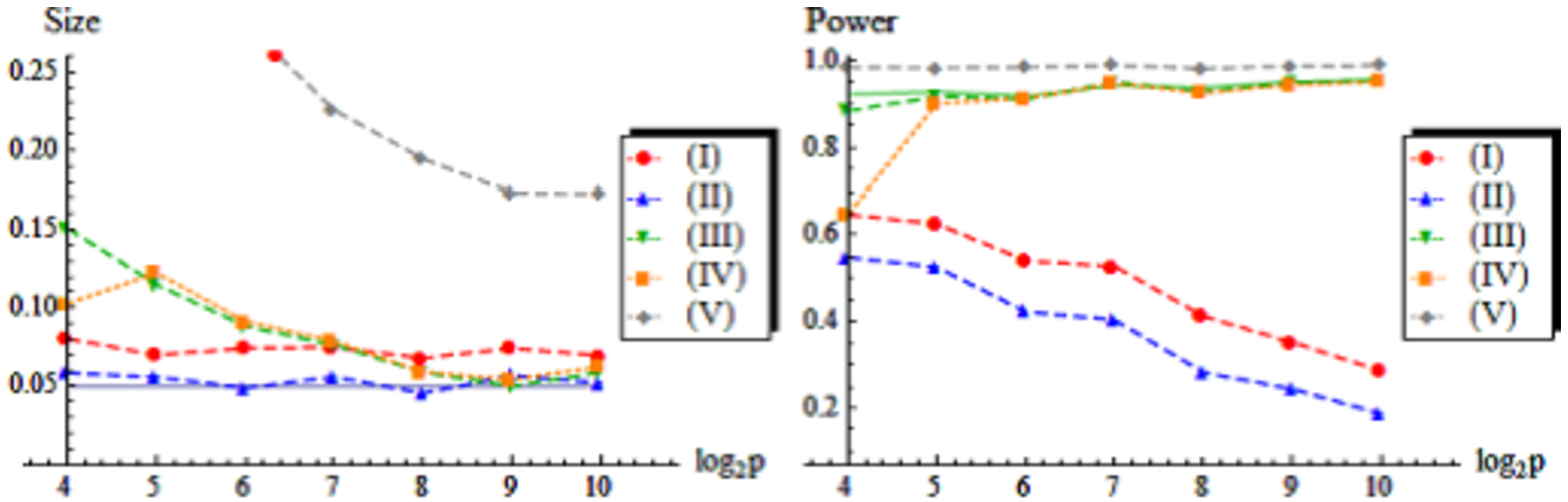} \\[-0.5mm]
(a) SN$_{p-2}(\bOme_i, \bal)$ with $\bal=4\bone_{p-2}$.
 \\[4mm]
\includegraphics[scale=0.63]{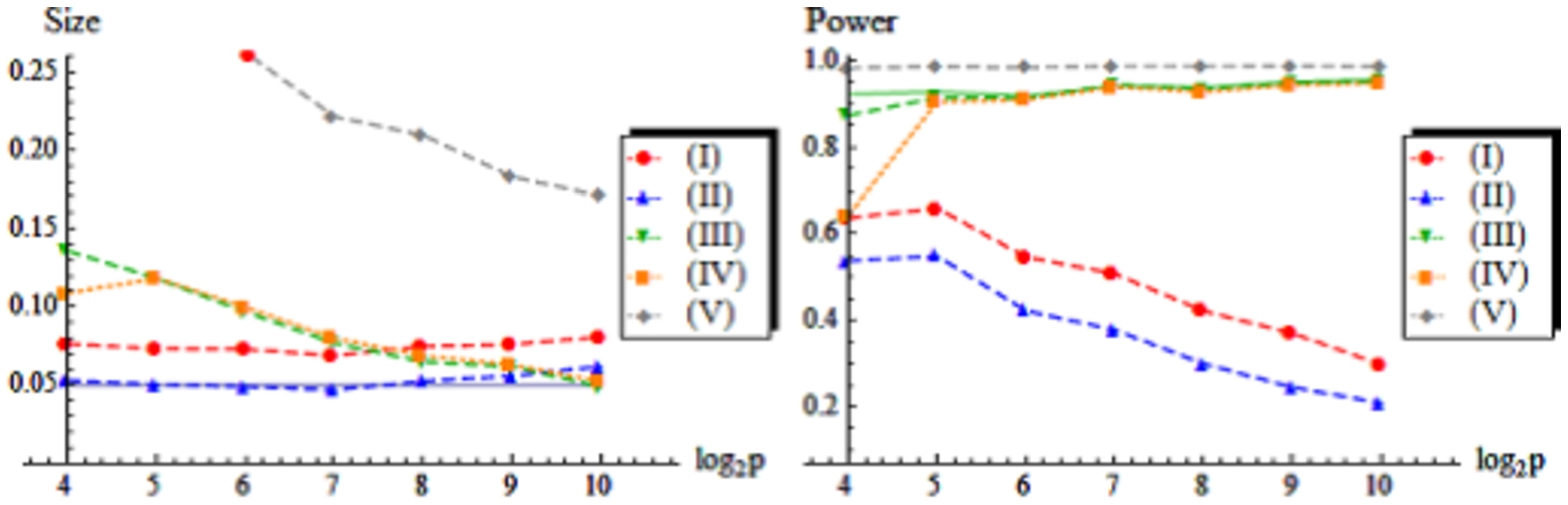}  \\[-0.5mm]
(b) SN$_{p-2}(\bOme_i, \bal)$ with $\bal=16\bone_{p-2}$.
\\[4mm]
{\small 
Figure S4.2: When $p=2^s,\ n_1=3 \lceil p^{1/2} \rceil$ and  $n_2=4 \lceil p^{1/2} \rceil$ for $s=4,...,10$, 
the performances of five tests: (I) from (3.1) with $\bA=\bI_p$, (II) from (4.2), (III) from (5.5), (IV) from (5.5) with $k_i=\hat{k}_i$, $i=1,2$, and (V) from (6.1). 
For (a) and (b), the values of $\overline{\alpha}$ are denoted by the dashed lines in the left panel and the values of $1-\overline{\beta}$ are denoted by the dashed lines in the right panel.
The asymptotic power of (III) was given by $\Phi(\Delta_*/K_*^{1/2}-z_{\alpha}(K_{1*}/K_*)^{1/2})$ which is denoted by the solid line in the right panels. 
When $p$ is small, $\overline{\alpha}$ for (V) was too high to describe in the left panels.
}
\end{figure}

Next, we checked the performance of the test procedures for the multivariate skew $t$ (MST) distribution. 
See \cite{Azzalini:2003} and \cite{Gupta:2003} for the details of the MST distribution.
We considered two cases: (i) $(n_1,n_2)=(40,60)$ and $p=50+100(s-1)$ for $s=1,...,7$; 
and (ii) $p=500$, $n_1=10s$ and $n_2=1.5n_1$ for $s=2,...,8$. 
We generated $\check{\bx}_{ij(2)}$, $j=1,2,...,\ (i=1,2)$ independently from a MST distribution, 
ST$_{p-2}(\bOme_i,\bal,\nu)$, with correlation matrix $\bOme_i$, shape parameter vector $\bal$ and 
degrees of freedom $\nu$, 
where $(x_{i1(j)},x_{i2(j)})^T$ and $\check{\bx}_{ij(2)}$ are independent for each $j$. 
We set $\bOme_1=(0.3^{|i-j|^{1/2}})$, $\bOme_2=(0.5^{|i-j|^{1/2}})$ and $\bal=10\bone_{p-2}$. 
We considered two cases: (a) $\nu=10$ and (b) $\nu=20$. 
Note that $E(\check{\bx}_{ij(2)})=(\nu/\pi)^{1/2} \{\Gamma(\nu/2-1/2)/\Gamma(\nu/2)\}\bOme_i\bal/(1+\bal^T\bOme_i\bal)^{1/2}\ (=\check{\bmu}_i$, say) and 
Var$(\check{\bx}_{ij(2)})=\nu \bOme_i/(\nu-2)-\check{\bmu}_i\check{\bmu}_i^T\ (=\check{\bSig}_i$, say), 
where $\Gamma(\cdot)$ denotes the gamma function.
We set $\bx_{ij(2)}=\check{\bx}_{ij(2)}-\check{\bmu}_i+\bmu_{i(2)}$ for all $i,j$. 
Then, we had $\bSig_{i(2)}=\check{\bSig}_i$, $i=1,2,$ in (\ref{sig1}). 
Note that (4.1) and (A-vi) with $k_1=k_2=2$ are met. 
However, (A-i) and (A-viii) are not met. 
Similar to Fig. S4.2, 
we plotted $\overline{\alpha}$ in the left panel and $1-\overline{\beta}$ in the right panel for (i) in Fig. S4.3 and for (ii) in Fig. S4.4. 
We observed the performances similar to those in Fig. 2 (b) and (c). 
\begin{figure}[h!]
\includegraphics[scale=0.63]{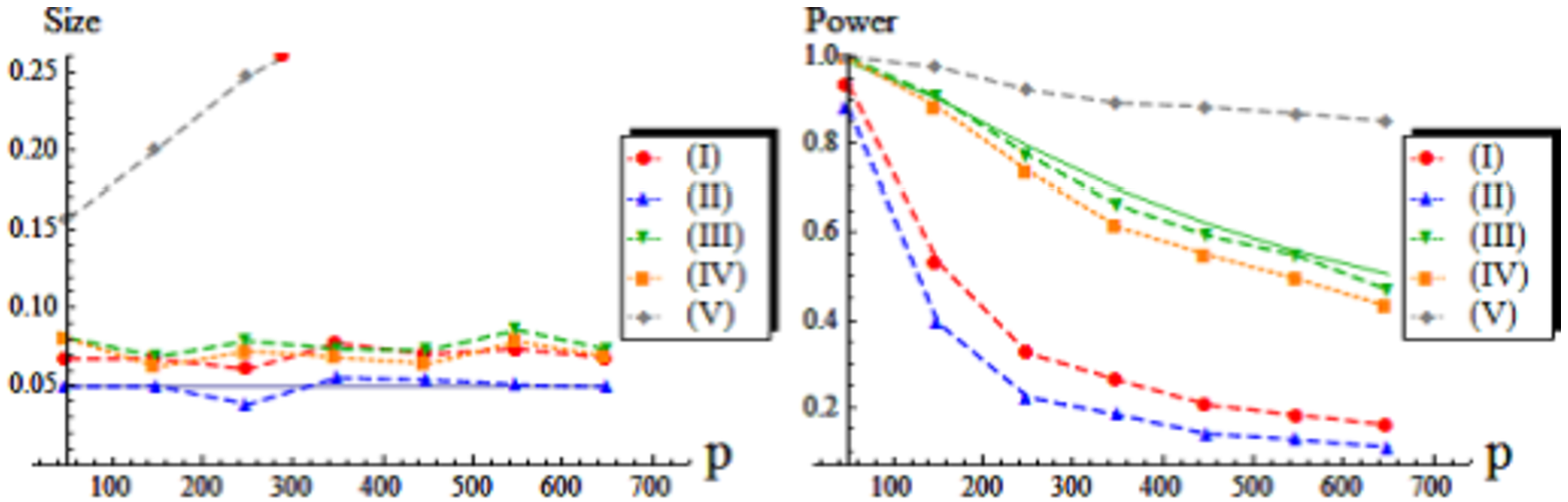} \\[-0.5mm] 
(a) ST$_{p-2}(\bOme_i,\bal,\nu)$ with $\nu=10$.
 \\[4mm]
\includegraphics[scale=0.63]{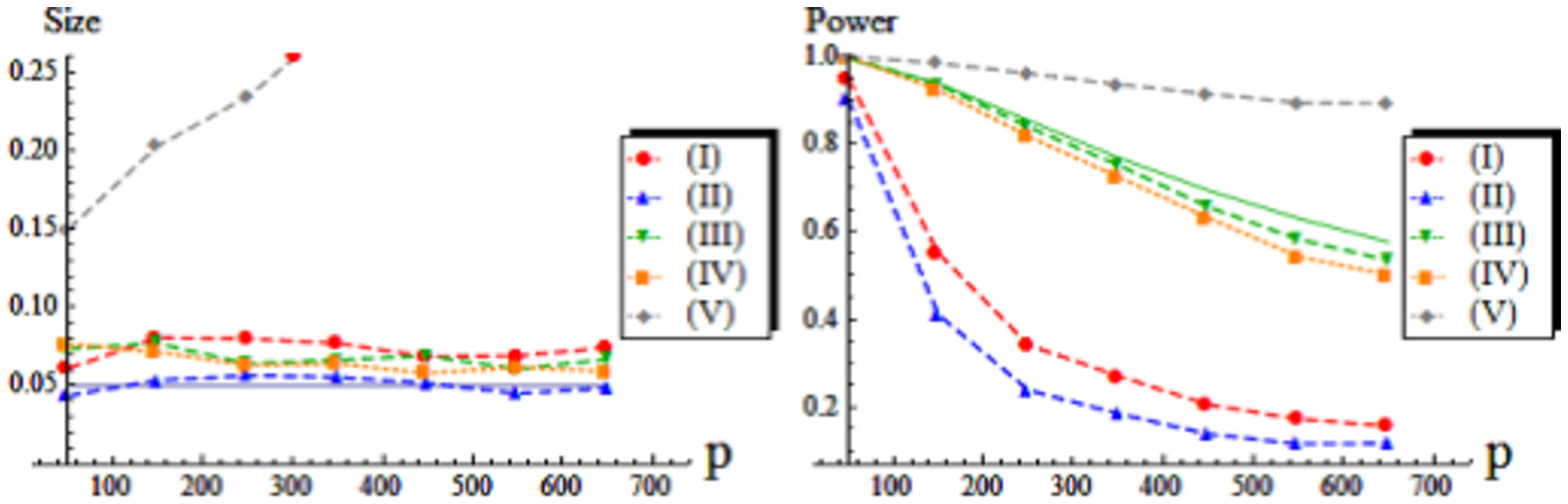}  \\[-0.5mm]
(b) ST$_{p-2}(\bOme_i,\bal,\nu)$ with $\nu=20$.
\\[4mm]
{\small 
Figure S4.3: When (i) $(n_1,n_2)=(40,60)$ and $p=50+100(s-1)$ for $s=1,...,7$, the performances of five tests: (I) from (3.1) with $\bA=\bI_p$, (II) from (4.2), (III) from (5.5), (IV) from (5.5) with $k_i=\hat{k}_i$, $i=1,2$, and (V) from (6.1). 
}
\end{figure}
\begin{figure}[h!]
\includegraphics[scale=0.63]{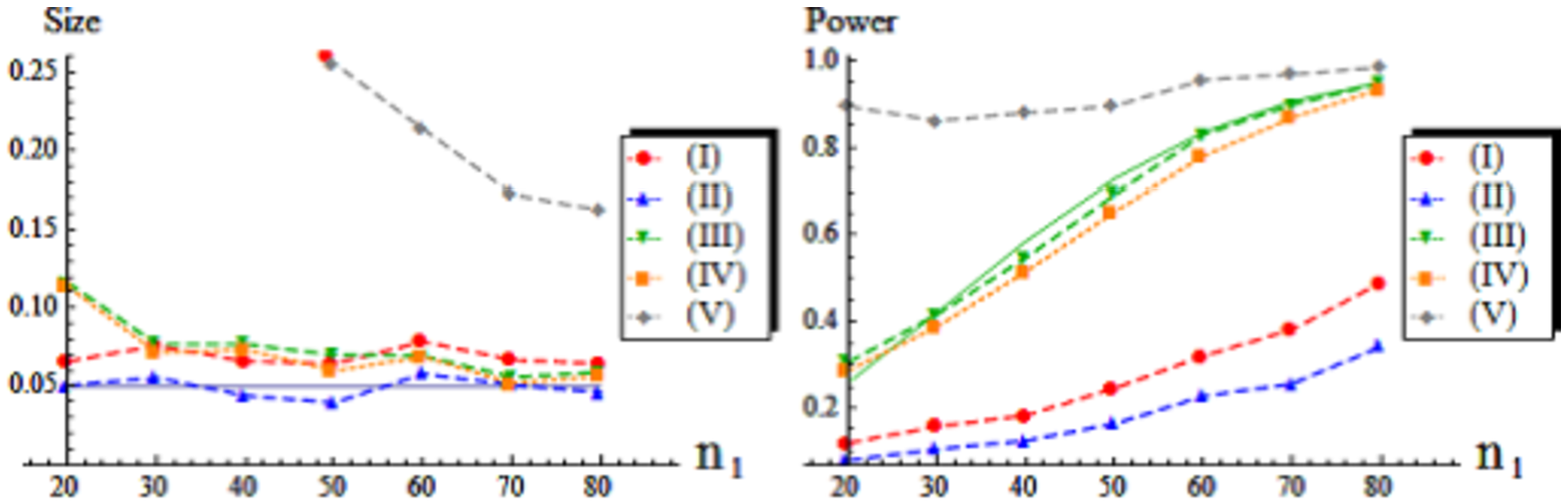} \\[-0.5mm] 
(a) ST$_{p-2}(\bOme_i,\bal,\nu)$ with $\nu=10$.
 \\[4mm]
\includegraphics[scale=0.63]{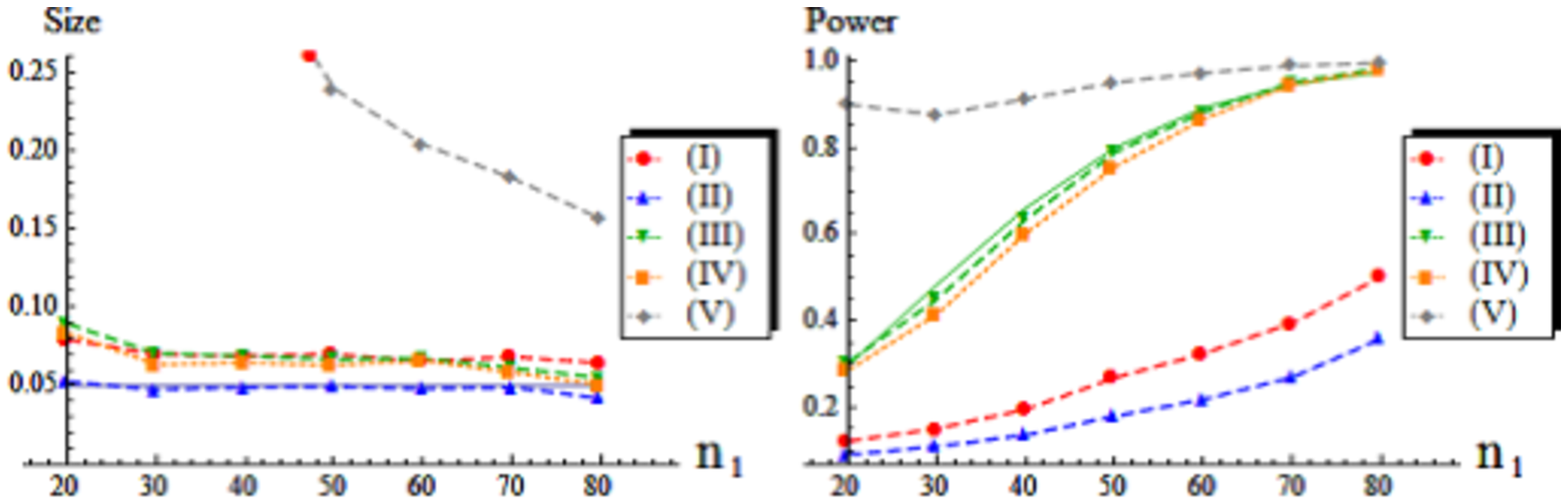}  \\[-0.5mm]
(b) ST$_{p-2}(\bOme_i,\bal,\nu)$ with $\nu=20$.
\\[4mm]
{\small 
Figure S4.4: When (ii) $p=500$, $n_1=10s$ and $n_2=1.5n_1$ for $s=2,...,8$, the performances of five tests: (I) from (3.1) with $\bA=\bI_p$, (II) from (4.2), (III) from (5.5), (IV) from (5.5) with $k_i=\hat{k}_i$, $i=1,2$, and (V) from (6.1).  
}
\end{figure}

Throughout, the test procedure by (5.5) with $k_i=\hat{k}_i$, $i=1,2$, gave adequate performances for high-dimensional 
cases even for the skewed and heavy tailed distributions.

\section{Appendix A}
\setcounter{equation}{0}
\noindent

In this appendix, we give proofs of the theoretical results in Sections 2 and 3. 

We simply write $T=T(\bA)$, $\Delta=\Delta(\bA)$, $K=K(\bA)$, $K_1=K_1(\bA)$, $\widehat{K}_1=\widehat{K}_1(\bA)$ and $K_2=K_2(\bA)$. 
\begin{proof}[Proof of Theorem 1]
We note that for $i=1,2$
\begin{equation}
\bmu_{\bsA}^T\bSig_{i,\bsA}\bmu_{\bsA}\le \Delta \lambda_{\max}(\bSig_{i,\bsA})\le \Delta \tr(\bSig_{i,\bsA}^2)^{1/2}.
\label{A.1}
\end{equation}
Hence, 
from the fact that $\tr(\bSig_{i,\bsA}^2)/n_i^2\le K_1$ for $i=1,2$, 
it holds that $K_2=O(\Delta K_1^{1/2})$, so that 
\begin{equation}
\Var(T/\Delta) = (K_1+K_2)/\Delta^2=K_1/\Delta^2
+O(K_1^{1/2}/\Delta).
\label{A.2}
\end{equation}
Thus, under (A-iii), from Chebyshev's inequality, we can claim the result. 
\end{proof}
\begin{proof}[Proof of Theorem 2]
We first consider the case when (A-iv) is met. 
From (\ref{A.1}), under (A-ii), it holds that 
$\bmu_{\bsA}^T\bSig_{i,\bsA}\bmu_{\bsA}/n_i=o(\Delta\tr(\bSig_{i,\bsA}^2)^{1/2}/n_i)=o(\Delta K_1^{1/2})$
as $m\to \infty$, so that 
\begin{equation}
K_2/K_1=O\{ K_2/(\Delta K_1^{1/2})\}\to 0
\label{A.3}
\end{equation}
under (A-ii) and (A-iv). 
Let ${\bx}_{ij,\bsA}=\bA^{1/2}\bx_{ij}\ (j=1,...,n_i)$, $\bmu_{i,\bsA}=\bA^{1/2}\bmu_i$ and $\bGamma_{i,\bsA}=\bA^{1/2}\bGamma_i$ for $i=1,2$. 
We write that 
\begin{equation}
{\bx}_{ij,\bsA}=\bGamma_{i,\bsA}\bw_{ij}+\bmu_{i,\bsA}\quad \mbox{for all $i,j$.}
\label{A.4}
\end{equation} 
Note that $\Var({\bx}_{ij,\bsA})=\bSig_{i,\bsA}$ for $i=1,2$. 
Then, from (\ref{A.3}), by using Theorem 5 given in \citet{Aoshima:2015}, we can obtain the result when (A-iv) is met. 

Next, we consider the case when (A-v) is met. 
Let $\bmu_{12}=\bmu_1-\bmu_2$.
Under (A-v), it holds that 
\begin{equation}
T-\Delta=2\bmu_{12}^T \bA(\overline{\bx}_{1n_1}-\overline{\bx}_{2n_2}-\bmu_{12})+o_P(K_2^{1/2})
\label{A.5}
\end{equation}
from the fact that 
$
\Var\{(\overline{\bx}_{1n_1}-\overline{\bx}_{2n_2}-\bmu_{12})^T\bA(\overline{\bx}_{1n_1}-\overline{\bx}_{2n_2}-\bmu_{12})-\tr(\bS_{1n_1}\bA)/n_1-\tr(\bS_{2n_2}\bA)/n_2\}=K_1.
$
Let $\omega_j=2\bmu_{12}^T \bA(\bx_{1j}-\bmu_1)/n_1$ for $j=1,...,n_1$, and $\omega_{j+n_1}=-2\bmu_{12}^T \bA(\bx_{2j}-\bmu_2)/n_2$ for $j=1,...,n_2$. 
Note that $\sum_{j=1}^{n_1+n_2}\omega_j=2\bmu_{12}^T \bA(\overline{\bx}_{1n_1}-\overline{\bx}_{2n_2}-\bmu_{12})$ and $\Var(\sum_{j=1}^{n_1+n_2}\omega_j)=K_2$. 
Note that $E(w_j^4)=O\{(\bmu_{\bsA}^T \bSig_{1,\bsA}\bmu_{\bsA})^2/n_1^4\}$ for $j=1,...,n_1$, and $E(w_j^4)=O\{(\bmu_{\bsA}^T \bSig_{2,\bsA}\bmu_{\bsA})^2/n_2^4\}$ for $j=n_1+1,...,n_1+n_2$, under (A-i).
Then, for Lyapunov's condition, it holds that as $n_{\min}\to \infty$
$$
\frac{\sum_{j=1}^{n_1+n_2}E(w_j^4)}{K_2^2}=
\frac{O\{(\bmu_{\bsA}^T \bSig_{1,\bsA}\bmu_{\bsA})^2/n_1^3+(\bmu_{\bsA}^T \bSig_{2,\bsA}\bmu_{\bsA})^2/n_2^3\} }{K_2^2}
=O(n_{\min}^{-1})\to 0.
$$
Hence, by using Lyapunov's central limit theorem, we have that $\sum_{j=1}^{n_1+n_2}\omega_j/K_2^{1/2} \Rightarrow N(0,1)$.
In view of (\ref{A.5}) and $K_2/K=1+o(1)$ as $m\to \infty$ under (A-v), we can obtain the result when (A-v) is met.
\end{proof}
\begin{proof}[Proof of Proposition 1]
From (\ref{A.1}) and the fact that $\tr(\bSig_{i,\bsA}^2)/n_i^2\le K_1,\ i=1,2$, it holds that $K_1/K_2 \ge K_1^{1/2}/(8\Delta)$.
Thus, (A-v) implies (A-iii). 
It concludes the result. 
\end{proof}
\begin{proof}[Proof of Lemma 1]
From (\ref{A.3}), the result is obtained straightforwardly. 
\end{proof}
\begin{proof}[Proofs of Lemma 2 and Corollary 1]
From (2.3), (\ref{A.4}) and the equation (23) given in \citet{Aoshima:2015}, we have that $\widehat{K}_1/K_1=1+o_P(1)$ as $m\to \infty$ under (A-i). 
It concludes the result of Lemma 2.
By using Lemmas 1 and 2, it holds that $\widehat{K}_1/K=1+o_P(1)$ under (A-i), so that the result of Corollary 1 is obtained from Theorem 2. 
\end{proof}
\begin{proof}[Proofs of Theorem 3 and Corollary 2]
First, we consider Corollary 2. 
From Theorem 1, under (A-i) and (A-iii), we have that as $m\to \infty$ 
$$
P( {T}/{\widehat{K}_1^{1/2}}>z_{\alpha})
=P( {T}/{\Delta}>z_{\alpha}{\widehat{K}_1^{1/2}}/{\Delta})=P\{1+o_P(1)>o_P(1)\}
\to 1
$$
from the fact that $\widehat{K}_1^{1/2}/\Delta={K}_1^{1/2}\{1+o_P(1)\}/\Delta=o_P(1)$ under (A-i) and (A-iii). 
It concludes the result of Corollary 2 when (A-iii) is met.
From Theorem 2, Lemmas 1 and 2, under (A-i), (A-ii) and (A-iv), we have that
\begin{align}
P( {T}/{\widehat{K}_1^{1/2}} >z_{\alpha})
&=P\{ (T-\Delta)/K^{1/2} > (z_{\alpha}{K}_1^{1/2}-\Delta)/K^{1/2}+o_P(1)\}\label{A.6}  \\
&=\Phi\{ (\Delta-z_{\alpha}K_1^{1/2})/K^{1/2}\}+o(1) =\Phi({\Delta}/{K_1^{1/2}}-z_{\alpha})+o(1) \notag. 
\end{align}
It concludes the result of Corollary 2 when (A-ii) and (A-iv) are met.
We note that $K/K_2\to 1$ as $m\to \infty$ under (A-v). 
Then, by combining (\ref{A.6}) and Theorem 2, we can conclude the result of Corollary 2 when (A-v) is met.

Next, we consider Theorem 3. 
By combining (\ref{A.6}) and Theorem 2, we can conclude the results about size and power in Theorem 3 when (A-iv) is met. 
From (\ref{A.2}) we note that $K/\Delta^2\to 0$ under (A-iii). 
It holds that $\Phi\{(\Delta-z_{\alpha}K_1^{1/2})/K^{1/2} \}\to 1$ under (A-iii), so that from Corollary 2 we obtain the result about power when (A-iii) is met.
Hence, by considering a convergent subsequence of $\Delta/K_1^{1/2}$, we can conclude the result about power in Theorem 3. 
\end{proof}

\section{Appendix B}
\setcounter{equation}{0}
\noindent

In this appendix, we give proofs of the theoretical results in Sections 4 and 5. Also, we give two lemmas and proofs of the lemmas.

Let $\bar{z}_{ij}=\sum_{l=1}^{n_i}z_{ijl}/n_i$ and $v_{i(j)}=\sum_{l=1}^{n_i}(z_{ijl}-\bar{z}_{ij})^2/(n_i-1)$ for all $i,j$. 
Let $\bu_{ij}=(z_{ij1},...,z_{ijn_i})^T/(n_i-1)^{1/2}$, ${\bu}_{oij}=\bP_{n_i}{\bu}_{ij}=(z_{ij1}-\bar{z}_{ij},...,z_{ijn_i}-\bar{z}_{ij})^T/(n_i-1)^{1/2}$ and $\dot{\bu}_{ij}={||\bu_{ij}||}^{-1}{{\bu}_{ij}}$ for all $i,j$. 
Let $\bzeta_i$ be an arbitrary unit random $n_i$-dimensional vector for $i=1,2$. 
Let $\by_{ij}=\sum_{s=1}^{k_i}\lambda_{is}^{1/2}\bh_{is}z_{isj}$ and $\bv_{ij}=\sum_{s=k_i+1}^{p}\lambda_{is}^{1/2}\bh_{is}z_{isj}$ for all $i,j$. 
Note that $\bx_{ij}=\by_{ij}+\bv_{ij}+\bmu_i$ for all $i,j$. 
Let $\psi_{ij}=\tr(\bSig_i^2)/\lambda_{ij}+n_i \bmu_i^T\bSig_i\bmu_i/\lambda_{ij}$ for $i=1,2;\ j=1,...,k_i$.
Let $h_{st}=\bh_{1s}^T\bh_{2t}$ for all $s,t$. 
We also let $\bM_i=\bmu_i\bone_{n_i}^T$ for $i=1,2$. 
\begin{proof}[Proof of Theorem 4]
We assume $\bmu_1=\bmu_2=\bze$ and $\bh_{11}^T\bh_{21}\ge 0$ without loss of generality. 
Let $\bH_{i1}=\bh_{i1}\bh_{i1}^T$, $\bH_{i2}=\bI_p-\bH_{i1}$, $\bSig_{i1}=\lambda_{i1}\bH_{i1}$ and $\bSig_{i2}=\sum_{j=2}^p\lambda_{ij} \bh_{ij}\bh_{ij}^T$ for $i=1,2$. 
Note that $\bSig_i=\bSig_{i1}+\bSig_{i2}$ for $i=1,2$. 
We write that 
$$
T_I=T(\bH_{11},\bH_{21})+T(\bH_{12},\bH_{22})
-2\overline{\bx}_{1n_1}^T(\bH_{11}\bH_{22}+\bH_{12}\bH_{21})\overline{\bx}_{2n_2}.
$$
We have that $\Var\{T(\bH_{11},\bH_{21})\}=K_1(\bH_{11},\bH_{21})=2 \sum_{i=1}^2\lambda_{i1}^2/\{n_i(n_i-1)\}+4\lambda_{11}\lambda_{21}(\bh_{11}^T\bh_{21})^2$
$/(n_1n_2)$ and 
$\Var\{T(\bH_{12},\bH_{22})\}=K_1(\bH_{12},\bH_{22})=2 \sum_{i=1}^2\tr(\bSig_{i2}^2)/\{n_i(n_i-1)\}+4\tr(\bSig_{12}\bSig_{22})$
$/(n_1n_2)$, 
where $K_1(\cdot\, ,\cdot)$ is defined in Section S1.2.
Let $\psi=({\lambda_{11}}/{n_1}+{\lambda_{21}}/{n_2})$. 
Then, under (4.1) it holds that as $m\to \infty$
$$
K_1(\bH_{11},\bH_{21})=2\psi^2\{1+o(1)\}\quad \mbox{and}\quad K_1(\bH_{12},\bH_{22})=o(\psi^2)
$$
because $\tr(\bSig_{12}\bSig_{22})\le\{\tr(\bSig_{12}^2)\tr(\bSig_{22}^2)\}^{1/2}=(\Psi_{1(2)}\Psi_{2(2)})^{1/2}$. 
Also, under (4.1) it follows that 
$$
\Var\{\overline{\bx}_{1n_1}^T(\bH_{11}\bH_{22}+\bH_{12}\bH_{21})\overline{\bx}_{2n_2}\}
=\frac{\tr(\bSig_{11}\bSig_{22})+\tr(\bSig_{12}\bSig_{21})}{n_1n_2}=o(\psi^2)
$$
because $\tr(\bSig_{11}\bSig_{22})\le \lambda_{11}\tr(\bSig_{22}^2)^{1/2}$ and $\tr(\bSig_{12}\bSig_{21})\le \lambda_{21}\tr(\bSig_{12}^2)^{1/2}$. 
Hence, under (4.1) we have that $K_{1(I)}=2\psi^2\{1+o(1)\}$ and 
\begin{align*}
T_I&=\sum_{i=1}^2\lambda_{i1}(\bar{z}_{i1}^2-v_{i(1)}/n_i)-2(\lambda_{11}\lambda_{21})^{1/2}\bar{z}_{11}\bar{z}_{21}(\bh_{11}^T\bh_{21})
+o_P(\psi)\notag \\
&=( \lambda_{11}^{1/2}\bar{z}_{11}-\lambda_{21}^{1/2}\bar{z}_{21})^2-\psi+o_P(\psi)
\end{align*}
from the fact that $v_{i(1)}=1+o_P(1)$, $i=1,2$. 
By noting that $E(z_{i1l}^4)$'s are bounded, for Lyapunov's condition, it holds that $\sum_{i=1}^2\sum_{l=1}^{n_i}(\lambda_{i1}^{1/2}z_{i1l}/n_i)^4=o(\psi^2)$. 
Hence, by using Lyapunov's central limit theorem, we have that $\psi^{-1/2}( \lambda_{11}^{1/2}\bar{z}_{11}-\lambda_{21}^{1/2}\bar{z}_{21})\Rightarrow N(0,1)$. 
Thus, from $\psi^{-1}T_I=\psi^{-1}( \lambda_{11}^{1/2}\bar{z}_{11}-\lambda_{21}^{1/2}\bar{z}_{21})^2-1+o_P(1)$ and $K_{1(I)}=2\psi^2\{1+o(1)\}$ under (4.1), we have that
$
T_I/(K_{1(I)}/2)^{1/2}+1 \Rightarrow  \chi_1^2.
$
From Lemma 2, it concludes the result. 
\end{proof}
\begin{proof}[Proof of Corollary 3]
From Theorem 2, the result is obtained straightforwardly.
\end{proof}

Throughout the proofs of Propositions 2 to 5, Lemmas B.1, B.2, 3 and Theorem 5, we assume (A-vi) and (A-viii). 
Throughout the proofs of Propositions 2 to 5 and Lemma B.1, we omit the subscript with regard to the population.
\begin{proof}[Proof of Proposition 2]
Let us write that $\bU_{1}=\sum_{s=1}^{k}\lambda_{s}\bu_{os}\bu_{os}^T$ and $\bU_{2}=\sum_{s=k+1}^p\lambda_{s}\bu_{s}\bu_{s}^T$. 
Note that $\bS_{D}=\bU_{1}+\bP_{n}\bU_{2}\bP_{n}$. 
Also, note that $\bP_{n}\hat{\bu}_{j}=\hat{\bu}_{j}$ and $\hat{\lambda}_{j}=\hat{\bu}_{j}^T\bS_{D} \hat{\bu}_{j}=\hat{\bu}_{j}^T(\bU_{1}+\bU_{2})\hat{\bu}_{j}$ when $\hat{\lambda}_{j}>0$. 
From Lemma 5 in \cite{Yata:2013b} we can claim that as $m_0 \to \infty$
\begin{equation}
\hat{\lambda}_{j}/\lambda_{j}-\delta_j=(\hat{\bu}_{j}^T\bU_{1}\hat{\bu}_{j})/\lambda_{j}+o_P(1)
\ \ \mbox{for $j=1,...,k$}. \notag
\end{equation}
Also, similar to the proofs of Lemmas 3 and 4 in \cite{Yata:2012}, we have that 
$\bu_j^T (\bU_{2}- \delta \bI_n)\bu_{j'}=O_P(\Psi_{(k+1)}^{1/2}/n)$ and 
$\bu_j^T (\bU_{2}-\delta \bI_n)\bzeta=O_P(\Psi_{(k+1)}^{1/2}/n^{1/2})$ for $j,j'=1,...,k$, 
where $\delta=\sum_{s=k+1}^p\lambda_s/(n-1)$.
Then, by noting that $\bu_{oj}^T\bu_{oj'}=O_P(n^{-1/2})\ (j\neq j')$ and $||\bu_{oj}||^2=||\bu_{j}||^2+O_P(n^{-1})=1+O_P(n^{-1/2})$ as $n\to \infty$, we can claim that 
\begin{align}
&\hat{\lambda}_{j}/\lambda_{j}=||\bu_{j}||^2+\delta_j+
O_P(n^{-1})=1+\delta_j+O_P(n^{-1/2})\notag \\
&\mbox{and} \quad \hat{\bu}_{j}^T\dot{\bu}_{j}=1+O_P(n^{-1})
\quad \mbox{for $j=1,...,k$;}
 \label{B.1} \\
&\hat{\bu}_{j'}^T{\bu}_{j}= O_P(n^{-1/2}\lambda_{j'}/\lambda_{j})
\quad \mbox{for $j<j'\le k$} \label{B.2}
\end{align}
in a way similar to the proof of Lemma 5 in \cite{Yata:2012} and the proof of Lemma 9 in \cite{Yata:2013b}.
By noting that $(\bX-\overline{\bX}) \hat{\bu}_{j}=(\bX-\bM)\hat{\bu}_{j}$ when $\hat{\lambda}_{j}>0$, we write that 
$$
({\bh}_{j}^T\hat{\bh}_{j})^2=
\{{\bh}_{j}^T(\bX-\overline{\bX}) \hat{\bu}_{j}\}^2/\{(n-1)\hat{\lambda}_{j}\}
=||\bu_{j}||^2(\hat{\bu}_{j}^T\dot{\bu}_{j})^2(\lambda_{j}/\hat{\lambda}_{j})
$$
when $\hat{\lambda}_{j}>0$. 
Thus, from (\ref{B.1}) we can conclude the results. 
\end{proof}
\begin{proof}[Proof of Proposition 3]
We can claim that as $m_0 \to \infty$
$$
\{\lambda_j (n-1-j)\}^{-1}\Big(\tr(\bS_D)-\sum_{l=1}^j\hat{\lambda}_l\Big)-\delta_j=O_P(n^{-1})\ \ \mbox{for $j=1,...,k$}
$$
in a way similar to the proof of Lemma 11 in \cite{Yata:2013b}. 
Then, it follows from (\ref{B.1}) that 
\begin{equation}
\tilde{\lambda}_{j}/\lambda_{j}=||\bu_{j}||^2+O_P(n^{-1})=1+O_P(n^{-1/2})
\quad \mbox{for $j=1,...,k$}. \label{B.3}
\end{equation}
Note that $({\bh}_{j}^T\tilde{\bh}_{j})^2=||\bu_{j}||^2(\hat{\bu}_{j}^T\dot{\bu}_{j})^2(\lambda_{j}/\tilde{\lambda}_{j})$.
Then, from (\ref{B.1}) and (\ref{B.3}) we can conclude the results.
\end{proof}
\begin{proof}[Proofs of Propositions 4 and 5]
First, we consider Proposition 4.
From (\ref{B.1}) there exists a unit random vector $\bep_{j}=(\varepsilon_{j1},...,\varepsilon_{jn})^T$ such that $\dot{\bu}_{j}^T\bep_{j}=0$ and 
\begin{align}
\hat{\bu}_{j}=\{1+O_P(n^{-1})\}\dot{\bu}_{j}+\bep_{j}\times O_P(n^{-1/2}) \ \ 
\mbox{for $j=1,...,k$} \label{B.4}
\end{align}
as $m_0 \to \infty$.
By noting that $\dot{\bu}_{j}={\bu}_{j}\{1+o_P(1)\}$ and ${\bu}_{j}^T{\bu}_{j'}=O_p(n^{-1/2})$ $(j\neq j')$ as $n\to \infty$, it follows from (\ref{B.4}) that
\begin{align}
\hat{\bu}_{j'}^T{\bu}_{j}= O_P(n^{-1/2})
\ \ \mbox{for $j'<j \le k$}.
\label{B.5}
\end{align}
Then, from (\ref{B.1}) to (\ref{B.3}) and (\ref{B.5}) it holds that for $j=1,...,k\ (l=1,...,n)$
\begin{align}
\frac{\tilde{\bh}_{j}^T\by_{l}}{{\lambda}_{j}^{1/2}}=
\frac{\hat{\bu}_{j}^T(\bX-\bM)^T\by_{l}}{\{(n-1)\tilde{\lambda}_{j}{\lambda}_{j} \}^{1/2}}
=\sum_{s=1}^{k}\frac{\lambda_{s} z_{sl}
\hat{\bu}_{j}^T{\bu}_{s}}{(\tilde{\lambda}_{j}\lambda_j)^{1/2}}=
z_{jl}+O_P(n^{-1/2}) \label{B.6}
\end{align}
because $z_{sl}=O_P(1)$ for $s=1,...,k.$ 
Let us write that 
$$
{\bu}_{j(l)}=(z_{j1},...,z_{jl-1},0,z_{jl+1},...,z_{jn})^T/(n-1)^{1/2}\ \ \mbox{for all $j,l$}. 
$$ 
We have that 
$E\{ (\sum_{s=k+1}^{p}\lambda_{s} z_{sl}{\bu}_{j(l)}^T{\bu}_{s(l)}/{\lambda}_{j})^2\}=O\{\Psi_{(k+1)}/(n{\lambda}_{j}^2)\}=O(n^{-1})$ 
and $E(||\sum_{s=k+1}^{p}$
$\lambda_{s} z_{sl}{\bu}_{s(l)}/{\lambda}_{j}||^2)=O(\Psi_{(k+1)}/{\lambda}_{j}^2)=O(1)$ for $j=1,...,k$. 
It follows that
\begin{align}
\sum_{s=k+1}^{p} \frac{\lambda_{s} z_{sl}{\bu}_{j(l)}^T{\bu}_{s(l)}}{\lambda_{j}}=O_P(n^{-1/2})
\ \ \mbox{and} \ \ \bzeta^T\sum_{s=k+1}^{p} \frac{\lambda_{s} z_{sl}{\bu}_{s(l)}}{{\lambda}_{j}}=O_P(1)
 \label{B.7}
\end{align}
from the fact that $|\bzeta^T\sum_{s=k+1}^{p}\lambda_{s} z_{sl}{\bu}_{s(l)}/{\lambda}_{j}|\le ||\bzeta||\cdot ||\sum_{s=k+1}^{p}\lambda_{s} z_{sl}{\bu}_{s(l)}/{\lambda}_{j} ||$ and Markov's inequality. 
Let $d_n=(n-1)/(n-2)$.
Here, from (\ref{B.4}) we write that for $j=1,...,k$
\begin{align}
d_n\hat{\bu}_{j(l)}=\{1+O_P(n^{-1})\}\bu_{j(l)}/||\bu_j||+\bep_{j(l)}\times O_P(n^{-1/2})+(n-2)^{-1}\hat{u}_{jl}\bone_{n(l)},
 \label{B.8}
\end{align}
where $\bep_{j(l)}=(\varepsilon_{j1},...,\varepsilon_{jl-1},0,\varepsilon_{jl+1},...,\varepsilon_{jn})^T$.
Note that $|| (n-2)^{-1}\hat{u}_{jl}\bone_{n(l)}||=O_P(n^{-1/2})$ since $|\hat{u}_{jl}|\le 1$. 
Then, it follows from (\ref{B.3}), (\ref{B.7}) and (\ref{B.8}) that for $j=1,...,k$
\begin{align}
\frac{\tilde{\bh}_{jl}^T\bv_{l}}{\lambda_j^{1/2}}=d_n
\frac{\hat{\bu}_{j(l)}^T(\bX-\bM)^T\bv_{l}}{\{(n-1)\tilde{\lambda}_{j}{\lambda}_{j} \}^{1/2}}
=d_n
\sum_{s=k+1}^{p}\frac{\lambda_{s} z_{sl}\hat{\bu}_{j(l)}^T{\bu}_{s(l)}}{(\tilde{\lambda}_{j}\lambda_j)^{1/2}}=O_P(n^{-1/2}).
\label{B.9}
\end{align}
We note that $\Var(\sum_{s=k+1}^{p}\lambda_{s} z_{sl}^2/\lambda_{j})=O(\Psi_{(k+1)}/\lambda_{j}^2)$, so that $(n-1)^{-1/2}$
$\sum_{s=k+1}^{p}\lambda_{s} z_{sl}^2/\lambda_{j}=(n-1)^{1/2}\delta_j+O_P(n^{-1/2})$ for $j=1,...,k$, 
because $E(\sum_{s=k+1}^{p}\lambda_{s} z_{sl}^2/\lambda_{j})=(n-1)\delta_j$. 
Then, it follows from (\ref{B.3}) and (\ref{B.7}) that for $j=1,...,k$
\begin{align}
\frac{(d_n\tilde{\bh}_{j}-\tilde{\bh}_{jl})^T\bv_{l}}{\lambda_j^{1/2}}&=d_n
\frac{(\hat{\bu}_{j}-\hat{\bu}_{j(l)})^T(\bX-\bM)^T\bv_{l}}{\{(n-1)\tilde{\lambda}_{j}{\lambda}_{j}\}^{1/2}} \notag \\
&=
d_n\hat{u}_{jl}\sum_{s=k+1}^{p}\frac{ \lambda_{s} z_{sl}^2
}{\{(n-1) \tilde{\lambda}_{j}\lambda_j)\}^{1/2}}
-\frac{\hat{u}_{jl}\bone_{n(l)}^T}{n-2} \sum_{s=k+1}^{p}\frac{\lambda_{s} z_{sl} {\bu}_{s(l)}}{(\tilde{\lambda}_{j}\lambda_j)^{1/2}}
\notag \\
&=d_n\hat{u}_{jl}(n-1)^{1/2}\delta_j\{1+o_P(1)\}+O_P(n^{-1/2}).
\label{B.10}
\end{align}
By combining (\ref{B.6}) and (\ref{B.9}) with (\ref{B.10}), we can conclude the result of $\tilde{\bh}_{j}$ in Proposition 4. 
As for $\hat{\bh}_{j}$, by noting that 
$\hat{\bh}_{j}=(\tilde{\lambda}_{j}/\hat{\lambda}_{j})^{1/2}\tilde{\bh}_{j}$, 
$||\bu_{j}||^2=1+O_P(n^{-1/2})$, (\ref{B.1}) and (\ref{B.3}), we can conclude the result. 

Next, we consider Proposition 5. 
From (\ref{B.3}) we have that for $j=1,...,k$
\begin{align}
\frac{(d_n\tilde{\bh}_{j}-\tilde{\bh}_{jl})^T\by_{l}}{\lambda_j^{1/2}}
&=d_n\hat{u}_{jl}\sum_{s=1}^{k}\frac{ \lambda_{s} z_{sl}^2
}{\{(n-1) \tilde{\lambda}_{j}\lambda_j)\}^{1/2}}
-\frac{\hat{u}_{jl}\bone_{n(l)}^T}{n-2} \sum_{s=1}^{k}\frac{\lambda_{s} z_{sl} {\bu}_{s(l)}}{(\tilde{\lambda}_{j}\lambda_j)^{1/2}} \notag \\
&
=
d_n\hat{u}_{jl}\times O_P\{(n^{1/2}\lambda_{j})^{-1} \lambda_{1}\}
 \label{B.11}
\end{align}
from the fact that $\bone_{n(l)}^T {\bu}_{s(l)}=O_P(1)$ and $z_{sl}=O_P(1)$, $s=1,...,k$. 
Then, by combining (\ref{B.6}) and (\ref{B.9}) with (\ref{B.11}), we can conclude the result. 
\end{proof}
\noindent
{\it {\bf Lemma B.1.} \ 
Assume (A-vi) and (A-viii). 
It holds for $j=1,...,k$ that as $m_0 \to \infty$}
$$
\sum_{l=1}^{n}\frac{\tilde{x}_{jl}-{x}_{jl}}{n}=O_P(\psi_{j}^{1/2}/n)\ \ 
\mbox{and}\ \ \sum_{l=1}^{n}\frac{(\tilde{x}_{jl}-{x}_{jl})^2}{n}=O_P(\psi_{j}/n).
$$
\begin{proof}
First, we consider the first result. 
Let $\bbbeta_{sj(l)}=\lambda_{s}z_{sl}{\bu}_{s(l)}/\lambda_{j}^{1/2}$, $\bxi_{sj(l)}=\lambda_{s}^{1/2} \mu_{(s)} {\bu}_{s(l)}/\lambda_{j}^{1/2}$ and $\bome_{sj(l)}=\bbbeta_{sj(l)}+\bxi_{sj(l)}$ for all $j,l,s$, where ${\bu}_{s(l)}$ is given in the proofs of Propositions 4 and 5. 
Then, we write that when $\hat{\lambda}_j>0$,
\begin{align}
\tilde{x}_{jl}=
d_n\frac{\hat{\bu}_{j(l)}^T(\bX-\bM)^T\bx_l}{\{(n-1)\tilde{\lambda}_{j} \}^{1/2}}
=d_n\frac{{\lambda}_{j}^{1/2} }{\tilde{\lambda}_{j}^{1/2}}\hat{\bu}_{j(l)}^T\sum_{s=1}^{p}\bome_{sj(l)},
\label{B.12}
\end{align}
where $d_n=(n-1)/(n-2)$.
Let $\be_{1}=(1,0,...,0)^T$,..., $\be_{n}=(0,...,0,1)^T$ be the standard basis vectors of dimension $n$. 
In view of (\ref{B.3}) and (\ref{B.8}), by noting that $||{\bu}_{j}||^2=1+O_P(n^{-1/2})$, $|| (n-2)^{-1}\hat{u}_{jl}\bone_{n(l)}||=O_P(n^{-1/2})$ as $n\to \infty$ and $\hat{\bu}_{j}-\hat{\bu}_{j(l)}=\hat{u}_{jl}\be_{l}-(n-1)^{-1}\hat{u}_{jl}\bone_{n(l)}$ for $l=1,...,n$, we have that as $m_0 \to \infty$
\begin{align}
d_n\hat{\bu}_{j(l)}({\lambda}_{j}/\tilde{\lambda}_{j})^{1/2}
=&{\bu}_{j(l)}/||{\bu}_{j}||^2+(n-2)^{-1}\hat{u}_{jl}\bone_{n(l)}+
(\bep_{j}-\varepsilon_{jl}\be_l )\times O_P(n^{-1/2}) \notag \\
&+\bzeta_{jl}\times O_P(n^{-1}) \quad
\mbox{for all $l$ and $j=1,...,k$}, \label{B.13}
\end{align}
where $\bep_{j}$ and $\varepsilon_{jl}$ are given in the proofs of Propositions 4 and 5 and $\bzeta_{jl}$ is a random unit vector depending on $j$ and $l$. 
Note that $O_P(n^{-1/2})$ and $O_P(n^{-1})$ in (\ref{B.13}) do not depend on $l$. 
In view of (A-viii), we have that for $j=1,...,k$
\begin{align}
E\Big\{\Big(\sum_{l=1}^n\bu_{j(l)}^T\sum_{s=1(\neq j)}^{p}\bbbeta_{sj(l)}\Big)^2 \Big\}=&\sum_{l\neq l'}^n\sum_{s,s'(\neq j)}^{p}\frac{\lambda_{s}\lambda_{s'}E(z_{jl}z_{sl}z_{s'l}z_{jl'}z_{sl'}z_{s'l'})}{(n-1)^2\lambda_{j}}\notag\\ 
&+O\{\tr(\bSig^2)/\lambda_{j}\}=O\{\tr(\bSig^2)/\lambda_{j}\}. \label{B.14}
\end{align}
On the other hand, we have that for $j=1,...,k$
\begin{align}
&E\Big\{\Big(\sum_{l=1}^n\bu_{j(l)}^T\sum_{s=1(\neq j)}^{p}\bxi_{sj(l)}\Big)^2 \Big\}=O\Big(n\sum_{s=1(\neq j)}^{p} \frac{ \lambda_{s} \mu_{(s)}^2 }{\lambda_{j}}\Big)
=O(n\bmu^T\bSig \bmu/\lambda_{j}).\label{B.15}
\end{align}
Then, by using Markov's inequality, it follows from (\ref{B.14}) and (\ref{B.15}) that
\begin{align}
\sum_{l=1}^n\bu_{j(l)}^T\sum_{s=1(\neq j)}^{p}{\bome_{sj(l)}}/{||{\bu}_{j}||^2}=O_P(\psi_j^{1/2}).
\label{B.16}
\end{align}
Also, we have that 
$E\{\sum_{l=1}^n (\bone_{n(l)}^T\sum_{s=1}^{p}\bome_{sj(l)})^2\}=O(n\psi_j)$, 
$E(\| \sum_{l=1}^n\sum_{s=1}^{p}$
$\bome_{sj(l)}\|^2)=O(n\psi_j)$ and 
$E(\sum_{l=1}^n \|\sum_{s=1}^{p}\bome_{sj(l)}\|^2)=O(n\psi_j)$ 
for $j=1,...,k$. 
Thus, it holds that
\begin{align}
&\Big|\sum_{l=1}^n \hat{u}_{jl}\bone_{n(l)}^T\sum_{s=1}^{p}\bome_{sj(l)}\Big| \le 
\Big(\sum_{l=1}^n\hat{u}_{jl}^2\Big)^{1/2}\Big\{ \sum_{l=1}^n\Big(\bone_{n(l)}^T\sum_{s=1}^{p}\bome_{sj(l)}\Big)^2 \Big\}^{1/2}
        =O_P(n^{1/2}\psi_j^{1/2}),\notag \\
&\Big|\bep_{j}^T\sum_{l=1}^n \sum_{s=1}^{p}\bome_{sj(l)}\Big| \le ||\bep_{j}|| \cdot 
\Big\|\sum_{l=1}^n\sum_{s=1}^{p}\bome_{sj(l)}\Big\|=O_P(n^{1/2}\psi_j^{1/2})\ \ \mbox{and} \notag \\
&\Big|\sum_{l=1}^n\bzeta_{jl}^T \sum_{s=1}^{p}\bome_{sj(l)}\Big|\le \Big(\sum_{l=1}^n ||\bzeta_{jl}||^2\Big)^{1/2}
\Big(\sum_{l=1}^n \Big\| \sum_{s=1}^{p}\bome_{sj(l)}\Big\|^2\Big)^{1/2}=O_P(n \psi_j^{1/2})
\label{B.17}
\end{align}
by using Markov's inequality and Schwarz's inequality.
Then, by noting that $\be_l^T\bome_{sj(l)}=0$ for all $l,s$, we have from (\ref{B.12}), (\ref{B.13}), (\ref{B.16}) and (\ref{B.17}) that for $j=1,...,k$
\begin{align}
\sum_{l=1}^n\frac{\tilde{x}_{jl}-x_{jl}}{n}&=\sum_{l=1}^n\frac{
x_{jl}}{n} \Big(
\frac{||{\bu}_{j(l)}||^2-||{\bu}_{j}||^2}{||{\bu}_{j}||^2}\Big)+
O_P(n^{-1}\psi_j^{1/2})
\notag \\
&=-\sum_{l=1}^n\frac{
 x_{jl}z_{jl}^2}{n(n-1)||{\bu}_{j}||^2 }+
O_P(n^{-1}\psi_j^{1/2})=O_P(n^{-1}\psi_j^{1/2})
\label{B.18}
\end{align}
because it holds that $|\sum_{l=1}^nx_{jl}z_{jl}^2|\le (\sum_{l=1}^nx_{jl}^2  \sum_{l'=1}^nz_{jl'}^4)^{1/2}$, 
$E(\sum_{l=1}^nx_{jl}^2)=n(\lambda_j+\mu_{(j)}^2)$, 
$E(\sum_{l=1}^nz_{jl}^4)=O(n)$, $\lambda_j\le \tr(\bSig^2)/\lambda_j$ and $\mu_{(j)}^2\le \bmu^T\bSig\bmu/\lambda_j$. 
Thus, we can conclude the first result.

Next, we consider the second result. 
From (\ref{B.12}) and (\ref{B.13}) we have that 
\begin{align}
\sum_{l=1}^n\frac{(\tilde{x}_{jl}-x_{jl})^2}{n}=&O_P\Big(\sum_{l=1}^n\frac{x_{jl}^2z_{jl}^4}{n^3}\Big)
+O_P\Big\{\sum_{l=1}^n \Big(\bu_{j(l)}^T\sum_{s=1(\neq j)}^{p}\bome_{sj(l)}\Big)^2/n  \Big\} \notag \\
&+O_P\Big\{\sum_{l=1}^n\hat{u}_{jl}^2 \Big(\bone_{n(l)}^T\sum_{s=1}^{p}\bome_{sj(l)}\Big)^2 /n^3 \Big\} \notag \\
&+O_P\Big\{\bep_{j}^T \sum_{l=1}^n\Big(\sum_{s=1}^{p}\bome_{sj(l)}\Big)\Big(\sum_{s=1}^{p}\bome_{sj(l)}\Big)^T\bep_{j}/n^2 \Big\} \notag \\
&+O_P\Big\{\sum_{l=1}^n \bzeta_{jl}^T\Big(\sum_{s=1}^{p}\bome_{sj(l)}\Big)\Big(\sum_{s=1}^{p}\bome_{sj(l)}\Big)^T\bzeta_{jl}/n^3 \Big\}. \label{B.19}
\end{align}
By using Markov's inequality, for any $\tau >0$, it holds that $P(\sum_{l=1}^nx_{jl}^2 \ge \tau n \psi_j)=O(\tau^{-1})$ and $\sum_{l=1}^n P(z_{jl}^4 \ge \tau n )=O(\tau^{-1})$ for $j=1,...,k$, so that 
\begin{equation}
\sum_{l=1}^nx_{jl}^2z_{jl}^4=O_P\big(n \psi_j \max_{l=1,...,n}z_{jl}^4\big)=O_P(n^2 \psi_j).
\label{B.20}
\end{equation}
We have that for $j=1,...,k$
\begin{align}
\sum_{l=1}^n\Big(\bu_{j(l)}^T\sum_{s=1(\neq j)}^{p}\bome_{sj(l)}\Big)^2&=O_P(\psi_j), \notag \\
\sum_{l=1}^n\hat{u}_{jl}^2 \Big(\bone_{n(l)}^T\sum_{s=1}^{p}\bome_{sj(l)}\Big)^2&\le \sum_{l=1}^n\Big(\bone_{n(l)}^T\sum_{s=1}^{p}\bome_{sj(l)}\Big)^2=O_P(n\psi_j), \notag
\\
\bep_{j}^T \sum_{l=1}^n\Big(\sum_{s=1}^{p}\bome_{sj(l)}\Big)\Big(\sum_{s=1}^{p}\bome_{sj(l)}\Big)^T\bep_{j}
&\le \sum_{l=1}^n \Big\| \sum_{s=1}^{p}\bome_{sj(l)} \Big\|^2=O_P(n\psi_j)\ \ \mbox{and} \notag \\
\sum_{l=1}^n \bzeta_{jl}^T\Big(\sum_{s=1}^{p}\bome_{sj(l)}\Big)\Big(\sum_{s=1}^{p}\bome_{sj(l)}\Big)^T\bzeta_{jl}
&\le \sum_{l=1}^n \Big\| \sum_{s=1}^{p}\bome_{sj(l)} \Big\|^2=O_P(n \psi_j) \label{B.21}
\end{align}
because it holds that $E\{\sum_{l=1}^n(\bu_{j(l)}^T\sum_{s=1(\neq j)}^{p}\bome_{sj(l)})^2\}=O(\psi_j)$, $E\{\sum_{l=1}^n(\bone_{n(l)}^T\sum_{s=1}^{p}\bome_{sj(l)})^2\}=O(n\psi_j)$,
$E(\sum_{l=1}^n \| \sum_{s=1}^{p}\bome_{sj(l)}\|^2)=O(n\psi_j)$ and $\hat{u}_{jl}^2\le 1$ for all $l$. 
Then, by combining (\ref{B.20}) and (\ref{B.21}) with (\ref{B.19}), we can conclude the second result. 
\end{proof}
\noindent
{\it {\bf Lemma B.2.} \ 
Assume (A-vi) and (A-viii). 
It holds that as $m \to \infty$ 
\begin{align*}
&\tilde{\bh}_{1j}^T\bh_{2j'}=h_{jj'}+O_P(n_1^{-1/2})\ \ \mbox{and} \ \ 
{\bh}_{1j}^T\tilde{\bh}_{2j'}=h_{jj'}+O_P(n_2^{-1/2}); \\
&\tilde{\bh}_{1j}^T\tilde{\bh}_{2j'}-h_{jj'}=\tilde{\bh}_{1j}^T\bh_{2j'}-h_{jj'}+
{\bh}_{1j}^T\tilde{\bh}_{2j'}-h_{jj'}+O_P\{(n_1 n_2)^{-1/2}\}
\end{align*}
for $j=1,...,k_1$ and $j'=1,...,k_2$.
}
\begin{proof}
First, we consider the first result. 
We note that $(\bX_i-\overline{\bX}_i) \hat{\bu}_{ij}=(\bX_i-\bM_i)\hat{\bu}_{ij}$ when $\hat{\lambda}_{ij}>0$. 
From (\ref{B.1}) to (\ref{B.5}) we have that as $m\to \infty$ 
\begin{align}
\tilde{\bh}_{1j}^T\bh_{2j'}&=\frac{\hat{\bu}_{1j}^T(\bX_1-\bM_1)^T\bh_{2j'}}{\{(n_i-1)\tilde{\lambda}_{1j}\}^{1/2}}=
\frac{\hat{\bu}_{1j}^T\sum_{s=1}^p\lambda_{1s}^{1/2}h_{sj'} \bu_{1s}}{\tilde{\lambda}_{1j}^{1/2}}\notag \\
&=h_{jj'}+
\frac{\hat{\bu}_{1j}^T\sum_{s=k_1+1}^p\lambda_{1s}^{1/2}h_{sj'} \bu_{1s}}{{\lambda}_{1j}^{1/2}\{1+o_P(1)\} }+O_P(n_1^{-1/2})
\label{B.22}
\end{align}
for $j=1,...,k_1$ and $j'=1,...,k_2$. 
It holds that for $j=1,...,k_1$ and $j'=1,...,k_2$
\begin{align}
&E\Big\{\Big(\bu_{1j}^T\sum_{s=k_1+1}^p\lambda_{1s}^{1/2}h_{sj'} \bu_{1s}\Big)^2\Big\} 
=O\Big(\sum_{s=k_1+1}^p \lambda_{1s}h_{sj'}^2/n_1 \Big)=O(\lambda_{1k_1+1}/n_1); \notag \\
&E\Big(\Big\|\sum_{s=k_1+1}^p\lambda_{1s}^{1/2}h_{sj'} \bu_{1s}\Big\|^2\Big)=O\Big(\sum_{s=k_1+1}^p \lambda_{1s}h_{sj'}^2\Big)=O(\lambda_{1k_1+1}) \notag 
\end{align}
because $\sum_{s=k_1+1}^ph_{sj'}^2 \le 1$. 
Then, by using Markov's inequality, it follows from (\ref{B.4}) that for $j=1,...,k_1$ and $j'=1,...,k_2$
\begin{align}
\frac{\hat{\bu}_{1j}^T\sum_{s=k_1+1}^p\lambda_{1s}^{1/2}h_{sj'} \bu_{1s}}{\lambda_{1j}^{1/2}}
=O_P\{n_1^{-1/2}(\lambda_{1k_1+1}/\lambda_{1j})^{1/2}\}.
\label{B.23}
\end{align}
Thus, by combining (\ref{B.22}) with (\ref{B.23}), we can conclude the result for $\tilde{\bh}_{1j}^T{\bh}_{2j'}$. 
As for ${\bh}_{1j}^T\tilde{\bh}_{2j'}$, we obtain the result similarly. 

Next, we consider the second result. 
From (\ref{B.2}), (\ref{B.5}) and (\ref{B.23}) we have that for $j\neq l=1,...,k_1$ and $j' \neq  l'=1,...,k_2$
\begin{align}
\frac{\hat{\bu}_{1j}^T(\sum_{s=k_1+1}^p\lambda_{1s}^{1/2}\lambda_{2l'}^{1/2}h_{sl'} \bu_{1s}  \bu_{2l'}^T)\hat{\bu}_{2j'}
}{\lambda_{1j}^{1/2}\lambda_{2j'}^{1/2}}&=O_P\Big(\frac{\hat{\bu}_{1j}^T\sum_{s=k_1+1}^p\lambda_{1s}^{1/2} h_{sl'} \bu_{1s} 
}{n_2^{1/2} \lambda_{1j}^{1/2}} \Big) \notag \\
&=O_P[\{\lambda_{1k_1+1}/(n_1n_2 \lambda_{1j})\}^{1/2}] \ \ \mbox{and} \notag \\
\frac{\hat{\bu}_{1j}^T(\lambda_{1l}^{1/2} \bu_{1l}\sum_{s'=k_2+1}^p \lambda_{2s'}^{1/2}h_{ls'}  \bu_{2s'}^T)\hat{\bu}_{2j'}
}{\lambda_{1j}^{1/2}\lambda_{2j'}^{1/2}}&=O_P[\{\lambda_{2k_2+1}/(n_1n_2 \lambda_{2j'})\}^{1/2}].
\label{B.24}
\end{align}
From (\ref{B.1}), (\ref{B.3}) and (\ref{B.23}) we have that for $j=1,...,k_1$ and $j'=1,...,k_2$
\begin{align}
&\frac{\hat{\bu}_{1j}^T(\sum_{s=k_1+1}^p\lambda_{1s}^{1/2}\lambda_{2j'}^{1/2}h_{sj'} \bu_{1s}  \bu_{2j'}^T)\hat{\bu}_{2j'}
}{\tilde{\lambda}_{1j}^{1/2}\tilde{\lambda}_{2j'}^{1/2}}=\frac{\hat{\bu}_{1j}^T\sum_{s=k_1+1}^p\lambda_{1s}^{1/2}h_{sj'} \bu_{1s} 
}{\tilde{\lambda}_{1j}^{1/2}}\{1+O_P(n_2^{-1})\} \notag \\
&=
\frac{\hat{\bu}_{1j}^T\sum_{s=k_1+1}^p\lambda_{1s}^{1/2}h_{sj'} \bu_{1s} 
}{\tilde{\lambda}_{1j}^{1/2}}+O_P\{(n_1n_2)^{-1/2}\}
\quad\mbox{and} \notag \\
&\frac{\hat{\bu}_{1j}^T(\lambda_{1j}^{1/2} \bu_{1j}\sum_{s'=k_2+1}^p \lambda_{2s'}^{1/2}h_{js'}  \bu_{2s'}^T)\hat{\bu}_{2j'}
}{\tilde{\lambda}_{1j}^{1/2}\tilde{\lambda}_{2j'}^{1/2}}=\frac{\sum_{s'=k_2+1}^p\lambda_{2s'}^{1/2}h_{js'} \bu_{2s'}^T \hat{\bu}_{2j'}
}{\tilde{\lambda}_{2j}^{1/2}}+O_P\{(n_1n_2)^{-1/2}\}. \label{B.25}
\end{align}
It holds that for $j=1,...,k_1$ and $j'=1,...,k_2$
\begin{align}
&E\Big\{ \Big(\sum_{s=k_1+1}^p\sum_{s'=k_2+1}^p\lambda_{1s}^{1/2}\lambda_{2s'}^{1/2} h_{ss'} {\bu}_{1j}^T \bu_{1s}\bu_{2s'}^T{\bu}_{2j'}\Big)^2\Big\}=O\Big(\frac{\tr(\bSig_{1*}\bSig_{2*})}{n_1n_2} \Big);
 \notag \\
&E\Big( \Big\|\sum_{s=k_1+1}^p\sum_{s'=k_2+1}^p\lambda_{1s}^{1/2}\lambda_{2s'}^{1/2} h_{ss'} {\bu}_{1j}^T \bu_{1s}\bu_{2s'} \Big\|^2 \Big) =O\{\tr(\bSig_{1*}\bSig_{2*})/n_1\}; \notag \\ 
&E\Big( \Big\|\sum_{s=k_1+1}^p\sum_{s'=k_2+1}^p\lambda_{1s}^{1/2}\lambda_{2s'}^{1/2} h_{ss'} \bu_{1s}\bu_{2s'}^T{\bu}_{2j'}\Big\|^2 \Big) =O\{\tr(\bSig_{1*}\bSig_{2*})/n_2\}; \quad\mbox{and}\notag \\
&E\Big( \Big\|\sum_{s=k_1+1}^p\sum_{s'=k_2+1}^p\lambda_{1s}^{1/2}\lambda_{2s'}^{1/2} h_{ss'} \bu_{1s}\bu_{2s'}^T\Big\|_F^2 \Big) =O\{\tr(\bSig_{1*}\bSig_{2*})\}, \notag  
\end{align}
where  $||\cdot ||_F$ is the Frobenius norm. 
Then, by noting that $\tr(\bSig_{1*}\bSig_{2*})\le \{\tr(\bSig_{1*}^2)\tr(\bSig_{2*}^2)\}^{1/2}$ and 
$|\bzeta_1^T (\sum_{s=k_1+1}^p\sum_{s'=k_2+1}^p$
$\lambda_{1s}^{1/2}\lambda_{2s'}^{1/2} h_{ss'} \bu_{1s}\bu_{2s'}^T) \bzeta_2|\le || \sum_{s=k_1+1}^p\sum_{s'=k_2+1}^p\lambda_{1s}^{1/2}\lambda_{2s'}^{1/2} h_{ss'} \bu_{1s}\bu_{2s'}^T||_F$, 
it follows from (\ref{B.4}) that for $j=1,...,k_1$ and $j'=1,...,k_2$ 
\begin{align}
\frac{\hat{\bu}_{1j}^T(\sum_{s=k_1+1}^p\sum_{s'=k_2+1}^p\lambda_{1s}^{1/2}\lambda_{2s'}^{1/2} h_{ss'} \bu_{1s}\bu_{2s'}^T)\hat{\bu}_{2j'} }{{\lambda}_{1j}^{1/2}{\lambda}_{2j'}^{1/2}}
=O_P\Big\{ \Big(\frac{\tr(\bSig_{1*}\bSig_{2*})}{n_1n_2\lambda_{1j}\lambda_{2j'} } \Big)^{1/2}\Big\}
=O_P\{(n_1n_2)^{-1/2}\}.
\label{B.26}
\end{align}
Then, from (\ref{B.1}) to (\ref{B.5}), (\ref{B.24}), (\ref{B.25}) and (\ref{B.26}) we have that for $j=1,...,k_1$ and $j'=1,...,k_2$
\begin{align}
\tilde{\bh}_{1j}^T\tilde{\bh}_{2j'}=&
\frac{\hat{\bu}_{1j}^T(\sum_{s,s'}^p\lambda_{1s}^{1/2}\lambda_{2s'}^{1/2} h_{ss'} \bu_{1s}\bu_{2s'}^T)\hat{\bu}_{2j'} }{\tilde{\lambda}_{1j}^{1/2}\tilde{\lambda}_{2j'}^{1/2}} \notag \\ 
=& \frac{\hat{\bu}_{1j}^T(\sum_{s=1}^{k_1}\sum_{s'=1}^{k_2} \lambda_{1s}^{1/2}\lambda_{2s'}^{1/2} h_{ss'} \bu_{1s}\bu_{2s'}^T)\hat{\bu}_{2j'} }{\tilde{\lambda}_{1j}^{1/2}\tilde{\lambda}_{2j'}^{1/2}}+O_P\{(n_1n_2)^{-1/2}\} \notag \\ 
&+\frac{\hat{\bu}_{1j}^T\sum_{s=k_1+1}^p\lambda_{1s}^{1/2}h_{sj'} \bu_{1s}
}{\tilde{\lambda}_{1j}^{1/2}}
+\frac{\sum_{s'=k_2+1}^p\lambda_{2s'}^{1/2} h_{js'} \bu_{2s'}^T\hat{\bu}_{2j'}
}{\tilde{\lambda}_{2j'}^{1/2}} \notag \\
=&h_{jj'}\Big(\frac{\lambda_{1j}^{1/2}\lambda_{2j'}^{1/2}   \hat{\bu}_{1j}^T \bu_{1j} \bu_{2j'}^T\hat{\bu}_{2j'} }{\tilde{\lambda}_{1j}^{1/2}\tilde{\lambda}_{2j'}^{1/2}}
-\frac{\lambda_{1j}^{1/2}   \hat{\bu}_{1j}^T \bu_{1j} }{\tilde{\lambda}_{1j}^{1/2}}
-\frac{\lambda_{2j'}^{1/2}   \hat{\bu}_{2j'}^T \bu_{2j'} }{\tilde{\lambda}_{2j'}^{1/2}} 
\Big) \notag \\
&+\tilde{\bh}_{1j}^T{\bh}_{2j'}+{\bh}_{1j}^T\tilde{\bh}_{2j'}+O_P\{(n_1n_2)^{-1/2}\} \notag \\
=&h_{jj'}\Big(\frac{\lambda_{1j}^{1/2}   \hat{\bu}_{1j}^T \bu_{1j} }{\tilde{\lambda}_{1j}^{1/2}}-1\Big)\Big(
\frac{\lambda_{2j'}^{1/2}   \hat{\bu}_{2j'}^T \bu_{2j'} }{\tilde{\lambda}_{2j'}^{1/2}}-1\Big) +\tilde{\bh}_{1j}^T{\bh}_{2j'}+{\bh}_{1j}^T\tilde{\bh}_{2j'}-h_{jj'}+O_P\{(n_1n_2)^{-1/2}\} \notag \\
=&\tilde{\bh}_{1j}^T{\bh}_{2j'}+{\bh}_{1j}^T\tilde{\bh}_{2j'}-h_{jj'}+O_P\{(n_1n_2)^{-1/2}\}
\notag
\end{align}
from the facts that 
$\tilde{\bh}_{1j}^T\bh_{2j'}=\hat{\bu}_{1j}^T\sum_{s=1}^p\lambda_{1s}^{1/2}h_{sj'} \bu_{1s}/\tilde{\lambda}_{1j}^{1/2}$ 
and 
${\bh}_{1j}^T \tilde{\bh}_{2j'}=\sum_{s'=1}^p\lambda_{2s'}^{1/2} h_{js'} \bu_{2s'}^T\hat{\bu}_{2j'}/\tilde{\lambda}_{2j'}^{1/2}$.
It concludes the second result. 
\end{proof}
\begin{proof}[Proof of Theorem 5]
We assume (A-ix) and (A-x).
Let $\bar{x}_{ij}=\sum_{l=1}^{n_i}x_{ijl}/n_i$ for $i=1,2;\ j=1,...,k_i$. 
For $i=1,2$ and $j=1,...,k_i$, we have that as $m\to \infty$
\begin{equation}
\bar{x}_{ij}-\mu_{i(j)}=O_P(\lambda_{ij}^{1/2}/n_i^{1/2}) \ \ 
\mbox{and}\ \ \sum_{l=1}^n\frac{(x_{ijl}-\mu_{i(j)})^2}{n_i}=O_P(\lambda_{ij})
\label{B.27}
\end{equation}
from the facts that $\Var(\sum_{l=1}^n x_{ijl}/n_i)=O(\lambda_{ij}/n_i)$ and $P(\sum_{l=1}^n(x_{ijl}-\mu_{i(j)})^2 \ge \tau n_i \lambda_{ij})=O(\tau^{-1})$ for any $\tau >0$. 
Let $\bar{x}_{ij\star}=\sum_{l=1}^{n_i}(\tilde{x}_{ijl}-x_{ijl})/n_i$ for $i=1,2;\ j=1,...,k_i$. 
Note that for $i=1,2;\ j=1,...,k_i$
$$
\frac{\psi_{ij}}{n_i^2K_{1*}^{1/2}}=O\Big(\frac{\lambda_{i1}^2}{n_i\tr(\bSig_{i*}^2)}+\frac{\bmu_{i*}^T\bSig_{i*}\bmu_{i*}+\sum_{s=1}^{k_i}\lambda_{is}\mu_{i(s)}^2}{\tr(\bSig_{i*}^2)}\Big)+o(1)\to 0
$$
from the facts that $\tr(\bSig_i^2)\le k_i \lambda_{i1}^2+\tr(\bSig_{i*}^2)$, 
$\tr(\bSig_{i*}^2)^{1/2}/(n_iK_{1*}^{1/2})=O(1)$, 
$\bmu_{i}^T\bSig_i\bmu_{i}=\bmu_{i*}^T\bSig_{i*}\bmu_{i*}+\sum_{s=1}^{k_i}\lambda_{is}\mu_{i(s)}^2$ 
and $\mu_{i(s)}^2=O(\lambda_{is}/n_i)$ for $s=1,...,k_i$.
Also, note that $\psi_{ij}^{1/2}\lambda_{ij}^{1/2}/(n_i^{3/2} K_{1*}^{1/2})\to 0$ for $i=1,2;\ j=1,...,k_i$. 
Then, with the help of Lemma B.1, we have that for $i=1,2;\ j=1,...,k_i$
\begin{align}
&2\sum_{l<l'}^{n_i}\frac{\tilde{x}_{ijl}\tilde{x}_{ijl'}-x_{ijl}x_{ijl'}}{n_i^2} =\sum_{l,l'}^{n_i}\frac{\tilde{x}_{ijl}\tilde{x}_{ijl'}-x_{ijl}x_{ijl'}}{n_i^2}
-\sum_{l=1}^{n_i}\frac{\tilde{x}_{ijl}^2-x_{ijl}^2}{n_i^2} \notag \\
&=\bar{x}_{ij\star} \sum_{l=1}^{n_i}\frac{\tilde{x}_{ijl}+x_{ijl}}{n_i}-\sum_{l=1}^{n_i}\frac{(\tilde{x}_{ijl}+x_{ijl})(\tilde{x}_{ijl}-x_{ijl})}{n_i^2}
\notag \\
&=\sum_{l=1}^{n_i} \frac{(\tilde{x}_{ijl}-x_{ijl})+2(x_{ijl}-\mu_{i(j)})+2\mu_{i(j)}}{n_i} \Big( \bar{x}_{ij\star}-\frac{\tilde{x}_{ijl}-x_{ijl}}{n_i}\Big)\notag \\
&=
O_P\{(\psi_{ij}^{1/2}/n_i)(\psi_{ij}^{1/2}/n_i+\lambda_{ij}^{1/2}/n_i^{1/2}+ \mu_{i(j)})\}=o_P(K_{1*}^{1/2})
\label{B.28}
\end{align}
from the fact that $\sum_{l=1}^{n_i}|(x_{ijl}-\mu_{i(j)})(\tilde{x}_{ijl}-x_{ijl})| \le \{\sum_{l=1}^{n_i}(x_{ijl}-\mu_{i(j)})^2\}^{1/2} \{\sum_{l=1}^{n_i}$
$(\tilde{x}_{ijl}-x_{ijl})^2\}^{1/2}$. 
From Lemma B.2 it holds that for $j=1,...,k_1$ and $j'=1,...,k_2$
\begin{align}
&\tilde{\bh}_{1j}^T\tilde{\bh}_{2j'}=h_{jj'}+O_P(n_{\min}^{-1/2}),\ \ \tilde{\bh}_{1j}^T(\tilde{\bh}_{2j'}-\bh_{2j'})=O_P(n_2^{-1/2}),\notag \\
&\tilde{\bh}_{2j'}^T(\tilde{\bh}_{1j}-\bh_{1j})=O_P(n_1^{-1/2})\ \ \mbox{and} \ \ (\tilde{\bh}_{1j}-{\bh}_{1j})^T(\tilde{\bh}_{2j'}-\bh_{2j'})=O_P\{(n_1n_2)^{-1/2}\}. \label{B.29}
\end{align}
Then, it follows from Lemma B.1, (\ref{B.27}) and (\ref{B.29}) that for $j=1,...,k_1$ and $j'=1,...,k_2$ 
\begin{align}
&\frac{\sum_{l=1}^{n_1}
(\tilde{x}_{1jl}\tilde{\bh}_{1j}-x_{1jl}\bh_{1j})^T\sum_{l'=1}^{n_2}
(\tilde{x}_{2j'l'}\tilde{\bh}_{2j'}-x_{2j'l'}\bh_{2j'})}{n_1n_2} \notag \\
&=\{\bar{x}_{1j\star}\tilde{\bh}_{1j}+\bar{x}_{1j}(\tilde{\bh}_{1j}-\bh_{1j})\}^T
\{\bar{x}_{2j'\star}\tilde{\bh}_{2j}+\bar{x}_{2j'}(\tilde{\bh}_{2j'}-\bh_{2j'})\}\notag \\
&=O_P\{(\psi_{1j}^{1/2}\psi_{2j'}^{1/2}/(n_1n_2)\}+O_P\{(\psi_{1j}^{1/2}/n_1)(\lambda_{2j'}^{1/2}/n_{2}^{1/2}+ \mu_{2(j')})/n_{2}^{1/2}\}
\notag \\ 
&\quad +O_P\{(\psi_{2j'}^{1/2}/n_2)(\lambda_{1j}^{1/2}/n_{1}^{1/2}+ \mu_{1(j)})/n_{1}^{1/2}\} \notag \\
&\quad +O_P\{(\lambda_{1j}^{1/2}/n_{1}^{1/2}+\mu_{1(j)})(\lambda_{2j'}^{1/2}/n_{2}^{1/2}+\mu_{2(j')})/(n_{1}n_2)^{1/2}\}\notag \\
&=o_P(K_{1*}^{1/2})
\label{B.30}
\end{align}
from the fact that $\lambda_{ij}/(n_{i}^2 K_{1*}^{1/2})=O\{\lambda_{ij}/(n_i\tr(\bSig_{i*}^2)^{1/2})\}=o(1)$ for $i=1,2;\ j=1,...,k_i$.
Note that $\bmu_{i*}=\bA_{i(k_i)}\bmu_i=\sum_{s=k_{i}+1}^p\mu_{i(s)}\bh_{is}$ for $i=1,2$. 
We write that when $\hat{\lambda}_{1j}>0$,
\begin{align}
\tilde{\bh}_{1j}^T\Big(\sum_{l=1}^{n_2}\frac{\bv_{2l}}{n_2}+\bmu_{2*} \Big)&=\frac{\hat{\bu}_{1j}^T(\bX_1-\bM_1)(\sum_{l=1}^{n_2}\bv_{2l}/n_2+\bmu_{2*}) }{(n_1-1)^{1/2}\tilde{\lambda}_{1j}^{1/2} }\notag \\
&=\hat{\bu}_{1j}^T\sum_{s=1}^{p}\sum_{s'=k_2+1}^{p}
\frac{ \lambda_{1s}^{1/2} h_{ss'} \bu_{1s}(\lambda_{2s'}^{1/2}\bar{z}_{2s'}+\mu_{2(s')}) }{\tilde{\lambda}_{1j}^{1/2}}. \label{B.31}
\end{align}
It holds that 
\begin{align}
&E\Big\{\Big( \bu_{1j}^T \sum_{s=k_1+1}^{p}\sum_{s'=k_2+1}^{p} \lambda_{1s}^{1/2} h_{ss'} \bu_{1s}(\lambda_{2s'}^{1/2}\bar{z}_{2s'}+\mu_{2(s')})\Big)^2
\Big\} \notag \\
&=O\Big\{(\sum_{s= k_1+1}^p\lambda_{1s}\bh_{1s}^T(\bSig_{2*}/n_2+\bmu_{2*}\bmu_{2*}^T) \bh_{1s} \Big)/n_1\Big\}\notag \\
&=O\{\tr(\bSig_{1*}\bSig_{2*})/(n_1n_2)+ \bmu_{2*}^T \bSig_{1*}\bmu_{2*}/n_1\} \ \ \mbox{for $j=1,...,k_1$}; \notag \\
&E\Big(\Big\|\sum_{s=k_1+1}^{p}\sum_{s'=k_2+1}^{p} \lambda_{1s}^{1/2} h_{ss'} \bu_{1s} (\lambda_{2s'}^{1/2}\bar{z}_{2s'}+\mu_{2(s')} )\Big\|^2 \notag
\Big)\\
&=O\{\tr(\bSig_{1*}\bSig_{2*})/n_2+\bmu_{2*}^T \bSig_{1*}\bmu_{2*}\}; \quad\mbox{and}\notag \\
&E\Big\{ \Big(\sum_{s'=k_2+1}^{p} h_{js'} (\lambda_{2s'}^{1/2}\bar{z}_{2s'}+\mu_{2(s')})\Big)^2\Big\}\notag \\
&=O(\bh_{1j}^T(\bSig_{2*}/n_2+\bmu_{2*}\bmu_{2*}^T) \bh_{1j} )=O\{\lambda_{2k_2+1}/n_2+(\bh_{1j}^T\bmu_{2*})^2\}\ \ \mbox{for $j=1,...,k_1$}.
\notag
\end{align}
In view of (\ref{B.1}) to (\ref{B.5}) and (\ref{B.31}), by using Markov's inequality, we have that for $j=1,...,k_1$
\begin{align}
\tilde{\bh}_{1j}^T\Big(\sum_{l=1}^{n_2}\frac{\bv_{2l}}{n_2}+\bmu_{2*}\Big)=&\sum_{s'=k_2+1}^{p}
 h_{js'} (\lambda_{2s'}^{1/2}\bar{z}_{2s'}+\mu_{2(s')} ) \notag \\
&+
O_P[\{\tr(\bSig_{1*}\bSig_{2*})/(\lambda_{1j} n_1n_2)\}^{1/2}+\lambda_{2k_2+1}^{1/2}/(n_1n_2)^{1/2}] \notag \\
&+
O_P\Big[\Big\{\bmu_{2*}^T \bSig_{1*}\bmu_{2*}/(\lambda_{1j} n_1)+\sum_{j'=1}^{k_1}(\bh_{1j'}^T\bmu_{2*})^2/n_1\Big\}^{1/2}\Big]. 
\label{B.32}
\end{align}
Note that ${\bh}_{1j}^T(\sum_{l=1}^{n_2}\bv_{2l}/n_2+\bmu_{2*})=\sum_{s'=k_2+1}^{p} h_{js'} (\lambda_{2s'}^{1/2}\bar{z}_{2s'}+\mu_{2(s')})$ and $\sum_{s'=k_2+1}^{p} h_{js'} (\lambda_{2s'}^{1/2}\bar{z}_{2s'}+\mu_{2(s')})=O_P(\lambda_{2k_2+1}^{1/2}/n_2^{1/2}+\bh_{1j}^T\bmu_{2*})$.
Also, note that $\lambda_{ik_i+1}=o\{\tr(\bSig_{i*}^2)^{1/2}\}$ for $i=1,2$. 
Then, it follows from Lemma B.1, (\ref{B.27}) and (\ref{B.32}) that for $j=1,...,k_1$
\begin{equation}
\{\bar{x}_{1j\star}\tilde{\bh}_{1j}+\bar{x}_{1j}(\tilde{\bh}_{1j}-\bh_{1j})\}^T\Big(\sum_{l=1}^{n_2}\frac{\bv_{2l}}{n_2}+\bmu_{2*}\Big)=o_P(K_{1*}^{1/2}). \label{B.33}
\end{equation}
Similarly, it follows that for $j=1,...,k_2$
\begin{equation}
\Big(\sum_{l=1}^{n_1}\frac{\bv_{1l}}{n_1}+\bmu_{1*} \Big)^T \{\bar{x}_{2j\star}\tilde{\bh}_{2j}+\bar{x}_{2j}(\tilde{\bh}_{2j}-\bh_{2j})\}=o_P(K_{1*}^{1/2}).
\label{B.34}
\end{equation}
In view of (\ref{B.30}), (\ref{B.33}) and (\ref{B.34}), we have that 
\begin{align}
&\frac{\sum_{l=1}^{n_1}
( \bx_{1l}-\sum_{j=1}^{k_1}\tilde{x}_{1jl}\tilde{\bh}_{1j})^T\sum_{l'=1}^{n_2}
( \bx_{2l'}-\sum_{j'=1}^{k_2}\tilde{x}_{2j'l'}\tilde{\bh}_{2j'})}{n_1n_2}\notag \\
&=
\frac{\sum_{l=1}^{n_1}
(\bv_{1l}+\bmu_{1*} )^T\sum_{l'=1}^{n_2}
(\bv_{2l'}+\bmu_{2*})}{n_1n_2}+o_P(K_{1*}^{1/2}). \label{B.35}
\end{align}
Then, by combining (\ref{B.28}) with (\ref{B.35}), we have that
$
\widehat{T}_*-T_{*}=o_P(K_{1*}^{1/2}),
$
so that from Corollary 3,
$(\widehat{T}_*-\Delta_*)/K_{*}^{1/2}\Rightarrow  N(0,1)$ under $ \limsup_{m\to \infty}\Delta_*^2/K_{1*}<\infty$. 
It concludes the results.  
\end{proof}
\begin{proof}[Proof of Lemma 3]
We assume (A-ix).
Let $\bS_{i(yy)}=(n_i-1)^{-1}\sum_{j=1}^{n_i}(\by_{ij}-\overline{\by}_i)(\by_{ij}-\overline{\by}_i)^T$, 
$\bS_{i(yv)}=(n_i-1)^{-1}\sum_{j=1}^{n_i}(\by_{ij}-\overline{\by}_i)(\bv_{ij}-\overline{\bv}_i)^T$, 
$\bS_{i(vy)}=\bS_{i(yv)}^T$
and $\bS_{i(vv)}=(n_i-1)^{-1}\sum_{j=1}^{n_i}(\bv_{ij}-\overline{\bv}_i)(\bv_{ij}-\overline{\bv}_i)^T$ for $i=1,2$, 
where $\overline{\by}_i=\sum_{j=1}^{n_i}\by_{ij}/n_i$ and $\overline{\bv}_i=\sum_{j=1}^{n_i}\bv_{ij}/n_i$. 
Note that $\bS_{in_i}=\bS_{i(yy)}+\bS_{i(yv)}+\bS_{i(vy)}+\bS_{i(vv)}$ for $i=1,2$. 
Also, note that $\bS_{i(yy)}=\sum_{j=1}^{k_i}\lambda_{ij}||\bu_{oij} ||^2\bh_{ij}\bh_{ij}^T+\sum_{j\neq j'}^{k_i}\lambda_{ij}^{1/2}\lambda_{ij'}^{1/2}\bu_{oij}^T\bu_{oij'}\bh_{ij}\bh_{ij'}^T$ for $i=1,2$. 
We write that for $i=1,2$
$$
\bS_{in_i}\sum_{j=1}^{k_i}\hat{\bh}_{ij}\hat{\bh}_{ij}^T=\sum_{j=1}^{k_i} \hat{\lambda}_{ij} \hat{\bh}_{ij}\hat{\bh}_{ij}^T
=\sum_{j=1}^{k_i} \tilde{\lambda}_{ij} \tilde{\bh}_{ij}\tilde{\bh}_{ij}^T\ \ (=\widehat{\bS}_{i(yy)},\ \mbox{say}). 
$$
Then, by noting that $||\bu_{oij} ||^2=||\bu_{ij} ||^2+O_P(n_i^{-1})$ and $\bu_{oij}^T\bu_{oij'}=O_P(n_i^{-1/2})$ $(j\neq j')$ as $n_i\to \infty$, it follows from (\ref{B.3}) that as $m\to \infty$
\begin{align}
\bS_{i(yy)}-\widehat{\bS}_{i(yy)}
=&\sum_{j=1}^{k_i}\tilde{\lambda}_{ij}(\bh_{ij}\bh_{ij}^T-\tilde{\bh}_{ij}\tilde{\bh}_{ij}^T)
+\sum_{j\neq j'}^{k_i}\lambda_{ij}^{1/2}\lambda_{ij'}^{1/2} \bu_{oij}^T\bu_{oij'}\bh_{ij}\bh_{ij'}^T+O_P(n_i^{-1})\sum_{j=1}^{k_i}\lambda_{ij}\bh_{ij}\bh_{ij}^T \notag \\ 
=& \sum_{j=1}^{k_i}\tilde{\lambda}_{ij}\{(\bh_{ij}-\tilde{\bh}_{ij})\bh_{ij}^T-\tilde{\bh}_{ij}(\tilde{\bh}_{ij}-{\bh}_{ij})^T\}\notag \\
&+O_P(n_i^{-1})\sum_{j=1}^{k_i}\lambda_{ij}\bh_{ij}\bh_{ij}^T +
O_P(n_i^{-1/2})\sum_{j\neq j'}^{k_i}\lambda_{ij}^{1/2}\lambda_{ij'}^{1/2}\bh_{ij}\bh_{ij'}^T  \label{B.36}
\end{align}
for $i=1,2$.
From Lemma B.2, (\ref{B.29}) and (\ref{B.36}) we have that 
$$
\tr\{(\bS_{1(yy)}-\widehat{\bS}_{1(yy)})(\bS_{2(yy)}-\widehat{\bS}_{2(yy)})\}=O_P\{\lambda_{11}\lambda_{21}(n_1n_2)^{-1/2}\}, 
$$
so that 
\begin{equation}
\tr\{(\bS_{1(yy)}-\widehat{\bS}_{1(yy)})(\bS_{2(yy)}-\widehat{\bS}_{2(yy)})\}/(n_1n_2)=o_P(K_{1*})
 \label{B.37}
\end{equation}
from the facts that $\lambda_{11}\lambda_{21}(n_1n_2)^{-3/2}\le \lambda_{11}^2/n_1^3+\lambda_{21}^2/n_1^3$ and 
$\lambda_{i1}=o(n_i^{1/2}\tr(\bSig_{i*}^2)^{1/2})$.
Note that $\bS_{i(yv)}=\sum_{j=1}^{k_i}\sum_{s=k_i+1}^{p} \lambda_{ij}^{1/2}\lambda_{is}^{1/2}\bu_{oij}^T\bu_{ois}\bh_{ij}\bh_{is}^T$ for $i=1,2$. 
Here, we write that when $\hat{\lambda}_{1j}>0$, 
\begin{align}
\tilde{\bh}_{1j}^T\bS_{2(yv)}{\bh}_{1j}
&= \sum_{j'=1}^{k_2} \lambda_{2j'}^{1/2} \frac{\hat{\bu}_{1j}^T(\bX_1-\bM_1)^T(  \sum_{s'=k_2+1}^{p} \lambda_{2s'}^{1/2}\bu_{o2j'}^T\bu_{o2s'} \bh_{2j'}\bh_{2s'}^T)}{(n_1-1)^{1/2}\tilde{\lambda}_{1j}^{1/2} } {\bh}_{1j} \notag \\
&=\sum_{j'=1}^{k_2} \lambda_{2j'}^{1/2} \hat{\bu}_{1j}^T\sum_{s=1}^{p}\sum_{s'=k_2+1}^{p}
\frac{ \lambda_{1s}^{1/2} h_{sj'}  h_{js'} \bu_{1s}\lambda_{2s'}^{1/2}\bu_{o2j'}^T\bu_{o2s'} }{\tilde{\lambda}_{1j}^{1/2}}. \label{B.38} 
\end{align}
It holds that for $j'=1,...,k_2$
\begin{align}
& E\Big\{ \Big( {\bu}_{1j}^T\sum_{s=k_1+1}^{p}\sum_{s'=k_2+1}^{p}
 \lambda_{1s}^{1/2}  h_{sj'}  h_{js'} \bu_{1s}\lambda_{2s'}^{1/2}\bu_{o2j'}^T\bu_{o2s'} \Big)^2\Big\} \notag\\
&=O\Big( \frac{ \bh_{1j}^T\bSig_{2*} \bh_{1j}\bh_{2j'}^T\bSig_{1*} \bh_{2j'}}{n_1n_2} \Big)=
O\Big( \frac{ \lambda_{1k_1+1}\lambda_{2k_2+1}  }{n_1n_2} \Big) \ \ \mbox{for $j=1,...,k_1$}, \notag \\
& E\Big( \Big\| \sum_{s=k_1+1}^{p}\sum_{s'=k_2+1}^{p}
\lambda_{1s}^{1/2} h_{sj'}  h_{js'} \bu_{1s}\lambda_{2s'}^{1/2}\bu_{o2j'}^T\bu_{o2s'} \Big\|^2 \Big)=
O\Big( \frac{ \lambda_{1k_1+1}\lambda_{2k_2+1}}{n_2} \Big) \notag \\
\mbox{and} \quad
&E\Big\{ \Big( \sum_{s'=k_2+1}^{p}h_{sj'}  h_{js'}  \lambda_{2s'}^{1/2}\bu_{o2j'}^T\bu_{o2s'} \Big)^2\Big\} =O(\lambda_{2k_2+1}/n_2)\ \ \mbox{for $j,s=1,...,k_1$}. \label{B.39}
\end{align}
In view of (\ref{B.1}) to (\ref{B.5}) and (\ref{B.38}), by using Markov's inequality, we have that for $j=1,...,k_1$ 
\begin{align}
\lambda_{1j} \tilde{\bh}_{1j}^T\bS_{2(yv)}{\bh}_{1j}
=\lambda_{1j} \sum_{j'=1}^{k_2}\lambda_{2j'}^{1/2} \sum_{s'=k_2+1}^{p} h_{jj'}h_{js'} \lambda_{2s'}^{1/2}\bu_{o2j'}^T\bu_{o2s'} +o_P(n_1n_2 K_{1*})
\label{B.40}
\end{align}
because $\lambda_{ik_i+1}=o\{\tr(\bSig_{i*}^2)^{1/2}\}$ for $i=1,2$. 
Similarly, it follows that
\begin{align}
\lambda_{1j} {\bh}_{1j}^T\bS_{2(yv)}\tilde{\bh}_{1j}
=\lambda_{1j} \sum_{j'=1}^{k_2}\lambda_{2j'}^{1/2} \sum_{s'=k_2+1}^{p} h_{jj'}h_{js'} \lambda_{2s'}^{1/2}\bu_{o2j'}^T\bu_{o2s'} +o_P(n_1n_2 K_{1*}).
\label{B.41}
\end{align}
We write that 
\begin{align}
\tilde{\bh}_{1j}^T\bS_{2(yv)}\tilde{\bh}_{1j} 
=\sum_{j'=1}^{k_2} \lambda_{2j'}^{1/2} \hat{\bu}_{1j}^T\sum_{s,t}^{p} \sum_{s'=k_2+1}^{p}
\frac{ \lambda_{1s}^{1/2}\lambda_{1t}^{1/2}  h_{sj'}  h_{ts'} \bu_{1s}\lambda_{2s'}^{1/2}\bu_{o2j'}^T\bu_{o2s'}\bu_{1t}^T }{\tilde{\lambda}_{1j}} \hat{\bu}_{1j}
 \label{B.42}.
\end{align}
It holds that for $j'=1,...,k_2$
\begin{align}
& E\Big\{ \Big\| 
\sum_{s,t\ge k_1+1}^{p} \sum_{s'=k_2+1}^{p}
 \lambda_{1s}^{1/2}\lambda_{1t}^{1/2}  h_{sj'}  h_{ts'}  \bu_{1s}\lambda_{2s'}^{1/2}\bu_{o2j'}^T\bu_{o2s'}\bu_{1t}^T \Big\|_F^2\Big\} \notag\\
&=O\{\tr(\bSig_{1*}\bSig_{2*})\bh_{2j'}^T\bSig_{1*}\bh_{2j'}/n_2 +\bh_{2j'}^T\bSig_{1*}\bSig_{2*}\bSig_{1*}\bh_{2j'}/(n_1n_2)\} \notag \\
&=O\{\tr(\bSig_{1*}\bSig_{2*})\lambda_{1k_1+1}/n_2+ \lambda_{1k_1+1}^2\lambda_{2k_2+1}/(n_1n_2)  \}.
 \label{B.43}
\end{align}
Then, in a way similar to (\ref{B.26}), by combing (\ref{B.39}) and (\ref{B.43}) with (\ref{B.42}), we have that for $j=1,...,k_1$ 
\begin{align}
\lambda_{1j} \tilde{\bh}_{1j}^T\bS_{2(yv)}\tilde{\bh}_{1j}
=\lambda_{1j} \sum_{j'=1}^{k_2}\lambda_{2j'}^{1/2} \sum_{s'=k_2+1}^{p} h_{jj'}h_{js'} \lambda_{2s'}^{1/2}\bu_{o2j'}^T\bu_{o2s'} +o_P(n_1n_2 K_{1*}).
\label{B.44}
\end{align}
Also, from (\ref{B.39}) we have that for $j,j'=1,...,k_1$ 
\begin{align}
\lambda_{1j}^{1/2}\lambda_{1j'}^{1/2} {\bh}_{1j}^T\bS_{2(yv)}{\bh}_{1j'}=
\lambda_{1j}^{1/2}\lambda_{1j'}^{1/2}
 \sum_{s=1}^{k_2}\lambda_{2s}^{1/2} \sum_{s'=k_2+1}^{p} h_{js}h_{j's'} \lambda_{2s'}^{1/2}\bu_{o2s}^T\bu_{o2s'}
=o_P(n_1^{3/2}n_2 K_{1*}).
\label{B.45}
\end{align}
Then, it follows from (\ref{B.36}), (\ref{B.40}), (\ref{B.41}), (\ref{B.44}) and (\ref{B.45}) that 
\begin{align}
{\tr\{ (\bS_{1(yy)}-\widehat{\bS}_{1(yy)})\bS_{2(yv)}\}}
={\tr\{ (\bS_{1(yy)}-\widehat{\bS}_{1(yy)})\bS_{2(vy)}\}}=o_P(n_1n_2K_{1*}).
\label{B.46}
\end{align}
Similarly, it follows that
$
{\tr\{ (\bS_{2(yy)}-\widehat{\bS}_{2(yy)})\bS_{1(yv)}\}}
={\tr\{ (\bS_{2(yy)}-\widehat{\bS}_{2(yy)})\bS_{1(vy)}\}}=o_P(n_1n_2K_{1*}).
$
Note that $
\bS_{i(vv)}=
\sum_{s, s'\ge k_i+1}^{p}\lambda_{is}^{1/2}\lambda_{is'}^{1/2}\bu_{ois}^T\bu_{ois'}\bh_{is}\bh_{is'}^T
$
for $i=1,2$.
Then, in a way similar to $\bS_{i(yv)}$, we can claim that for $i=1,2\ (j\neq i)$
\begin{align}
{\tr\{ (\bS_{i(yy)}-\widehat{\bS}_{i(yy)})\bS_{j(vv)}\}}=o_P(n_1n_2K_{1*}).
\label{B.47}
\end{align}
Then, by combining (\ref{B.46}) and (\ref{B.47}) with (\ref{B.37}), we have that 
\begin{align}
{\tr(\bS_{1n_1}\widehat{\bA}_{1(k_1)}\bS_{2n_2}\widehat{\bA}_{2(k_2)})}={\tr\{(\bS_{1n_1}-\bS_{1(yy)})(\bS_{2n_1}-\bS_{2(yy)})\}}+o_P(n_1n_2K_{1*}).
\label{B.48}
\end{align}
Let $\bSig_{i \star}=\sum_{j=1}^{k_i}\lambda_{ij}\bh_{ij}\bh_{ij}^T$ for $i=1,2$. 
We can evaluate that 
\begin{align*}
E[\{\tr(\bS_{1(yv)}\bS_{2(yv)})\}^2]=&O\Big(\frac{\tr(\bSig_{1\star}\bSig_{2 *})\tr(\bSig_{1 *}\bSig_{2 \star})}{n_1n_2}\Big) =O\Big(\frac{ \lambda_{11}\lambda_{21}\tr(\bSig_{1*}^2)^{1/2}\tr(\bSig_{2*}^2)^{1/2} }{n_1n_2}\Big); \\
E[\{\tr(\bS_{1(yv)}\bS_{2(vy)})\}^2]=&O\Big(\frac{\tr(\bSig_{1\star}\bSig_{2 \star})\tr(\bSig_{1 *}\bSig_{2 *})}{n_1n_2}\Big) =O\Big(\frac{ \lambda_{11}\lambda_{21}\tr(\bSig_{1*}^2)^{1/2}\tr(\bSig_{2*}^2)^{1/2} }{n_1n_2}\Big); \\
\mbox{and}\quad
E[\{\tr(\bS_{i(yv)}\bS_{j(vv)})\}^2]=&O\Big(\frac{\tr(\bSig_{i\star}\bSig_{j *}\bSig_{i*}\bSig_{j *})}{n_i}\Big)+O\Big(\frac{\tr(\bSig_{i\star}\bSig_{j *})\tr(\bSig_{1 *}\bSig_{2 *})}{n_1n_2} \Big) \\
=&o\Big(\frac{ \lambda_{i1} \tr(\bSig_{i*}^2)^{1/2} \tr(\bSig_{j*}^2)  }{n_i}\Big)
\end{align*}
for $i=1,2\ (j\neq i)$. 
Then, we have that 
\begin{align}
{\tr\{(\bS_{1n_1}-\bS_{1(yy)})(\bS_{2n_1}-\bS_{2(yy)})\}-\tr(\bS_{1(vv)}\bS_{2(vv)})}=o_P(n_1n_2K_{1*}).
\label{B.49}
\end{align}
With the help of (23) in \cite{Aoshima:2015}, we claim that $\tr(\bS_{1(vv)}\bS_{2(vv)})/\tr(\bSig_{1*}\bSig_{2*})=1+o_P(1)$. 
Hence, from (\ref{B.48}) and (\ref{B.49}) we have that 
$$
{\tr(\bS_{1n_1}\widehat{\bA}_{1(k_1)}\bS_{2n_2}\widehat{\bA}_{1(k_2)})}/(n_1n_2)=
{\tr(\bSig_{1*}\bSig_{2*})}/(n_1n_2)+o_P(K_{1*}). 
$$ 
By using Lemma S2.1, we can conclude the result. 
\end{proof}
\begin{proof}[Proof of Theorem 6]
Similar to the proof of Theorem 3, by combining Theorem 5 and Lemma 3, we can conclude the result.
\end{proof}
\section{Appendix C}
\setcounter{equation}{0}
\noindent

In this appendix, we give proofs of the theoretical results in Section S2.

\begin{proof}[Proofs of Propositions S2.1 and S2.2]
We omit the subscript with regard to the population for the sake of simplicity.
First, we consider Proposition S2.1. 
By using Lemmas 1 and 5 in \cite{Yata:2013b}, under (A-i) and (1.4), we can obtain
$\bzeta^T\{\sum_{s=1}^{p}\lambda_{s}\bu_{s}\bu_{s}^T-\tr(\bSig)\bI_n/(n-1)\}\bzeta=o_P\{\tr(\bSig^2)^{1/2}\}$ 
as $m_0 \to \infty$, 
so that $\hat{\lambda}_{1}=\tr(\bSig)/(n-1)+o_P\{\tr(\bSig^2)^{1/2}\}$ because 
$\hat{\lambda}_{1}-\tr(\bSig)/(n-1)=\hat{\bu}_{1}^T\{\sum_{s=1}^{p}\lambda_{s}\bu_{s}\bu_{s}^T-\bI_n \tr(\bSig)/(n-1)\}\hat{\bu}_{1}$. 
Then, 
by noting $\tr(\bS_{D})-\tr(\bSig)=o_P\{\tr(\bSig^2)^{1/2}\}$ under (A-i), 
we can claim $\tilde{\lambda}_{1}/\tr(\bSig^2)^{1/2}=o_P(1)$ under (A-i) and (1.4). 
Thus, from (2.3) we conclude the first result of Proposition S2.1.
Under (1,6) there exists a fixed integer $j_{\star}$ such that $\lambda_{j_{\star}}/\lambda_1\to 0$. 
Note that $(\sum_{i=j_{\star}}^p\lambda_i^4 )/\lambda_1^4\le \lambda_{j_{\star}}^2 \tr(\bSig^2)/\lambda_1^4=o(1)$. 
Then, by using Lemma 1 and Corollary 4.1 in \cite{Yata:2013b}, we can claim that $\tilde{\lambda}_{1}/{\lambda}_{1}=1+o_P(1)$ under (A-i) and (1.6). 
It concludes the results of Proposition S2.1. 

Next, we consider Proposition S2.2. 
Let $\phi(n)$ be any function such that $\phi(n)\to 0$ and $n^{1/4}\phi(n)\to \infty$ as $n\to \infty$.
Let $\lambda=\phi(n)\tr(\bSig^2)^{1/2}$. 
Assume that $\lambda_{1}^2/\tr(\bSig^2)=O(n^{-c})$ as $m_0\to \infty$ with some fixed constant $c>1/2$. 
Then, there is at least one positive integer $t \ (> 2)$ satisfying $c(t/2-1)>t/4$, so that 
$\tr(\bSig^{t})/\lambda^{t} \le \lambda_{1}^{t-2}/(\phi(n)^{t}\tr(\bSig^2)^{t/2-1})=o(1)$. 
Then, by using Lemmas 1 and 5 in \cite{Yata:2013b}, we can obtain 
$$
\hat{\lambda}_1/\lambda=\hat{\bu}_{1}^T \bS_{D}\hat{\bu}_{1}/\lambda=\tr(\bSig)/\{(n-1)\lambda\}+o_P(1)
$$
under (A-viii). 
Then, 
by noting that $\tr(\bS_{D})-\tr(\bSig)=o_P\{\tr(\bSig^2)^{1/2}\}$ under (A-viii), 
we can claim that $\tilde{\lambda}_1^2/\tr(\bSig^2)=o_P[\{\phi(n)\}^2]$ under (A-viii). 
Thus from (2.3) we conclude the result of Proposition S2.2. 
\end{proof}
\begin{proof}[Proofs of Lemma S2.1, Propositions S2.3 and S2.4]
We assume (A-i) and (A-vi).
We omit the subscript with regard to the population for the sake of simplicity. 
Let $\bV_{1}=\sum_{j=1}^k\lambda_j \bu_{j(1)}\bu_{j(2)}^T$ and $\bV_{2}=\sum_{j=k+1}^p\lambda_j \bu_{j(1)}\bu_{j(2)}^T$, 
where $\bu_{j(1)}=(z_{j1},....,z_{jn_{(1)}})^T/(n_{(1)}-1)^{1/2}$ and 
$\bu_{j(2)}=(z_{jn_{(1)}+1},....,z_{jn})^T/(n_{(2)}-1)^{1/2}$. 
Let $\bV_{o1}=\bP_{n_{(1)}}\bV_{1}\bP_{n_{(2)}}$ and $\bV_{o2}=\bP_{n_{(1)}}\bV_{2}\bP_{n_{(2)}}$.
Note that $\bS_{D(1)}=\bP_{n_{(1)}}(\bV_{1}+\bV_{2})\bP_{n_{(2)}}=\bV_{o1}+\bV_{o2}$.
Let us write the singular value decomposition of $\bS_{D(1)}$ as $\bS_{D(1)}=\sum_{j=1}^{n_{(2)}-1}\acute{\lambda}_{j}\acute{\bu}_{j(1)}\acute{\bu}_{j(2)}^T $, where $\acute{\bu}_{j(1)}$ (or $\acute{\bu}_{j(2)}$) denotes a unit left- (or right-) singular vector 
corresponding to $\acute{\lambda}_{j}$. 
First, we consider Lemma S2.1. 
By using Lemma 1 and Corollary 5.1 in \cite{Yata:2013b}, we can claim for $j=1,...,k$ that as $m_0\to \infty$
\begin{equation}
\acute{\lambda}_{j}/\lambda_{j}=1+o_P(1). \label{B.50}
\end{equation}
By noting that $\tr(\bSig_*^4)/\tr(\bSig_*^2)^2 \le \lambda_{k+1}^2/\Psi_{(k+1)}=o(1)$ under (A-vi) and by using Lemmas 1 and 4 in \cite{Yata:2013b}, we can claim that
$
\bzeta_{(1)}^T\bP_{n_{(1)}}\bV_{2}\bP_{n_{(2)}}\bzeta_{(2)}/\Psi_{(k+1)}^{1/2}=o_P(1),
$
where $\bzeta_{(i)}$ is an arbitrary unit random $n_{(i)}$-dimensional vector for $i=1,2$.
Hence, we have that
\begin{equation}
\bzeta_{(1)}^T \bS_{D(1)}\bzeta_{(2)}/\Psi_{(k+1)}^{1/2}=\bzeta_{(1)}^T \bV_{o1}\bzeta_{(2)}/\Psi_{(k+1)}^{1/2}+o_P(1).
 \label{B.51}
\end{equation}
Then, in a way similar to (A.10) in \cite{Yata:2013b}, we have that $\hat{\Psi}_{(k+1)}/\Psi_{(k+1)}=1+o_P(1)$. 
In view of (\ref{B.50}), we can claim that $\hat{\Psi}_{(j)}/\Psi_{(j)}=1+o_P(1)$ for $j=1,...,k$. 
It concludes the result of Lemma S2.1. 

Next, we consider Proposition S2.3. 
By noting (\ref{B.51}) and rank$(\bV_{o1})\le k$, we have that $\acute{\lambda}_{j}/\Psi_{(k+1)}^{1/2}=o_P(1)$ for $j>k$. 
Then, by combining Lemma S2.1 with (\ref{B.50}), we can conclude the result of Proposition S2.3.

Finally, we consider Proposition S2.4.
We assume (A-viii) and (A-ix). 
Let $\lambda_*=\phi(n)\Psi_{(k+1)}^{1/2}$, where $\phi(n)$ is defined in the proofs of Propositions S2.1 and S2.2. 
Assume that $\lambda_{k+1}^2/\Psi_{(k+1)}=O(n^{-c})$ as $m_0\to \infty$ with some fixed constant $c>1/2$. 
Then, there is at least one positive integer $t \ (> 2)$ satisfying $c(t/2-1)>t/4$, so that 
$\tr(\bSig_*^{t})/\lambda_*^{t} \le \lambda_{k+1}^{t-2}/(\phi(n)^{t}\Psi_{(k+1)}^{t/2-1})=o(1)$.
Hence, similar to (\ref{B.51}), we have that 
\begin{equation}
\bzeta_{(1)}^T \bS_{D(1)}\bzeta_{(2)}/\lambda_*=\bzeta_{(1)}^T \bV_{o1}\bzeta_{(2)}/\lambda_*+o_P(1). 
 \label{B.52}
\end{equation}
Let $\acute{\bV}_{o1}= \bV_{o1}-\sum_{j=1}^{k}\acute{\lambda}_j \acute{\bu}_{j(1)} \acute{\bu}_{j(2)}^T$. 
From (\ref{B.52}), it holds that 
$\bzeta_{(1)}^T \acute{\bV}_{o1} \bzeta_{(2)}/\lambda_*=o_P(1)$, so that 
all the singular values of $\acute{\bV}_{1}/\lambda_*$ are of the order $o_P(1)$. 
Then, from the fact that rank$(\acute{\bV}_{o1})\le 2k$, 
it holds that 
\begin{equation}
\tr( \acute{\bV}_{o1} \acute{\bV}_{o1}^T )/\Psi_{(k+1)}=k\times o_P[\{\phi(n)\}^2].
 \label{B.53}
\end{equation}
Here, in view of (A-viii), we have that $\Var(\bu_{j(1)}^T\bV_{o2} \bu_{j(2)} )=O(\Psi_{(k+1)}/n^2)$ for $j=1,...,k$, so that 
$\bu_{j(1)}^T\bV_{o2} \bu_{j(2)} =O_P(\Psi_{(k+1)}^{1/2}/n )$ for $j=1,...,k$. 
In view of (A-ix), it holds that 
\begin{equation}
\tr( \bV_{o1}{\bV}_{o2}^T )/\Psi_{(k+1)}=\tr( \bV_{1}{\bV}_{o2}^T )/\Psi_{(k+1)}=k\times o_P(n^{-1/2}).
 \label{B.54}
\end{equation}
On the other hand, we have that $E(||\bu_{j(1)}^T\bV_{o2}||^2)=O(\Psi_{(k+1)}/n)$ and $E(||\bu_{j(2)}^T\bV_{o2}^T||^2)=O(\Psi_{(k+1)}/n)$ for $j=1,...,k$, so that $\bu_{j(1)}^T\bV_{o2}\bzeta_{(2)} =O_P(\Psi_{(k+1)}^{1/2}/n^{1/2})$ and 
$\bzeta_{(1)}^T\bV_{o2}\bu_{j(2)}=O_P(\Psi_{(k+1)}^{1/2}/n^{1/2})$ for $j=1,...,k$. 
Then, in a way similar to the proof of Lemma 12 in \cite{Yata:2013b}, we have that 
$\acute{\bu}_{j(l)} =||\bu_{j(l)}||^{-1} \bu_{j(l)}\{1+O_P(n^{-1/2})\}+\bep_{jl}\times O_P(n^{-1/2})$ with 
some unit random vector $\bep_{jl}$ for $j=1,...,k;\ l=1,2$. 
Hence, from (\ref{B.50}) and $\bzeta_{(1)}^T\bV_{o2}\bzeta_{(2)}=o_P(\Psi_{(k+1)}^{1/2})$,
we have that $\tr(\sum_{j=1}^{k}\acute{\lambda}_j\acute{\bu}_{j(1)} \acute{\bu}_{j(2)}^T{\bV}_{o2}^T )/\Psi_{(k+1)}= k\times o_P(n^{-1/2})$. 
Hence, from (\ref{B.54}), it holds that 
\begin{equation}
\tr( \acute{\bV}_{o1} {\bV}_{o2}^T )/\Psi_{(k+1)}=k\times o_P(n^{-1/2}).
 \label{B.55}
\end{equation}
Note that $E\{\tr(\bV_{o2}\bV_{o2}^T)\}=\Psi_{(k+1)}$ and $\Var\{\tr(\bV_{o2}\bV_{o2}^T)/\Psi_{(k+1)}\}=O(n^{-1})$. 
Then, by noting that $\widehat{\Psi}_{(k+1)}=\tr\{(\acute{\bV}_{o1}+{\bV}_{o2})(\acute{\bV}_{o1}+{\bV}_{o2})^T\}$, 
from (\ref{B.53}) and (\ref{B.55}), 
we obtain that 
\begin{equation}
\widehat{\Psi}_{(k+1)}/\Psi_{(k+1)}=
\tr(\bV_{o2}\bV_{o2}^T)/\Psi_{(k+1)}+k\times o_P[\{\phi(n)\}^2]=1+k\times o_P[\{\phi(n)\}^2]. \label{B.56}
\end{equation}
Similarly, by noting that $\acute{\lambda}_{k+1}/\lambda_*=o_P(1)$ from (\ref{B.52}), 
we can claim that 
\begin{equation}
\widehat{\Psi}_{(k+2)}/\Psi_{(k+1)}=\{1+o(n^{-1/2})\}\widehat{\Psi}_{(k+2)}/\Psi_{(k+2)}=1+(k+1)\times o_P[\{\phi(n)\}^2]. \label{B.57}
\end{equation}
By combining (\ref{B.56}) and (\ref{B.57}), we can conclude the result of Proposition S2.4.
\end{proof}

\vskip 14pt
\noindent {\large\bf Acknowledgements}

Research of the first author was partially supported by Grants-in-Aid for Scientific Research (A) and 
Challenging Exploratory Research, Japan Society for the Promotion of Science (JSPS), under Contract Numbers 15H01678 and 26540010. 
Research of the second author was partially supported by Grant-in-Aid for Young Scientists (B), JSPS, under Contract Number 26800078.
\par


\bibhang=1.7pc
\bibsep=2pt
\fontsize{9}{14pt plus.8pt minus .6pt}\selectfont
\renewcommand\bibname
\vskip .65cm
\noindent
Institute of Mathematics, University of Tsukuba, Ibaraki 305-8571, Japan. 
\vskip 2pt
\noindent
E-mail: aoshima@math.tsukuba.ac.jp
\vskip 2pt
\noindent
Institute of Mathematics, University of Tsukuba, Ibaraki 305-8571, Japan. 
\vskip 2pt
\noindent
E-mail: yata@math.tsukuba.ac.jp
\end{document}